\documentclass[reqno]{amsart}
\usepackage[margin=1.5in,bottom=1.25in]{geometry}		% for tighter margins (80 CPL)
\usepackage{amsmath}		% for AMS macros
\usepackage{amssymb}		% for AMS symbols
\usepackage{amsfonts}		% for AMS fonts
\usepackage{amsthm}		% for theorems
\usepackage[foot]{amsaddr}		% for author footnotes

\usepackage{mathtools}		% for advanced math

\mathtoolsset{%
}

\usepackage{tikz} %for tikz graphics

\usepackage[utf8]{inputenc}		% for source encoding
\usepackage[T1]{fontenc}		% for font encoding

\usepackage[%		% for math font selection
cal=cm,
]
{mathalfa}

\usepackage{dsfont}		% for blackboard bold font
\usepackage[proportional,tabular,lining,sf,mono=false]{libertine}
\usepackage{acronym}		% for acronyms
\newcommand{\acdef}[1]{\emph{\acl{#1}} \textup{(\acs{#1})}\acused{#1}}		% for acro def
\usepackage[labelfont={bf,small},labelsep=colon,font=small]{caption}	% for caption control
\usepackage{subcaption}		% for subfigures
\usepackage{orcidlink}

\newcommand{\afterhead}{.\;}		% for changing headings
\newcommand{\para}[1]{\medskip\paragraph{\textbf{#1\afterhead}}}

\usepackage{array}		% for flexible tables
\usepackage{booktabs}		% for better tables
\usepackage[inline,shortlabels]{enumitem}		% for lists
\setenumerate{itemsep=\smallskipamount,topsep=2pt,left=\parindent}
\setitemize{itemsep=\smallskipamount,topsep=2pt,left=\parindent}
\usepackage{tabularx}

\usepackage[kerning=true]{microtype}		% for microtypography

\usepackage{comment}
\usepackage{tabto}		% for spacing
\usepackage{xspace}		% for flexible spaces

\usepackage[sort&compress]{natbib}		% for citations

\bibpunct[, ]{[}{]}{,}{n}{,}{,}

\usepackage{hyperref}
\hypersetup{
final,
colorlinks=true,
linktocpage=true,
pdfstartview=FitH,
breaklinks=true,
pdfpagemode=UseNone,
pageanchor=true,
pdfpagemode=UseOutlines,
plainpages=false,
bookmarksnumbered,
bookmarksopen=false,
bookmarksopenlevel=1,
hypertexnames=false,
pdfhighlight=/O,
pdftitle={},
pdfauthor={},
pdfsubject={},
pdfkeywords={},
pdfcreator={pdfLaTeX},
pdfproducer={LaTeX with hyperref}
}

\newcommand{\EMAIL}[1]{\email{\href{mailto:#1}{#1}}}

\usepackage[sort&compress,capitalize,nameinlink]{cleveref}		% for cleveref formatting
\crefname{algo}{Algorithm}{Algorithms}
\crefname{assumption}{Assumption}{Assumptions}
\crefname{figure}{Fig.}{Figs.}
\crefname{model}{Model}{Models}
\usepackage{algorithm}		% for algorithm environments
\usepackage{algpseudocode}		% for algorithm macros
\usepackage{thmtools}		% for theorem tools
\usepackage{thm-restate}		% for restating theorems

\usepackage{multicol}
\theoremstyle{plain}
\newtheorem{theorem}{Theorem}		% for theorems
\newtheorem{corollary}{Corollary}		% for corollaries
\newtheorem{lemma}{Lemma}		% for lemmas
\newtheorem{proposition}{Proposition}		% for propositions
\newtheorem*{corollary*}{Corollary}		% for corollaries (unnumbered)

\theoremstyle{definition}
\newtheorem*{example*}{Example}		% for examples (unnumbered)

\theoremstyle{remark}
\newtheorem{remark}{Remark}
\newtheorem*{remark*}{Remark}

\usepackage[showdeletions]{color-edits}		% for editing macros / use [suppress] for final
\newcommand{\draft}[1]{#1}		% for removing macro coloring

\newcommand{\newmacro}[2]{\newcommand{#1}{\draft{#2}}}		% for shorthand definitions
\newcommand{\newop}[2]{\DeclareMathOperator{#1}{\draft{#2}}}		% for shorthand definitions

\DeclarePairedDelimiter{\bracks}{[}{]}		% for brackets
\DeclarePairedDelimiter{\pospart}{[}{]_{+}}		% for positive part

\DeclarePairedDelimiterX{\setof}[1]{\{}{\}}{#1}		% for set builder notation
\DeclarePairedDelimiterX{\setdef}[2]{\{}{\}}{#1:#2}		% for set builder notation
\DeclarePairedDelimiterXPP{\exclude}[1]{\mathopen{}\setminus}{\{}{\}}{}{#1}

\DeclareMathOperator{\dist}{dist}		% for distance
\DeclareMathOperator{\sign}{sgn}		% for sign
\newop{\simplex}{\Delta}		% for simplices
\newcommand{\cf}{cf.\xspace}		% for consistency
\newcommand{\ie}{i.e.,\xspace}		% for consistency
\newcommand{\textpar}[1]{\textup(#1\textup)}		% for upshape parentheses

\newcommand{\alt}[1]{#1'}		% for alternate version
\newcommand{\altalt}[1]{#1''}		% for second alternate
\newmacro{\dd}{\:d}		% for integration
\newmacro{\const}{c}		% for generic constant
\newmacro{\Const}{C}		% for generic constant
\newmacro{\coef}{\lambda}		% for generic coefficient

\newmacro{\param}{\theta}		% for parameter
\newmacro{\params}{\Theta}		% for set of parameters

\newmacro{\tstart}{0}		% for time start
\renewcommand{\time}{\draft{t}}		% for continuous time
\newmacro{\timealt}{s}		% for dummy continuous time
\newmacro{\horizon}{T}		% for horizon

\newmacro{\flow}{\phi}		% for (semi)flow map
\DeclarePairedDelimiterXPP{\flowof}[2]{\flow_{#1}}{(}{)}{}{#2}		% for flow

\newop{\Nash}{NE}		% for Nash equilibria
\newop{\brep}{br}		% for best responses
\newop{\reg}{Reg}		% for regret
\newop{\val}{val}		% for value function

\newmacro{\players}{\mathcal{I}}		% for set of players

\newmacro{\pure}{i}		% for pure strategy
\newmacro{\purealt}{j}		% for alternate pure strategy
\newmacro{\purealtalt}{k}		% for second alternate
\newmacro{\nPures}{m}		% for number of pure strategies
\newmacro{\pures}{\mathcal{A}}		% for set of pure strategies

\newmacro{\strat}{\mathbf{x}}
\newmacro{\stratx}{x}		% for mixed strategy
\newmacro{\straty}{y}
\newmacro{\stratcl}{c_l}
\newmacro{\stratclalt}{\widetilde{c}_l}
\newmacro{\stratcr}{c_r}
\newmacro{\stratcralt}{\widetilde{c}_r}
\newmacro{\stratalt}{\alt\strat}		% for alternate strategy
\newmacro{\strataltalt}{\altalt\strat}		% for second alternate
\newmacro{\strats}{\mathcal{X}}		% for set of mixed strategies
\newmacro{\intstrats}{\strats^{\circle}}		% for set of interior strategies

\newmacro{\eq}{p}		% for Nash equilibrium

\newmacro{\pay}{u}		% for payoff function
\newmacro{\payv}{v}		% for payoff vector
\newmacro{\pot}{\Phi}		% for potential function

\newmacro{\game}{\mathcal{G}}		% for game
\newmacro{\gamefull}{\game(\pures,\payv)}		% for game with all elements

\newmacro{\fingame}{\Gamma}		% for finite game
\newmacro{\fingamefull}{\Gamma(\players,\pures,\pay)}		% for finite game with all elements
\newmacro{\mixgame}{\simplex(\fingame)}

\newmacro{\mat}{M}		% for generic matrix
\newmacro{\hmat}{H}		% for Hessian matrix

\newmacro{\ones}{\mathbf{1}}		% for matrix of ones
\newmacro{\eye}{I}		% for identity matrix
\newmacro{\zer}{\mathbf{0}}		% for zero matrix
\newmacro{\ttop}{{\!\top\!}}		% for transposes

\newmacro{\beforestart}{0}		% for before start index
\newmacro{\start}{1}		% for start index
\newmacro{\afterstart}{2}		% for second index
\newmacro{\running}{\start,\afterstart,\dotsc}		% for running

\newmacro{\run}{n}		% for main sequence index
\newmacro{\runalt}{k}		% for variant index
\newmacro{\nRuns}{T}		% for total number of runs

\newmacro{\state}{\strat}		% for main iterate
\newmacro{\statealt}{\score}		% for variant state

\renewcommand{\phi}{\varphi}
\newop{\RD}{RD}

\newmacro{\updmap}{F}
\newmacro{\bio}{\cramped{\updmap_{\mathclap{\step}}^{\mathrm{I}}}}
\newmacro{\econ}{\cramped{\updmap_{\mathclap{\step}}^{\mathrm{II}}}}
\newmacro{\learn}{\cramped{\updmap_{\mathclap{\step}}^{\mathrm{III}}}}
\newmacro{\econp}{\cramped{\updmap_{\mathclap{\; \step,p}}^{\mathrm{II}}}}
\newmacro{\econpp}{\cramped{\updmap_{\mathclap{\; \step,1-p}}^{\mathrm{II}}}}
\newmacro{\point}{x}
\newmacro{\bimap}{f}
\newmacro{\set}{\mathcal{S}}
\newmacro{\size}{z}
\newmacro{\score}{y}
\newmacro{\step}{\delta}
\newmacro{\switch}{\rho}
\newmacro{\gaina}{a}
\newmacro{\gainb}{b}
\newmacro{\gainc}{c}
\newmacro{\gaind}{d}
\newmacro{\csr}{\rho}
\newmacro{\iswitch}{r}
\newmacro{\payvec}{\pi}

\newcommand{\cm}{\color{black}}

\begin{document}

%**********************************************************************
%***    FRONT MATTER AND METADATA
%**********************************************************************

%----------------------------------------------------------------------
%%% TITLE & AUTHORS
%----------------------------------------------------------------------
\title
%[Discretely select, imitate or reinforce]
%{Discretely select, imitate or reinforce: different game dynamics arising from replicator dynamics}%amediscretizations of replicator dynamics for a simple  game}% --- convergence or instability}
[The emergence of chaos in population game dynamics]
{The emergence of chaos in population game dynamics\\ induced by comparisons}

%-------------------------------------------------------------------

\author
[J.~Bielawski]
{Jakub Bielawski$^1$}
\address{$^1$\,%
Department of Mathematics, Krakow University of Economics, Rakowicka 27, 31-510 Krak\'{o}w, Poland.}
\EMAIL{bielawsj@uek.krakow.pl}
%\orcidlink{0000-0001-6816-4501}

\author
[\L .~Cholewa]
{\L ukasz Cholewa$^2$}
\address{$^2$\,%
Department of Statistics, Krakow University of Economics, Rakowicka 27, 31-510 Krak\'{o}w, Poland.}
\EMAIL{cholewal@uek.krakow.pl}
%\orcidlink{0000-0003-3120-1238}

\author
[F.~Falniowski]
{Fryderyk Falniowski$^3$}
\address{$^3$\,%
Department of Mathematics, Krakow University of Economics, Rakowicka 27, 31-510 Krak\'{o}w, Poland.}
\EMAIL{falniowf@uek.krakow.pl}
%\orcidlink{0000-0002-1994-9019}

%-------------------------------------------------------------------

%-------------------------------------------------------------------

%----------------------------------------------------------------------
%%% KEYWORDS
%----------------------------------------------------------------------
%\subjclass[2010]{Primary XXXX, YYYY; Secondary ZZZZ.}
%\keywords{%
%separate;
%by;
%semicolon.}

%----------------------------------------------------------------------
%%% ACKNOWLEDGMENTS
%----------------------------------------------------------------------
%\thanks{\input{Thanks.tex}}

%----------------------------------------------------------------------
%%% ACRONYMS
%----------------------------------------------------------------------
\newacro{LHS}{left-hand side}
\newacro{RHS}{right-hand side}
\newacro{iid}[i.i.d.]{independent and identically distributed}
\newacro{lsc}[l.s.c.]{lower semi-continuous}
\newacro{whp}[w.h.p]{with high probability}
\newacro{wp1}[w.p.$1$]{with probability $1$}
\newacro{ODE}{ordinary differential equation}

\newacro{CCE}{coarse correlated equilibrium}
\newacroplural{CCE}[CCE]{coarse correlated equilibria}
\newacro{NE}{Nash equilibrium}
\newacroplural{NE}[NE]{Nash equilibria}
\newacro{ESS}{evolutionarily stable state}

\newacro{RD}{replicator dynamics}
\newacro{MWU}{multiplicative weights update}
\newacro{PPI}{pairwise proportional imitation}
\newacro{EW}{exponential\,/\,multiplicative weights}
\newacro{EXP3}{exponential-weights algorithm for exploration and exploitation}

\newacro{GAN}{generative adversarial network}


\maketitle

\begin{abstract}
Precise description of population game dynamics introduced by revision protocols - an economic model describing the agent's propensity to switch to a better-performing strategy - is of importance in economics and social sciences in general. In this setting imitation of others and innovation are forces which drives the evolution of the economic system. As the continuous-time game dynamics is relatively well understood, the same cannot be said about revision driven dynamics in the discrete time.
We investigate the behavior of agents in a $2\times 2$ anti-coordination game with symmetric random matching and a unique mixed Nash equilibrium. %In continuous time the Nash equilibrium is attracting and induces a global evolutionary stable state. 
We show that in the discrete time 
one can construct (either innovative or imitative) revision protocol and choose the length of a unit time interval between periods in a discrete time horizon ($\step$), under which the game dynamics is Li-Yorke chaotic, inducing complex and unpredictable behavior of the system, precluding stable predictions of equilibrium. This is in the stark contrast with the continuous case. Moreover, we reveal that this unpredictability is encoded in any imitative revision protocol. Furthermore, we show that for any such game there exists a perturbed pairwise proportional imitation protocol introducing chaotic behavior of the agents for sufficiently large $\step$. 
\end{abstract}

\section{Introduction}
\label{sec:introduction}
One of the most fundamental questions in the field of noncooperative game theory is to determine whether, and under what conditions, players eventually follow a close-to-equilibrium behavior through repeated interactions; that is, whether a dynamic process driven by the individual interests of agents leads to a rational outcome. Historically, this question has formed enormous interest in game dynamics that emerged with the birth of evolutionary game theory in the mid-1970s, followed by a surge in activity in economic theory in the 1990s, and more recently through various connections with machine learning and artificial intelligence in theoretical computer science \cite{cressman2003evolutionary,cressman2013stability,HS88,CBL06,San10,Wei95}.
In the economic context the dynamics of the game is usually derived from a set of microeconomic foundations that express the growth rate of a type (or strategy) in the population via a revision protocol (an economic model describing the agent's propensity to switch to a better-performing strategy), see e.g. \citet{FL98,Wei95,San10}.

A revision protocol, as it builds on the agent's propensity to switch to a better-performing strategy, has comparison encoded in its roots. A common feature of many revision protocols is the notion of imitation: the revising agent observes the behavior of a randomly selected individual and then switches to the strategy of the observed agent with a probability that may depend on the revising agent's current payoff, the payoff of the observed strategy, or both. Different models of imitation have been considered by \citet{bjornerstedt1996nash,BinSam97,schlag1998imitate,fudenberg2006imitation,fudenberg2008monotone,mertikopoulos2024nested,MV22} and many others, see e.g., \citet{San10,HLMS22}.
%In this regard, one of the most widely studied game dynamics are ,. 
Nevertheless, the {\it primus inter pares} is undoubtedly a mechanism known as \acdef{PPI}, originally due to \citet{Hel92}\footnote{See also \citet{BinSam97} for a derivation via a related mechanism known as ``imitation driven by dissatisfaction'', complementing the ``imitation of success'';
for a comprehensive account, \cf \citet{San10}.} as in continuous time it gives economic foundations for one of the most widely studied dynamics of evolutionary game theory --- the \emph{replicator dynamics} of \citet{TJ78}.%,  see also .
\footnote{Originally derived as a model for the evolution of biological populations under selective pressure, the \acl{RD} was subsequently rederived in 90's in economic theory via pairwise proportional imitation and, at around the same time,
as the mean dynamics of a stimulus-response model known as the \acdef{EW} algorithm, \cf \citet{Vov90,LW94,ACBFS95} and \citet{Rus99,Rus99b}.
Although these dynamics are merged in continuous time, they give qualitatively different outcomes in the discrete time \cite{falniowski2024discrete} leading, for instance to chaotic behavior for a wide range of games for pairwise proportional imitation model.}

Under imitative protocols, agents obtain candidate strategies by observing the behavior of randomly chosen opponents. In settings where agents are aware of the full set of available strategies, one can assume that they choose candidate strategies directly, without having to see choices of others. Those strategies don't even have to be used in the population. These, so called innovative/direct protocols, generate dynamics that differ in number of important ways from these generated by imitative dynamics.  For instance, imitative dynamics in continuous time setting always eliminate strictly dominated strategies \cite{nachbar1990evolutionary,akin1980domination} while innovative dynamics do not do so \cite{sandholm2010pairwise,hofbauer2011survival,MV22}.  Nevertheless, recently \citet{MV22} showed that these perspectives are not so far from each other. 
This raises a number of questions for discrete time models:
Do these dynamics differ in the discrete time? Do direct (innovative) comparisons are less prone to chaos?
Although game dynamics introduced by revision protocols in continuous time is relatively well understood \cite{San10,MV22,mertikopoulos2018riemannian}, the same cannot be said for discrete time models.

Putting the problem on a wider ground, by studying the mean dynamics one can hope to obtain plausible predictions for the long-run behavior when the length of a unit time interval between periods in a discrete time horizon (denoted by $\step$)  is sufficiently small to justify the descent to continuous time.
However, since real-life modeling considerations often involve larger values of $\step$ we are led to the following natural questions:

\begin{center}
    \it To what extent the discrete-time model of revision protocols leads to qualitatively different outcomes than the continuous one? How predictable the game dynamics driven by comparisons we can expect to be? And do direct (innovative) comparison driven dynamics differ in predictability with those where imitation drives the evolution of the system?
\end{center}

{\bf Our contributions.} 
Somewhat surprisingly, we show that even in the class of potential games \cite{monderer1996potential}, the class of games with most reliable predictions in the continuous setting, and even for cases where agents are symmetric with only two actions at players' disposal, their behavior can always become chaotic. In particular, we show that for any symmetric random matching in a $2\times 2$ anti-coordination game {\cm with continuum of agents,} there exist revision protocols under which game dynamics is chaotic for sufficiently large $\step$.\footnote{In particular, we propose revision protocols imposing chaotic behavior of game dynamics for $\step=1$.} This applies both to imitative and innovative protocols. Therefore, in this manner, comparisons through imitation or innovation are showing a similar story. Nevertheless, we show that chaotic behavior can arise in a very natural way when agents imitate others. In fact, we show that chaotic behavior is an inevitable component of imitative dynamics in $2\times 2$ anti-coordination games when gains from playing competing (different/opposite) strategies are close to each other.
%None of them should be seen as more predictable in the discrete time. ...
Finally, we show that for any $2\times 2$ anti-coordination game agents can use an imitative revision protocol, which can be seen as a (slight) modification of pairwise proportional model \eqref{eq:PPI}, implementing chaotic behavior and periodic orbits of any period. In this sense chaotic, unpredictable behavior can arise in very simple manner in game dynamics, which can be seen as a perturbation of the replicator dynamics in the discrete time setting. 
%One may ask how complicated such revision protocol has to be to induce chaotic game dynamics. We show that chaotic behavior is an inevitable component of imitative dynamics for anti-coordination games {\color{red}for equal gains} for sufficiently large step size. Then, for any anti-coordination game starting from pairwise proportional imitation model \cite{falniowski2024discrete} we give the construction of the imitative revision protocol introducing chaotic game dynamics.

In the above, when speaking about chaotic behavior, we mean the notion of \emph{Li-Yorke chaos} \textendash\ as introduced in the seminal paper of \citet{liyorke} \textendash\ for which there exists an uncountable set of initial conditions that is \emph{scrambled}, \ie every pair of points in this set eventually comes arbitrarily close and then drifts apart again infinitely often.
In dynamical systems that we consider here, Li-Yorke chaos implies other features of chaotic behavior like positive topological entropy or the existence of a set on which one can detect sensitive dependence on initial conditions  \cite{Ruette}.
In this sense, the system is truly unpredictable. % which comes in stark contrast to the convergent landscape from the continuous-time.

\para{Related work}

There is a significant number of recent results suggesting that complex, non-equilibrium behaviors of boundedly rational agents (employing learning rules) seems to be common rather than exceptional. For population games \citet{cressman2003evolutionary} (Example 2.3.3) was one of the first who suggested complex behavior even for simple games of discrete time evolutionary models.  %For population games \citet{cressman2003evolutionary}\footnote{Throughout the literature there are described some special cases (see e.g. \citet{cressman2003evolutionary}, Example 2.3.3) of complex behavior of discrete time evolutionary models.}
In this aspect, the seminal work of \citet{SatoFarmer_PNAS} showed analytically that even in a simple two-player zero-sum game of Rock-Paper-Scissors, the (symmetric) \acl{RD} exhibit Hamiltonian chaos.
\citet{sato2003coupled} subsequently extended this result to more general multiagent systems, opening the door to detecting chaos in many other games (in the continuous-time regime). 

More recently, \citet{becker2007dynamics} and \citet{geller2010microdynamics} revealed a chaotic behavior for Nash maps in games like matching pennies, while \citet{VANSTRIEN2008259} and \citet{VANSTRIEN2011262} showed that fictitious play also possesses rich periodic and chaotic behavior in a class of 3x3 games, including
Shapley's game and zero-sum dynamics.
%possess rich periodic and chaotic behaviors.
In a similar vein, \citet{Soda14} showed that  the replicator dynamics are Poincar\'{e} recurrent in zero-sum games, a result which was subsequently generalized to the so-called ``follow-the-regularized-leader'' (FTRL) dynamics \cite{mertikopoulos2017cycles}, even in more general classes of games \cite{LMB23-CDC};
see also \cite{MV22,HMC21,BM23,MHC24} for a range of results exhibiting convergence to limit cycles and other non-trivial attractors.

There is growing evidence that a class of algorithms from behavioral game theory known as experience-weighted attraction (EWA) also exhibits chaotic behavior for two-agent games with many strategies in a large parameter space \cite{GallaFarmer_PNAS2013}, or in games with many agents \cite{GallaFarmer_ScientificReport18}.
In particular, \citet{2017arXiv170109043P} showed experimentally that EWA leads to limit cycles and high-dimensional chaos in two-agent games with negatively correlated payoffs.
Chaotic behavior has also been detected in anti-coordination games under discrete-time dynamics similar in spirit to the  model discussed in this article. In more detail, \citet{vilone2011chaos} showed numerically the presence of Lyapunov chaos for some values of parameters in the Snowdrift game, while other works detected complex behavior in the Battle of the Sexes and Leader games \cite{pandit2018weight,mukhopadhyay2021replicator,mukhopadhyay2020periodic}.
All in all, careful examination suggests a complex behavioral landscape in many games (small or large) for which no single theoretical framework currently applies.

However, none of the above results implies chaos in the formal sense of Li-Yorke.
The first formal proof of Li-Yorke chaos was established for the % EW
{\cm multiplicative weights update (MWU)} algorithm in a symmetric two-player two-strategy congestion game by \citet{palaiopanos2017multiplicative}.
This result was generalized and strengthened (in the sense of positive topological entropy) for all two-agent
two-strategy congestion games~\cite{Thip18}.
In \cite{CFMP2019} topological chaos in nonatomic congestion game where agents use %EW
{\cm MWU} was established.
This result was then extended to FTRL with steep regularizers~\cite{BCFKMP21} and EWA algorithms \cite{bielawski2024memory}, but the resulting framework does not apply to the revision protocols case.
In arguably the main precursor of our work \citet{falniowski2024discrete} showed Li-Yorke chaos for \acl{PPI} for $2\times 2$ congestion games, which is in stark contrast with intra-species competition model, where we see convergence to Nash equilibrium. This suggests that the microeconomic model of revision protocols is closer to the machine learning counterpart than its (historical) predecessor from biology \cite{cressman2014replicator,chastain2014algorithms,argasinski2018evolutionary,argasinski2018interaction}.

In Section \ref{sec:prelim}, we formulate settings of our game and introduce the revision protocols. Then in Section \ref{sec:discrete} we review properties of (one-dimensional) dynamical systems and present and prove the main mathematical tool we use in this paper --- Proposition \ref{lem:chaos_cond}. In Section \ref{sec:thm1} we show that for any population anti-coordination game with two strategies there exist innovative (and imitative as well) revision protocols under which game dynamics will be chaotic (Theorem \ref{thm:anticoord-chaos}). Then, as revision protocols from Theorem \ref{thm:anticoord-chaos} can be seen as artificial ones, in Sections \ref{sec:thm2} and \ref{sec:thm3} for any $2\times 2$ anti-coordination  game we derive imitative revision protocols based on the PPI protocol, for which game dynamics is chaotic.
Proofs of all theorems (except Proposition \ref{lem:chaos_cond}) can be found in appendices.

\section{Settings}
\label{sec:prelim}
 In general, showing that a system exhibits chaotic behavior is a task of considerable difficulty and satisfactory theory exists only for low-dimensional systems. Thus, in our paper, we will consider games with a continuum of nonatomic players modeled by the unit interval $\players = [0, 1]$, with each player choosing (in a measurable way) an \emph{action} from the set of two available strategies  $ \pures \equiv \setof{A,B}$. % with 2 available strategies.
Denoting by $\stratx_A \in [0,1]$ the mass of agents playing strategy $A$ (thus $\stratx_B=1-\stratx_A$ being $B$-strategists), the overall distribution of actions at any point in time will be specified by the \emph{state of the population} $\strat = (\stratx_A,\stratx_B)$, being a point in the unit simplex $\strats$. %\defeq \simplex(\pures)$. %= \setdef*{\strat\in\R_{+}^{2}}{\stratx_{A}+\stratx_B =1}$ of $\R^{2}$. 
%We will also simplify notation by denoting the $A$-strategists ratio by $\stratx=\stratx_A$, thus $\stratx_B=1-\stratx$.
We will be interested in  \emph{symmetric random matching} \cite{HS98,Wei95,San10,HACM22},
 when two individuals are selected randomly from the population and are matched to play a symmetric two-player game with payoff matrix $\mat = (\mat_{\pure\purealt})_{\pure,\purealt\in\pures}$.
In this case, the payoff to agents playing $\pure\in\pures$ at state $\strat$ will be $\pay_{\pure}(\strat) = \sum_{\purealt\in\pures} \mat_{\pure\purealt} \stratx_{\purealt}$.
We will use
$\pay(\strat) = \sum_{\pure\in\pures} \stratx_{\pure} \pay_{\pure}(\strat)$ for the population's \emph{mean payoff} at state $\strat\in\strats$,
$\payv(\strat) = (\pay_{A}(\strat),\pay_{B}(\strat))$ for the associated \emph{payoff vector} at state $\strat$,
and
we will refer to the tuple $\game \equiv \gamefull$ as a \emph{population game}.
Finally, a state $\strat\in\strats$ is a Nash equilibrium of the game $\game$ if every strategy in use earns a maximal payoff (equivalently, each agent in population chooses an optimal strategy with respect to the choices of others).

{\bf Game.} In this setting, we will focus on a random matching scenario induced by $2\times 2$ symmetric anti-coordination games with %actions $\pures=\{A,B\}$ and 
payoff matrix
\begin{equation}
\label{eq:game}
\tag{game}
\begin{array}{l|cc}
	&A	&B\\
\hline
A	&\gaina		&\gainb\\
B	&\gainc	& \gaind
\end{array}
\end{equation}
where $\gaina,\gainb,\gainc,\gaind\in\mathbb{R}$ and $a<c$, $d<b$. 
Thus,
\begin{equation}\label{payoffs}
 \pay_A(\strat)=(\gaina-\gainb)\stratx_A+\gainb;\qquad \pay_B(\strat)=(\gainc-\gaind)\stratx_A+\gaind.
\end{equation}
The unique Nash equilibrium of this game is $\mathbf{p}=(p,1-p)\in \strats$, where the {\cm fraction} of $A$-strategists is given by 
 \begin{equation} \label{eq:eq} 
 \eq=\frac{\gainb-\gaind}{\gainc-\gaina+\gainb-\gaind}.
 \tag{eq}
 \end{equation}

{\bf Revision protocols for \eqref{eq:game}.} How do agents behave (choose actions) in this game? The model we consider has its roots in the mass-action interpretation of game theory and, more precisely, the theory of revision protocols \cite{San10, Wei95}, which gives microeconomic foundations of evolutionary game theory. Referring to the textbook of \cite{San10} for the details of the general model, and \cite{falniowski2024discrete} for its discrete counterpart, suppose that each agent occasionally receives an opportunity to switch actions \textendash\ say, based on the rings of a Poisson alarm clock \textendash\  and, at such moments, they reconsider their choice of action by comparing its payoff to that of a randomly chosen individual in the population.
A \emph{revision protocol} of this kind is defined by specifying the \emph{conditional switch rate} $\switch_{\pure\purealt}(\strat)$ at which a revising $\pure$-strategist switches to strategy $\purealt$ when the population is at state $\strat\in\strats$.
Usually revision protocols belong to one of two most prominent (and natural) types of revision protocols: imitative and innovative (also known as direct) revision protocols.

\noindent {\bf Imitative protocols.} 
The archetypical example of imitative revision protocol is the \acdef{PPI} protocol of \citet{Hel92}, as described by the switch rate functions
\begin{equation}
\label{eq:PPI}
\tag{PPI}
\switch_{\pure\purealt}(\strat)
	= \stratx_{\purealt} \pospart{\pay_{\purealt}(\strat) - \pay_{\pure}(\strat)},
\end{equation}
where $[d]_+ = \max\{d,0\}$.
Under this protocol, {\cm each agent} %a revising agent 
first observes the action of a randomly selected opponent, so a $\purealt$-strategist is observed with probability $\stratx_{\purealt}$ when the population is at state $\strat\in\strats$.
Then, if the payoff of the incumbent strategy $\pure\in\pures$ is lower than that of the benchmark strategy $\purealt$, the agent imitates the selected agent with probability proportional to the payoff difference $\pay_{\purealt}(\strat) - \pay_{\pure}(\strat)$;
otherwise, the revising agent skips the revision opportunity and sticks to their current action.
In general, an \emph{imitative protocol} is any protocol, where agents behave according to the revision protocol
\begin{equation}
\label{eq:imitation}
\tag{imit}
\switch_{\pure\purealt}(\strat)=\stratx_{\purealt} \iswitch_{\pure\purealt} (\strat) 
\end{equation}
described by the {\it conditional imitation rate} $\iswitch_{\pure\purealt}\colon \strats \to \mathbb{R}$, which is assumed to be Lipschitz continuous and the net conditional imitation rates are monotone, that is
\begin{equation}\label{eq:net_cond_imit_rates}
 \pay_{\purealt}\geq \pay_{\pure} \;\Leftrightarrow \; \iswitch_{k\purealt}-\iswitch_{\purealt k}\geq \iswitch_{k\pure}-\iswitch_{\pure k}\;\;\forall \pure,\purealt,k.
\end{equation}
{\it Imitation} comes from the fact, that an agent can switch only to a strategy, which is already used by part of the population (thus, he imitates). It is worth mentioning that in any imitative dynamics per-capita growth rates of pure strategies  are ordered as their payoffs (see \cite{nachbar1990evolutionary,Vio15,MV22}).

\noindent {\bf Innovative protocols.} In innovative (direct) protocols %are protocols, in which 
{\cm all} agents compare their payoffs with payoffs resulting from other strategies (used or not), changing it only when the payoff from the benchmark strategy is more profitable, that is when
 \begin{equation}
\label{eq:paircomp}
\tag{innov}
sgn (\switch_{\pure\purealt}(\strat))=sgn ([\pay_{\purealt}(\strat)-\pay_{\pure}(\strat)]_+). 
\end{equation}
{\cm Under these protocols agents are assumed to choose candidate strategies directly, and a strategy's popularity does not directly influence the probability with which it is considered by a revising agent.}\footnote{\cm Although there is no general definition of innovative dynamics it usually refers to dynamics in which a better strategy may increase its share even if no agent is currently using it. Nevertheless, the condition \eqref{eq:paircomp} is quite restrictive as, for instance, the best response dynamics \cite{gilboa1991social} may not fulfill it if agents have more than two available actions.}
A well-known example of innovative {\cm (non-imitative)} revision protocols is the pairwise comparison protocol \cite{San10, Smi84}, where
\begin{equation} \label{eq:Smithprot}\switch_{\pure\purealt}(\strat)= \pospart{\pay_{\purealt}(\strat) - \pay_{\pure}(\strat)}. \tag{PC} \end{equation}
for the others see \cite{San10,San10b,MV22}.

%{\color{red}  best response dynamics....} 

{\bf Dynamics.} By specifying conditional switch rates $\switch_{\pure\purealt}(\strat)$ we are able to describe how agents choose their actions and thus derive the formula for game dynamics driven by comparisons. In this case, the population share of agents switching from strategy $\pure$ to strategy $\purealt$ over a small interval of time $\step$ ({\cm length of} a unit time interval between periods in a discrete time horizon) will be be $\stratx_{\pure} \switch_{\pure\purealt} \step$, leading to the {\it inflow-outflow} equation
\begin{equation}
\label{eq:in-out-gen}
%\tag{in-out}
\stratx_{\pure}(\time+\step)
	= \stratx_{\pure}(\time)
	+ \step \bracks*{
		\sum_{\purealt\neq\pure} \stratx_{\purealt}(\time) \switch_{\purealt\pure}(\strat(\time))
		- \stratx_{\pure}(\time) \sum_{\purealt\neq\pure} \switch_{\pure\purealt}(\strat(\time))
		}.
\end{equation}
In particular, the game dynamics can be described by the formula  for {\cm the population share of} $A$-strategists, which is given by
\begin{equation}
\label{eq:in-out}
\tag{in-out}
\stratx(\time+\step)= \stratx(\time) +\step\left[(1-\stratx(\time))\switch_{BA}(\stratx(\time))-\stratx(\time) \switch_{AB}(\stratx (\time))\right]
\end{equation}
where $\stratx=\stratx_A$ is the $A$-strategists {\cm fraction}.

%Evidently, the Smith protocol \eqref{eq:Smithprot} belongs to this class.
Therefore, for the 2-strategy population game \eqref{eq:game} the dynamics given by \eqref{eq:in-out} is boiled down to the dynamics of the unit interval map
\begin{equation}
\tag{map}
\label{eq:updmap} \updmap(\stratx)=\stratx +\step\left[(1-\stratx)\switch_{BA}(\stratx)-\stratx \switch_{AB}(\stratx)\right],
\end{equation}
where $\stratx$ is the population share of $A$-strategists and $1-\stratx$ is the population share of $B$-strategists.\footnote{Although payoff functions, conditional switch rates, conditional imitation rates and the map $\updmap$ itself depend on $\strat=(\stratx,1-\stratx)$, to simplify the notation, in the rest of the article we will treat all functions as dependent only on the population share $\stratx $.} 

This system admits only unique interior fixed point $\eq$ given by \eqref{eq:eq} both in game dynamics introduced by \eqref{eq:paircomp} and by \eqref{eq:imitation}.
In addition, any imitative game dynamics has two other fixed points: $0$ and $1$.% and $\eq$ for imitative game dynamics \eqref{eq:imitation}.

\section{Review on discrete-time dynamical systems}
\label{sec:discrete}
Before discussing game dynamics in detail, we give a short overview of concepts from dynamical systems and present a crucial fact which will be applied frequently in proofs of main theorems of the paper. 
Let $(X,\dist)$ be a nonempty compact metric space, and 
let $f\colon X\to X$ be a continuous map, that is, let $(X,f)$ be a dynamical system. 
We will be interested in the long-term behavior, with special emphasize on the limiting behavior of the system. To this aim we introduce few ideas. 
A fixed point $x$ of a dynamical system $(X,f)$ is called
\begin{itemize}
    \item attracting, if there is an open neighborhood $U\subset X$ of $x$ such that $f(U)\subset U$ and for every $y\in U$, we have $\lim_{n\to\infty}f^n(y)=x$, where $f^n$ is a composition of the map $f$ with itself $n$-times;
    \item repelling, if there is an open neighborhood $U\subset X$ of $x$ such that for every $y\in U$, $y\neq x$, there exists $n\in\mathbb{N}$ such that $f^n(y)\notin U$.
\end{itemize}

As we will study maps of the unit interval, the fixed point $\stratx \in \mathcal{I} =[0,1]$ is attracting if $|f'(\stratx)| < 1$ (or when the derivative doesn't exist, both one-sided derivatives fulfill this condition), and  $\stratx$ is repelling when $|f'(\stratx)| > 1$ (or when the derivative doesn't exist, both one-sided derivatives fulfill this condition). If  $|f'(\stratx)| =1$, or if the condition is met only for one of one-sided derivatives, we need more information.

An orbit $(f^n(\stratx))$ is called \emph{periodic} of period $T$ if $f^{n+T}(\stratx)=f^n(\stratx)$ for any $n\in \mathbb{N}$. The smallest such $T$ is called the period of $\stratx$.
The periodic orbit is called attracting, if $\stratx$ is an attracting fixed point of $(X,f^T)$, and
 repelling, if $\stratx$ is a repelling fixed point of $(X,f^T)$.

On the antipodes of convergence to the fixed point lays chaotic behavior. The most widely used definition of chaos is due to \citet{liyorke}:
%\begin{definition}[Li-Yorke chaos]
%\label{def:chaos}
%Consider a dynamical system of the form \eqref{eq:dyn} for some continuous map $f\from\strats\to\strats$.
A pair of points $\stratx,\stratx'\in X$ is said to be \emph{scrambled} \textendash\ or a \emph{Li\textendash Yorke pair} \textendash\ if
\begin{equation}
\label{eq:chaos}
\liminf_{\run\to\infty} \dist(f^{\run}(\stratx),f^{\run}(\stratx'))
	= 0
	\quad
	\text{and}
	\quad
\limsup_{\run\to\infty} \dist(f^{\run}(\stratx),f^{\run}(\stratx'))
	> 0.
\end{equation}
%Extending this definition to sets, a subset $\set\subseteq\strats$ is called itself \emph{scrambled} if every pair of distinct points $\strat,\stratalt\in\set$ is scrambled.
We then say that $(X,f)$ is \emph{chaotic} \textpar{in the Li\textendash Yorke sense} if it admits an uncountable \emph{scrambled set}, \ie a set $\set\subseteq X$ such that every pair of distinct points $\stratx,\stratx'\in\set$ is scrambled.

%{\color{red} co jeszcze tutaj dodać? czy dopiero tutaj wprowadzać dynamikę wynikającą z \eqref{eq:in-out}/\eqref{eq:in-out2}? Jesli wczesniej, to gdzie powinna byc definicja ukladu dynamicznego itp.?}

\vskip 0.2in

In this note we will be interested in the long-term dynamics of one-dimensional dynamical systems. For such systems Li-Yorke chaos implies other (stronger) notions of chaos like positive topological entropy or existence of the subset of the space with sensitive dependence on initial conditions and Devaney chaos \cite{Ruette}. Moreover, due to Li-Yorke theorem \cite{liyorke,li1982odd}, it can be shown by finding a periodic orbit of an odd period.

Now, we present a simple observation that will be used many times in the following sections.

\begin{proposition}\label{lem:chaos_cond}
Let $f\colon[0,1]\to[0,1]$ be a continuous map. If there exist points $0\leq z_l<z_r\leq 1$ such that $f^2(z_l)>z_r$ and either
    \begin{multicols}{2}
    \begin{enumerate}[$(1)$]
		\item\label{lem:chaos_cond1} $f(z_r)<z_l$,
            \item\label{lem:chaos_cond2} $f(z_l)>z_r$,
       %     \item\label{lem:chaos_cond3} $f^2(z_l)>z_r$,
	\end{enumerate}
    \end{multicols}
    hold    or
    \begin{multicols}{2}
    \begin{enumerate}[$(1')$]
		\item\label{lem:chaos_condA} $f(z_r)>z_l$,
            \item\label{lem:chaos_condB} $f(z_l)<z_l$,
    %        \item\label{lem:chaos_condC} $f^2(z_l)>z_r$,
	\end{enumerate}
    \end{multicols}
    hold, then $f(x)<x<f^3(x)$ for some point $x\in(z_l,z_r)$.
\end{proposition}
    \begin{proof}
First we will show that there exists  $x\in(z_l,z_r)$ such that $f(x)=z_l$.
Under conditions \ref{lem:chaos_cond1} and \ref{lem:chaos_cond2} we get the inequalities
    \[
    f(z_l)>z_r>z_l>f(z_r).
    \]
    Hence, by the continuity of $f$, we obtain
    \[
    f\left((z_l,z_r)\right)\supset\left((f(z_r),f(z_l)\right)\supset[z_l,z_r],
    \]
    so there is a point $x\in(z_l,z_r)$ such that $f(x)=z_l$. 
Similarly, if we assume conditions  \ref{lem:chaos_condA} and \ref{lem:chaos_condB}, then $z_l\in (f(z_l,z_r))$ and by continuity of $f$ there exists $x\in (z_l,z_r)$ such that $f(x)=z_l$.

   Finally, the condition~$f^2(z_l)>z_r$ ensures that
    \[
    f^3(x)=f^2(z_l)>z_r>x,
    \]
    which completes the proof.
     \end{proof}

In this article we will usually deal with bimodal maps with critical points $\stratcl, \stratcr \in (0,1)$. Thus, we will use Proposition \ref{lem:chaos_cond} putting $z_l=\stratcl$ and $z_r=\stratcr$. 
In Figure \ref{Cobweb} we show how the set of assumptions of Proposition \ref{lem:chaos_cond} impacts the dynamics then. %If none of these conditions is satisfied then the fixed point attracts trajectories (case {\bf (a)} on Figure \ref{Cobweb}). 
%If one of these conditions holds then we may still observe convergence of the trajectories to the fixed point (case {\bf (c)} on Figure \ref{Cobweb}).
%However, if one of these conditions holds then the fixed point may lose stability but usually we observe periodic behavior.
In particular, when conditions of Proposition \ref{lem:chaos_cond} are satisfied we face unpredictability and chaos (cases {\bf (b)} and {\bf (d)} on Figure \ref{Cobweb}).

The result of Proposition \ref{lem:chaos_cond} will be crucial in the proof of chaotic behavior of the agents in our model: by Li-Misiurewicz-Panigiani-Yorke theorem \cite{li1982odd}, Proposition \ref{lem:chaos_cond} implies the existence of a periodic orbit of period 3. By the Sharkovsky Theorem (\cite{sha}), existence of a periodic orbit of period 3 guarantees existence of periodic orbits of all periods, and by \cite{liyorke} it implies Li-Yorke chaos.

\begin{remark}
As already mentioned, to meet assumptions of Proposition \ref{lem:chaos_cond} it will be (usually) convenient to assume that we deal with bimodal maps (and fulfill conditions for critical points). Although the dynamics of unimodal maps is quite well understood \cite{Ruette,lyubich2002almost}, dynamics of bimodal is less understood \cite{milnor2000entropy}, and gives place for behaviors impossible for unimodal maps \cite{BCFKMP21}. 
\end{remark}
\begin{remark}
The conditions from Proposition \ref{lem:chaos_cond} are going to be met for game dynamics introduced by revision protocols which we will construct. In particular, conditions \ref{lem:chaos_condA} and \ref{lem:chaos_condB} together with $f^2(z_l)>z_r$ will be fulfilled for the game dynamics introduced by innovative revision protocols \eqref{eq:paircomp}, while \ref{lem:chaos_cond1} and \ref{lem:chaos_cond2} together with $f^2(z_l)>z_r$ %ensure assumptions of
will hold for game dynamics driven by imitative ones \eqref{eq:imitation} for sufficiently large $\step$.
\end{remark}

%{\color{red} introduced by continuous bimodal maps (with two critical points ).}

\vskip 0.2in
%{\color{red} tu koniecznie rysunek!}

\begin{figure}
\centering
\begin{subfigure}{0.48\textwidth}
 \includegraphics[width=0.49\textwidth]{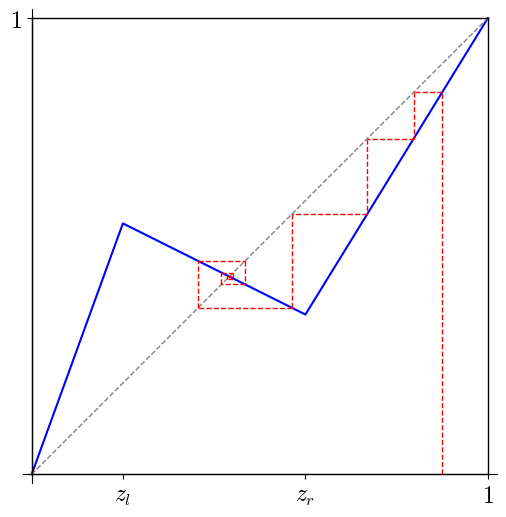}
\hfill
 \includegraphics[width=0.49\textwidth]{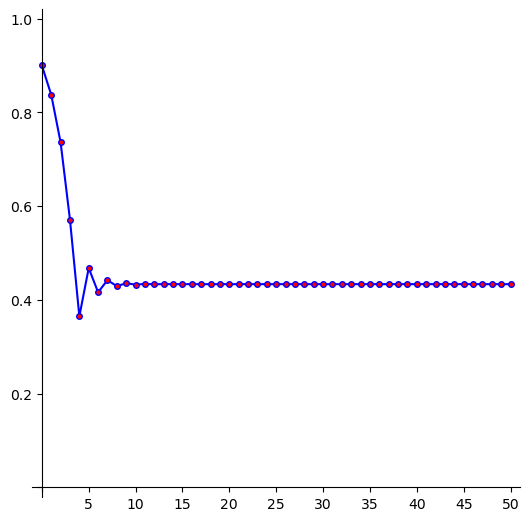}
 \caption{$f(z_l)=0.55$, $f(z_r)=0.35$}
\end{subfigure}
\hfill
\begin{subfigure}{0.48\textwidth}
 \includegraphics[width=0.49\textwidth]{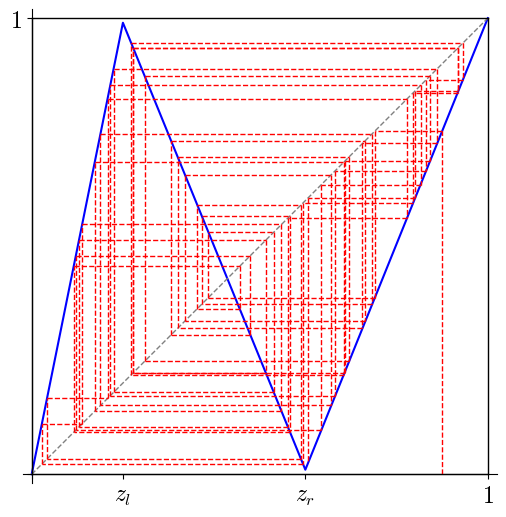}
\hfill
 \includegraphics[width=0.49\textwidth]{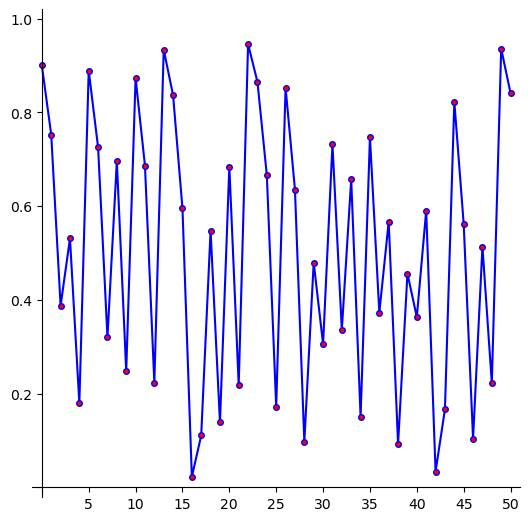}
 \caption{$f(z_l)=0.99$, $f(z_r)=0.01$}
\end{subfigure}
\\
\begin{subfigure}{0.48\textwidth}
 \includegraphics[width=0.49\textwidth]{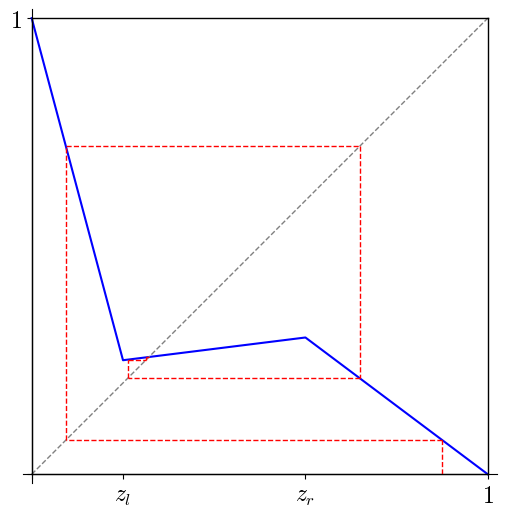}
\hfill
 \includegraphics[width=0.49\textwidth]{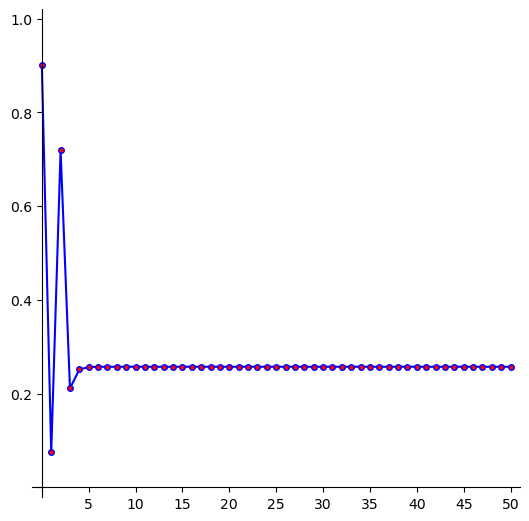}
 \caption{$f(z_l)=0.25$, $f(z_r)=0.3$}
\end{subfigure}
\hfill
\begin{subfigure}{0.48\textwidth}
 \includegraphics[width=0.49\textwidth]{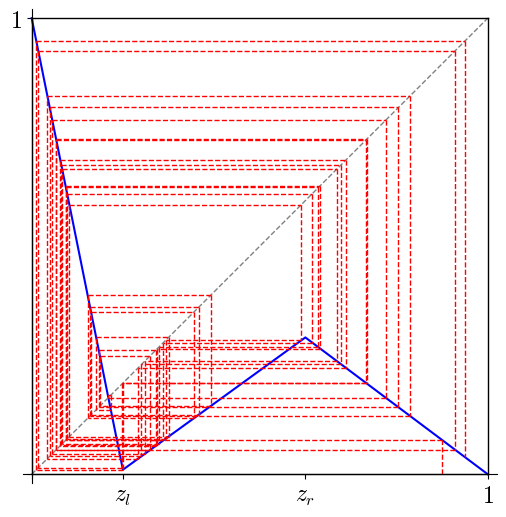}
\hfill
 \includegraphics[width=0.49\textwidth]{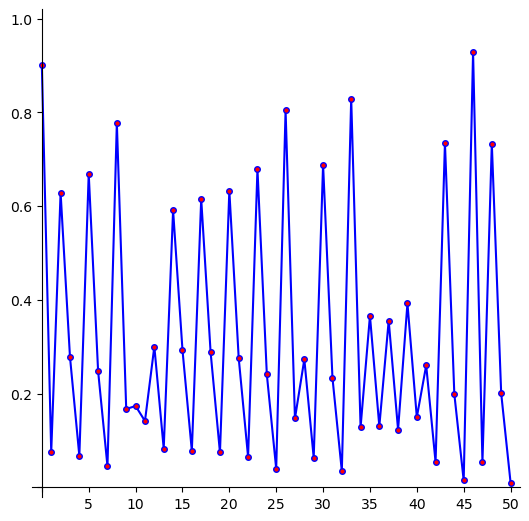}
 \caption{$f(z_l)=0.01$, $f(z_r)=0.3$}
\end{subfigure}
\caption{Diagrams illustrating conditions \ref{lem:chaos_cond1} and \ref{lem:chaos_cond2} of Proposition \ref{lem:chaos_cond} are in the first row. Diagrams illustrating conditions \ref{lem:chaos_condA} and \ref{lem:chaos_condB} are in the second row. Cobweb diagrams (columns 1. and 3.) and the first 50 iterations of a starting point (columns 2. and 4.). Here we test conditions of Proposition \ref{lem:chaos_cond} for $z_l=0.2$ and $z_r=0.6$ when they are critical points of piecewise linear maps:
{\bf (a)} None of the conditions is satisfied; %the trajectory of the starting point converges to the fixed point. 
{\bf (c)} The condition \ref{lem:chaos_condA} holds; %the trajectory of the starting point converges to the fixed point.
{[\bf (b)}~and~{\bf (d)]} Conditions of Proposition \ref{lem:chaos_cond} are satisfied; we see a chaotic behavior of the trajectory of the starting point.}\label{Cobweb}
\end{figure}

\begin{table*}[t]
\centering
\caption{\cm Description of the model}
\label{tableparameters}
\begin{tabularx}{\textwidth}{p{3.5cm}X} %p{1cm}
    \toprule
%{\cm Symbol} & 
{\cm Notation} & {\cm Description}  \cr
    \midrule
%{\cm $\pures$} & 
{\cm $\pures \equiv \setof{A,B}$} & {\cm  set of the strategies}\cr
%{\cm $\strats$} & 
{\cm the unit simplex $\strats$} & {\cm set of all possible states of the population (which consists of continuum of nonatomic agents)}\cr
%{\cm $\strat$} & 
{\cm $\strat = (\stratx_A,\stratx_B) \in \strats$} & {\cm state of the population}\cr
%{\cm $\mat$} & 
{\cm $\mat = (\mat_{\pure\purealt})_{\pure,\purealt\in\pures}$} & {\cm payoff matrix given by \eqref{eq:game}}\cr
%{\cm $\pay_{\pure}$} & 
{\cm $\pay_{\pure}(\strat) = \sum_{\purealt\in\pures} \mat_{\pure\purealt} \stratx_{\purealt}$} & {\cm payoff of agents playing $\pure\in\pures$ at state $\strat \in \strats$}\cr
%{\cm $\payv$} & 
{\cm $\payv(\strat) = (\pay_{A}(\strat),\pay_{B}(\strat))$} & {\cm   payoff vector at state $\strat \in \strats$}\cr
%{\cm $\game$} & 
{\cm $\game \equiv \gamefull$} & {\cm population game}\cr
%{\cm $\mathbf{p}$} & 
{\cm $\mathbf{p}=(p,1-p)$} & {\cm unique Nash equilibrium of \eqref{eq:game}, where $p$ is given by \eqref{eq:eq}}\cr
%{\cm $\switch_{\pure\purealt}(\strat)$} & 
{\cm $\switch_{\pure\purealt}(\strat) \in [0,\infty)$} & {\cm conditional switch rate at which a revising $\pure$-strategist switches to an another strategy $\purealt$ when the population is at state $\strat\in\strats$}\cr
%{\cm $\iswitch_{\pure\purealt}(\strat)$} & 
{\cm $\iswitch_{\pure\purealt}(\strat) \in [0,\infty)$} & {\cm  for imitative protocols, given by \eqref{eq:imitation}, conditional imitation rate at which a revising $\pure$-strategist switches to strategy $\purealt$ of an another individual when the population is at state $\strat\in\strats$}\cr
%{\cm $(X,f)$} & & {\cm $(X,f)$ is a dynamical system with compact metric space $X$ and a continuous map $f$.}\cr
%{\cm $\step$} & 
{\cm $\step\in (0,\infty)$} & {\cm length of a unit time interval between periods in a discrete time horizon, bounded due to the requirement for the dynamics introduced by \eqref{eq:updmap} being well-defined, in our considerations usually $\delta\leq 1$}\cr
%{\cm $\updmap$} & 
{\cm $\updmap \colon [0,1] \to [0,1]$} & {\cm map describing the dynamics of  \eqref{eq:game} introduced by \eqref{eq:in-out}; defined by \eqref{eq:updmap}}\cr
%{\cm $\stratcl, \stratcr$} & {\cm $\stratcl, \stratcr \in (0,1)$} & {\cm $\stratcl, \stratcr$ are the critical points of the map $\updmap$.}\cr
%{\cm $\eta$, $\xi$} & 
{\cm $\eta,\xi \in (0,\infty)$} & {\cm multiplicative constants that perturb the conditional imitation rates in Sections \ref{sec:thm2} and \ref{sec:thm3}}\cr
%{\cm $\gamma$} & 
{\cm $\gamma\in[0,1]\backslash\{p\}$} & {\cm  the level at which one of the conditional imitation rates is truncated in Section \ref{sec:thm3}}\cr
    \midrule
\end{tabularx}
\end{table*}

\section{Chaos for innovative dynamics}
\label{sec:thm1}

We will show that in discrete time setting the revision protocol driven game dynamics can be far from simple even in the simplest nontrivial case. For concreteness, we will focus on showing the extreme asymptotic behavior in comparison to the convergent landscape --- deterministic chaos. We will show that if only the length of a unit time interval between periods $\step$ is large enough, one can find revision protocols under which \eqref{eq:in-out} equation determines chaotic behavior. Why do we point to this extreme? If the dynamics is globally convergent, the population's initial state can be eventually forgotten, and all initializations ultimately settle down to the same state (or two states for coordination games). Instead, if the dynamics is chaotic, even arbitrarily small differences in the population's initial state may lead to drastically different behavior. As such, convergence and chaos can be seen as antithetical to each other.  
The main result of this note is the following theorem.
%{\color{red} a moze jedno twierdzenie?efektywnie chcemy powiedziec ze dla kazdej gry istnieje taki protokol. A jako ostatnie zdanie twierdzenia/ wniosek z twierdzenia mozemy podac, ze on jest PC albo imit? Co wiecej wazniejsze jest punkt 2 kazdego z twierdzen (choc p. 1 trzeba wskazac)}

\begin{theorem} \label{thm:anticoord-chaos}
    For every $2\times 2$ anti-coordination game \eqref{eq:game} there exist revision protocol and $\step$ such that the interior Nash equilibrium $\eq$ is repelling and game dynamics is chaotic in the sense of Li and Yorke with periodic orbits of any period. 
\end{theorem}

\begin{remark}
    By the proof of Theorem \ref{thm:anticoord-chaos} we can point out that such revision protocol can be either innovative  \eqref{eq:paircomp} or imitative \eqref{eq:imitation}.
\end{remark}

 Theorem \ref{thm:anticoord-chaos} shows that even for simplest nontrivial population games, where agents use revision protocols in discrete time, game dynamics can be complex and unpredictable, with periodic orbits of any period. Nevertheless, this doesn't give the answer to how attractors of the system look like. Moreover, 
note that we are considering chaos in the sense of Li and Yorke, which is local in nature. The set on which chaotic behavior occurs, although uncountable, can be very small (in the sense of the Lebesgue measure). Therefore, even an analytical proof of the existence of an uncountable scrambled set does not ensure that chaos will be \textit{observable}; in the extreme case we can still see convergence to the interior Nash equilibrium $p$ for almost all points \cite{collet1980abundance,bruin2010lebesgue,misiurewicz2005ergodic,BCFKMP21}. 
For this reason we will also be interested in the stability of the point $p$.\footnote{For innovative dynamics this is a unique fixed point of the system, while imitative dynamics admit also 0 and 1 as fixed points, but these are always repelling.}

 For the game \eqref{eq:game} we can focus on the population share of $A$-strategists. Then from \eqref{eq:updmap} we get
\begin{equation}\label{eq:game_dynamics-PC}
	\updmap(\stratx)=\begin{cases}
		\stratx+\step(1-\stratx)\switch_{BA}(\stratx),& \text{ for }\stratx\in[0,\eq),\\
		\stratx(1-\step\switch_{AB}(\stratx)),& \text{ for }\stratx\in[\eq,1],
	\end{cases}
\end{equation}
for revision protocols given by \eqref{eq:paircomp}, and

\begin{equation}\label{eq:game_dynamics_imit}
\updmap(\stratx)=\stratx\big(1+\step(1-\stratx)h(\stratx)\big),
\end{equation}
with $h(\stratx)=\iswitch_{BA}(\stratx)-\iswitch_{AB}(\stratx)$,
for imitative revision protocols, given by \eqref{eq:imitation}.

%In the proof of Theorem \ref{thm:anticoord-chaos} 
We will show that for any anti-coordination game \eqref{eq:game} there exists an innovative revision protocol \eqref{eq:paircomp}, which (via \eqref{eq:in-out} equation and by \eqref{eq:updmap}) introduces Li-Yorke chaotic game dynamics with periodic orbits of any period. One can also show existence of an imitative one with these properties. As the proof for the imitative revision protocol relies on the construction for the innovative one, and can be interpreted as a slight modification of the construction presented in this section, we leave the construction of an imitative revision protocol to the Appendix \ref{app:imitative}.\footnote{Moreover, in Sections \ref{sec:thm2} and \ref{sec:thm3} we will show another way of constructing an imitative revision protocol, which is a perturbation of the protocol \eqref{eq:PPI}, with chaotic game dynamics.}
 Both these results come from the construction of piecewise linear bimodal (with two critical points) maps with desired properties. 
 Our construction leads to piecewise linear maps, as those maps constitute the simplest possible class of bimodal maps, which at the same time captures a large part of the aspects of their dynamics, see e.g. \citet{milnor2000entropy}. %We use such maps because these are simplest maps of unit interval with interesting dynamics. 
 The choice of a bimodal map is a consequence of observation that these maps arise naturally in the context of anti-coordination and congestion games, when we meet unpredictable and complex behavior of trajectories \cite{palaiopanos2017multiplicative,Thip18,CFMP2019,BCFKMP21,bielawski2024memory,falniowski2024discrete}.
We denote by $\mathcal{B}_L$ the family of piecewise linear bimodal maps of the form:
\begin{equation}\label{PLBmap}
	\bimap(x)=\begin{cases}
		\beta_1x+\alpha_1, &\text{for}\ x\in[0,\stratcl),\\
		\beta_2x+\alpha_2, &\text{for}\ x\in [\stratcl,\stratcr),\\
		\beta_3x+\alpha_3, &\text{for}\ x\in [\stratcr,1],
	\end{cases}
\end{equation}
where $0<\stratcl<\stratcr<1$, $\beta_1\beta_2<0$, $\beta_1\beta_3>0$ and $\alpha_1,\alpha_2,\alpha_3\in\mathbb{R}$. As we will focus on continuous maps, $\stratcl$ and $\stratcr$ will be left and right critical points respectively (see Figure \ref{Cobweb}).
 
 In the construction of the innovative revision protocol we will focus on bimodal maps which are decreasing on the first and the third lap (such maps appear i.e. in the pairwise comparison dynamics \eqref{eq:Smithprot}), that is $\beta_1,\beta_3<0$ and $\beta_2>0$. On the other hand imitative protocols have to lead to interval maps, for which $0$ and $1$ are fixed points (this is a general property of imitation \cite{San10}). Therefore, an appropriate bimodal map should be increasing on the first (and third) lap (thus, $\beta_1,\beta_3>0$, $\beta_2<0$). 
 The proofs of all results presented in this section are in Appendix \ref{app:proof-thm1}.

\subsection{ Skeleton of the proof of Theorem \ref{thm:anticoord-chaos} for an innovative protocol}
Let us consider a $2\times 2$ anti-coordination game $\game \equiv \gamefull$ defined by \eqref{eq:game} and denote by $p$ its unique Nash equilibrium.  First, existence of the revision protocol which introduces chaotic behavior for sufficiently large $\step$ for $\eq=\frac 12$ was shown already in \cite{falniowski2024discrete}.\footnote{We will extend this result in Section \ref{sec:thm2}, see Proposition \ref{thm:imitation12}.}
We will show the proof of Theorem \ref{thm:anticoord-chaos} for $p \neq \frac 12$ in three steps.
%Assume that $\eq\in(0,\frac{1}{2})$ (we will solve the problem for $\eq>\frac 12$ later). %We want to reduce the number of parameters, therefore we set:

\vspace{0.2cm}

\noindent {\bf Step 1. Construction of a family of piecewise linear bimodal maps that satisfy the assumptions of Proposition ~\ref{lem:chaos_cond}.}
%In this note we 
We are interested in revision protocols for which the map $\updmap$ from \eqref{eq:game_dynamics-PC} (or from \eqref{eq:game_dynamics_imit}) are continuous and bimodal with two critical points $\stratcl, \stratcr \in (0,1)$.
For innovative dynamics the formula in \eqref{eq:game_dynamics-PC} results in a continuous bimodal map $\updmap$ with local minimum at $\stratcl$ and local maximum at $\stratcr$, where $0<\stratcl<\stratcr<1$ (i.e., the map $\updmap$ is decreasing on intervals $(0,\stratcl)$ and $(\stratcr,1)$).

As we mentioned earlier we will restrict our considerations to the family $\mathcal{B}_L$ of piecewise linear bimodal maps.
For fixed  points $\stratcl,\stratcr\in(0,1)$ values of parameters $\beta_1$, $\beta_2$, $\beta_3$ %and $\alpha_1$, $\alpha_2$, $\alpha_3$ 
for which the map  $\bimap$ from \eqref{PLBmap} satisfies the assumptions of Proposition ~\ref{lem:chaos_cond} is described in the following lemma. 

\begin{lemma}\label{cor:lin_bimodal1}
	Let $0<\stratcl<\stratcr<1$ and $\bimap\in\mathcal{B}_L$ be defined as
    \begin{equation}\label{eq:f_form}
	\bimap(x)=\begin{cases}
		\beta_1(x-\stratcl), &\text{for}\ x\in[0,\stratcl),\\
		\beta_2(x-\stratcl), &\text{for}\ x\in [\stratcl,\stratcr),\\
		\beta_3(x-\stratcr)+\beta_2(\stratcr-\stratcl), &\text{for}\ x\in [\stratcr,1],
	\end{cases}
\end{equation}
    where $\beta_1 \in \left[ -\frac{1}{\stratcl},-\frac{\stratcr}{\stratcl}\right)$, $\beta_2 \in \left(\frac{\stratcl}{\stratcr-\stratcl},\frac{1}{\stratcr-\stratcl}\right]$ and $\beta_3 \in \left[ -\frac{\stratcl}{1-\stratcr},0\right)$.
%	$$
%	(\beta_1,\beta_2,\beta_3)\in B_1\times B_2\times B_3=\left[ -\frac{1}{\stratcl},-\frac{\stratcr}{\stratcl}\right)\times\left(\frac{\stratcl}{\stratcr-\stratcl},\frac{1}{\stratcr-\stratcl}\right]\times\left[ -\frac{\stratcl}{1-\stratcr},0\right).
%	$$
	Then $\bimap$ is a well-defined continuous bimodal map of the interval $[0,1]$ with local minimum at $\stratcl$ and local maximum at $\stratcr$. Moreover, there is a point $x\in(\stratcl,\stratcr)$ such that $\bimap(x)<x<\bimap^3(x)$.
\end{lemma}

\vspace{0.2cm}

\noindent{\bf Step 2. Proof of Theorem~\ref{thm:anticoord-chaos} for $p\in(0,\frac{1}{2})$.}
We first show that Theorem~\ref{thm:anticoord-chaos} holds for $\eq\in(0,\frac{1}{2})$. We begin by reducing the number of parameters of the model by setting critical points of the constructed map:
\begin{equation}\label{clcrp}
\stratcl:=\frac{\eq}{1-\eq}\quad\text{and}\quad \stratcr:=2\eq.
\end{equation}
Our aim is to find switch rates $\switch_{AB}$ and $\switch_{BA}$ for which the map $\updmap$ in \eqref{eq:game_dynamics-PC}, where $\step>0$, has the form \eqref{eq:f_form}. In order to guarantee that the conditions \ref{lem:chaos_condA} and \ref{lem:chaos_condB} from Proposition \ref{lem:chaos_cond} hold and to simplify calculations we assume that $\step=1$.
%condition \ref{prop:cond2} from Corollary \ref{prop:cond123} holds and to simplify calculations we assume that $\step=1$.

\begin{proposition}\label{prop:inn_switch_rates}
Let $\game \equiv \gamefull$ be a $2\times 2$ anti-coordination game defined by \eqref{eq:game} with a Nash equilibrium $\eq\in(0,\frac{1}{2})$.\footnote{Thus, $\gaind-\gainb<\gaina-\gainc$.} We define the switch rates
\begin{equation}\label{rho12inn}
 \switch_{AB}(x) = \begin{cases}
	0, &\text{for}\ \stratx\in[0,\eq),\\
	\frac{\stratx-\eq}{\eq\stratx}  , &\text{for}\ \stratx\in[\eq,\frac{\eq}{1-\eq}),\\
	1-\beta_2\left(1-\frac{\eq}{(1-\eq)\stratx} \right)   , &\text{for}\ \stratx\in[\frac{\eq}{1-\eq},2\eq),\\
	1-\beta_3\left(1-\frac{2\eq}{\stratx} \right)-\beta_2\frac{\eq(1-2\eq)}{(1-\eq)\stratx}  , &\text{for}\ \stratx\in[2\eq,1],\\
\end{cases}
\end{equation}
and
\begin{equation}\label{rho21inn}
 \switch_{BA}(x) = \begin{cases}
	\frac{\eq-\stratx}{\eq(1-\stratx)}, &\text{for}\ \stratx\in[0,\eq),\\
	0, &\text{for}\ \stratx\in[\eq,1],\\
\end{cases}
\end{equation}
where $\beta_2 \in \left(\frac{1}{1-2\eq},\frac{2(1-\eq)}{1-2\eq}\right)$ and $\beta_3 \in \left[ -\frac{\eq}{(1-2\eq)(1-\eq)},0\right)$.
Then $\switch_{AB}$ and $\switch_{BA}$ are Lipschitz continuous on $\players$ and fulfill \eqref{eq:paircomp}. Moreover, game dynamics induced by revision protocols in  \eqref{eq:game_dynamics-PC} has periodic orbits of any period, is Li-Yorke chaotic and has exactly one fixed point $p$ which is repelling.
%the map $f$ in \eqref{eq:game_dynamics-PC} is piecewise linear and has the form \eqref{eq:map_f_form} from Proposition \ref{prop:chaos_inn_p<}.
\end{proposition}

We prove Proposition \ref{prop:inn_switch_rates} by deriving the form of the map $\updmap$ such that %: (i) it includes the simplification in \eqref{clcrp} and the fact that it has a unique Nash equilibrium; (ii) 
it satisfies the assumptions of Lemma \ref{cor:lin_bimodal1}. We then use the formula \eqref{eq:game_dynamics-PC} to obtain the corresponding switch rates $\switch_{AB}$ and $\switch_{BA}$.
From Proposition \ref{prop:inn_switch_rates} follows Theorem~\ref{thm:anticoord-chaos} for $p\in(0,\frac{1}{2})$.

Observe that the switch rates $\switch_{AB}$ and $\switch_{BA}$ from Proposition \ref{prop:inn_switch_rates} depend on three parameters: $\eq$, $\beta_2$ and $\beta_3$. While $\eq$ is determined by the game, the remaining parameters can be chosen based on Lemma \ref{cor:lin_bimodal1}. To show how these switch rates may look like, as an example, we plot switch rates on Figure \ref{fig:switch_rates} for $p=0.2$, $\beta_2=2$, and $\beta_3=-\frac{1}{3}$.

\begin{figure}[!h]
	\centering
	\begin{tikzpicture}[scale=5.3]
		\draw[thick] (0,0) -- (1,0);
		\draw[thick] (0,0) -- (0,1);
		\draw (0,1) node[left]{1};
		\draw (1,0) node[below]{1};
		\draw (-0.02,0) node[below]{0}; 
		\draw[thick,domain=0:1,variable=\y] plot ({1},{\y});
		\draw[thick,domain=0:1,variable=\x] plot ({\x},{1});
		%y=x
		%\draw[domain=0:1,smooth,variable=\x,red,very thick] plot ({\x},{\x});
		%(0,b)
		\draw[domain=0:0.2,smooth,variable=\x,violet,very thick] plot ({\x},{0});
		%(b,\stratcl)
		\draw[domain=0.2:0.25,smooth,variable=\x,violet,very thick] plot ({\x},{1/0.2-1/(\x)});
		%(\stratcl,\stratcr)
		\draw[domain=0.25:0.4,smooth,variable=\x,violet,very thick] plot 	({\x},{1-2*(1-0.25/(\x))});
		%(\stratcr,1)
		\draw[domain=0.4:1,smooth,variable=\x,violet,very thick] plot 	({\x},{1+1/3*(1-0.4/(\x))-2*(0.4-0.25)/(\x))});
		%points
		%\filldraw [red] (0.2,0) circle (0.2pt) node[anchor=north] {$\stratcl$};
		\filldraw [red] (0.25,0) circle (0.2pt);
		\draw (0.25-0.03,-0.04) node[right,red]{$\stratcl$};
		\filldraw [red] (0.4,0) circle (0.2pt) node[anchor=north] {$\stratcr$};
		%\filldraw [black] (-4.5*0.2/-5.5,0) circle (0.2pt) node[anchor=north] {$b$};
		\filldraw [black] (0.2,0) circle (0.2pt) node[anchor=north] {$p$};
		%\draw (-4.5*0.2/-5.5+0.007,-0.04) node[left]{$b$};
	\end{tikzpicture}
	\qquad
	\begin{tikzpicture}[scale=5.3]
		\draw[thick] (0,0) -- (1,0);
		\draw[thick] (0,0) -- (0,1);
		\draw (0,1) node[left]{1};
		\draw (1,0) node[below]{1};
		\draw (-0.02,0) node[below]{0}; 
		\draw[thick,domain=0:1,variable=\y] plot ({1},{\y});
		\draw[thick,domain=0:1,variable=\x] plot ({\x},{1});
		%y=x
		%\draw[domain=0:1,smooth,variable=\x,red,very thick] plot ({\x},{\x});
		%(0,b)
		\draw[domain=0:0.2,smooth,variable=\x,orange,very thick] plot ({\x},{(0.2-\x)/(0.2*(1-\x))});
		%(b,1)
		\draw[domain=0.2:1,smooth,variable=\x,orange,very thick] plot ({\x},{0});
		%points
		%\filldraw [red] (0.2,0) circle (0.2pt) node[anchor=north] {$\stratcl$};
		%\filldraw [red] (0.6,0) circle (0.2pt) node[anchor=north] {$\stratcr$};
		\filldraw [black] (0.2,0) circle (0.2pt) node[anchor=north] {$p$};
		%\filldraw [black] (-4.5*0.2/-5.5,0) circle (0.2pt);
		%\draw (-4.5*0.2/-5.5+0.007,-0.02) node[left]{$b$};
	\end{tikzpicture}
\caption{Switch rates \textcolor{violet}{$\switch_{AB}$} (on the left) and \textcolor{orange}{$\switch_{BA}$} (on the right) from Proposition \ref{prop:inn_switch_rates} for the game \eqref{eq:game} with Nash equilibrium  $\eq=0.2$, and parameters $\beta_2=2$, $\beta_3=-\frac{1}{3}$.}
	\label{fig:switch_rates}
\end{figure}

\vspace{0.2cm}

\noindent{\bf Step 3. Proof of Theorem~\ref{thm:anticoord-chaos} for $\widetilde{p}\in(\frac{1}{2},1)$.}
Now, consider the game $\widetilde{\mathcal{G}}\equiv\widetilde{\mathcal{G}}(\mathcal{A},\widetilde{\payv})$, given be \eqref{eq:game}, where $\widetilde{\payv}(\stratx)=(\widetilde{\pay}_1(\stratx),\widetilde{\pay}_2(\stratx))$ is its payoff vector, with the unique (symmetric) Nash equilibrium $\widetilde{\eq}\in(\frac{1}{2},1)$. 
Then there exists a population game $\game \equiv \gamefull$ with $\payv(\stratx)=(\pay_A(\stratx),\pay_B(\stratx))$ with the unique Nash equilibrium $\eq=1-\widetilde{\eq}\in (0,\frac 12)$ and the switch rates $\switch_{AB}$, $\switch_{BA}$ from Proposition \ref{prop:inn_switch_rates}. Thus, to complete the proof of Theorem \ref{thm:anticoord-chaos} it is sufficient to show the following fact.

%lead to the Li-Yorke chaotic dynamical system with the unique fixed point $p$, which is repelling. We will show this using

\begin{proposition}\label{prop:chaos_inn_p>}
Let $\widetilde{\mathcal{G}}\equiv\widetilde{\mathcal{G}}(\mathcal{A},\widetilde{v})$ be a $2\times 2$ anti-coordination game defined by \eqref{eq:game} with Nash equilibrium $\widetilde{\eq}\in(\frac{1}{2},1)$. Define
\begin{equation}\label{def_tilda_switch}
    \widetilde{\switch}_{AB}(\stratx):=\switch_{BA}(1-\stratx) \qquad\text{and}\qquad \widetilde{\switch}_{BA}(\stratx):=\switch_{AB}(1-\stratx),
\end{equation}
%$\widetilde{\switch}_{AB}(\stratx):=\switch_{BA}(1-\stratx)$ and $\widetilde{\switch}_{BA}(\stratx):=\switch_{AB}(1-\stratx)$,
for every point $\stratx\in[0,1]$, where $\switch_{AB}$ and $\switch_{BA}$ are given in Proposition \ref{prop:inn_switch_rates}. Then the map
\[
\widetilde{\updmap}(x)=\begin{cases}
	x+\delta(1-x)\widetilde{\switch}_{BA}(x), &\text{ for } x\in[0,\widetilde{\eq}),\\
	x(1-\delta\widetilde{\switch}_{AB}(x)), &\text{ for } x\in[\widetilde{\eq},1],
\end{cases}
\]
is a piecewise linear bimodal map with local minimum at $\stratclalt:=1-\stratcr$ and local maximum at $\stratcralt:=1-\stratcl$, and the dynamical system induced by $\widetilde{\updmap}$ is Li-Yorke chaotic. Moreover, $\widetilde{\eq}$ is the unique fixed point of $\widetilde{\updmap}$ and it is repelling.
\end{proposition}

The proof of Proposition \ref{prop:chaos_inn_p>} is based on the fact that the maps $\widetilde{\updmap}$, derived from the switch rates $\widetilde{\switch}_{AB}$, $\widetilde{\switch}_{BA}$, and $\updmap$, obtained from the switch rates $\switch_{AB}$, $\switch_{BA}$, are topologically conjugate (one can compare the graphs of the maps $\updmap$ and $\widetilde{\updmap}$ in Figure \ref{fig:f_and_widetildef}). As a result, the chaotic behavior of agents observed in the game $\game$ will be also present in the game $\widetilde{\mathcal{G}}$. The proof of the Theorem~\ref{thm:anticoord-chaos} is completed.

\begin{figure}[!h]
	\centering
	\begin{tikzpicture}[scale=5.3]
		\draw[thick] (0,0) -- (1,0);
		\draw[thick] (0,0) -- (0,1);
		\draw (0,1) node[left]{1};
		\draw (1,0) node[below]{1};
		\draw (-0.02,0) node[below]{0}; 
		\draw[thick,domain=0:1,variable=\y] plot ({1},{\y});
		\draw[thick,domain=0:1,variable=\x] plot ({\x},{1});
		%y=x
		\draw[domain=0:1,smooth,variable=\x,black,dashed] plot ({\x},{\x});
		%(0,\stratcl)
		\draw[domain=0:0.25,smooth,variable=\x,blue,very thick] plot ({\x},{(0.2-1)/(0.2)*\x+1});
		%(\stratcl,\stratcr)
		\draw[domain=0.25:0.4,smooth,variable=\x,blue,very thick] plot 	({\x},{2*(\x-0.2/(1-0.2))});
		%(\stratcr,1)
		\draw[domain=0.4:1,smooth,variable=\x,blue,very thick] plot 	({\x},{-1/3*(\x-2*0.2)+2*(0.2*0.6)/0.8});
		%points
		%\filldraw [red] (0.2,0) circle (0.2pt) node[anchor=north] {$\stratcl$};
		\filldraw [red] (0.25,0) circle (0.2pt);
		\draw (0.25-0.03,-0.04) node[right,red]{$\stratcl$};
		\filldraw [red] (0.4,0) circle (0.2pt) node[anchor=north] {$\stratcr$};
		%\filldraw [black] (-4.5*0.2/-5.5,0) circle (0.2pt) node[anchor=north] {$b$};
		\filldraw [black] (0.2,0) circle (0.2pt) node[anchor=north] {$p$};
		%\draw (-4.5*0.2/-5.5+0.007,-0.04) node[left]{$b$};
	\end{tikzpicture}
	\qquad
	\begin{tikzpicture}[scale=5.3]
		\draw[thick] (0,0) -- (1,0);
		\draw[thick] (0,0) -- (0,1);
		\draw (0,1) node[left]{1};
		\draw (1,0) node[below]{1};
		\draw (-0.02,0) node[below]{0}; 
		\draw[thick,domain=0:1,variable=\y] plot ({1},{\y});
		\draw[thick,domain=0:1,variable=\x] plot ({\x},{1});
		%y=x
		\draw[domain=0:1,smooth,variable=\x,black,dashed] plot ({\x},{\x});
		%(0,\stratcl)
		\draw[domain=0:1-0.4,smooth,variable=\x,blue,very thick] plot ({\x},{-1/3*\x-0.6*(2*0.2/0.8-1/3)+1});
		%(\stratcl,\stratcr)
		\draw[domain=1-0.4:1-0.25,smooth,variable=\x,blue,very thick] plot 	({\x},{2*\x+2*(0.4-1)/0.8+1});
		%(\stratcr,1)
		\draw[domain=1-0.25:1,smooth,variable=\x,blue,very thick] plot 	({\x},{(0.2-1)/(0.2)*(\x-1)});
		%points
		%\filldraw [red] (0.2,0) circle (0.2pt) node[anchor=north] {$\stratcl$};
		\filldraw [red] (1-0.25,0) circle (0.2pt);
		\draw (1-0.22,-0.055) node[left,red]{$\stratcralt$};
		\filldraw [red] (1-0.4,0) circle (0.2pt) node[anchor=north] {$\stratclalt$};
		%\filldraw [black] (-4.5*0.2/-5.5,0) circle (0.2pt) node[anchor=north] {$b$};
		\filldraw [black] (1-0.2,0) circle (0.2pt) node[anchor=north] {$\widetilde{\eq}$};
		%\draw (-4.5*0.2/-5.5+0.007,-0.04) node[left]{$b$};
	\end{tikzpicture}
	\caption{The graphs of maps $\updmap$ induced by the switch rates $\switch_{AB}$ and $\switch_{BA}$ from Proposition \ref{prop:inn_switch_rates} for $\eq=0.2$ (on the left) and $\widetilde{\updmap}$ from Proposition \ref{prop:chaos_inn_p>} for $\widetilde{\eq}=0.8$ (on the right) %for the game with Nash equilibrium $\eq=0.2$,
    and for parameters $\beta_2=2$, $\beta_3=-\frac{1}{3}$.}
	\label{fig:f_and_widetildef}
\end{figure}

%{\color{red} tu wstawiłbym remark podsumowujący roznice oraz rysunki przykładowych revision protocols/funkcji dających chaos...}

%\begin{remark}    {\color{red} Case of $\eq=1/2$ can be also shown by a slight modification of our construction. In particular,...}
%\end{remark}
\begin{remark}
One can also construct an appropriate imitative revision protocol introducing chaotic game dynamics. This construction follows similar steps as presented above, that is, it is based on analogs of 
%Corollary \ref{prop:cond123},
Lemma \ref{cor:lin_bimodal1} and Propositions \ref{prop:inn_switch_rates} and \ref{prop:chaos_inn_p>}. Nevertheless, bimodal maps introduced by imitative protocols has two additional fixed points: 0 and 1 (see Figure  \ref{fig:g_and_gtilde} presenting the switch rates $\switch_{AB}$, $\switch_{BA}$ and the maps $\updmap$ and $\widetilde{\updmap}$ of an imitative revision protocol). %for Nash equilibrium $\eq=0{.}4$).
We present details of the construction in the Appendix \ref{app:imitative}. %\ref{app:imitative}. %(see Proposition \ref{prop:f_imit_chaotic} in Appendix \ref{app:imitative}).
\end{remark}

\begin{figure}[!h]
\begin{subfigure}{0.48\textwidth}
	\begin{tikzpicture}[scale=5.3]
		\draw[thick] (0,0) -- (1,0);
		\draw[thick] (0,0) -- (0,1);
		\draw (0,1) node[left]{1};
		\draw (1,0) node[below]{1};
		\draw (-0.02,0) node[below]{0}; 
		\draw[thick,domain=0:1,variable=\y] plot ({1},{\y});
		\draw[thick,domain=0:1,variable=\x] plot ({\x},{1});
		%y=x
		%\draw[domain=0:1,smooth,variable=\x,red,very thick] plot ({\x},{\x});
		%p_12
		%(0,b)
		\draw[domain=0:2/5,smooth,variable=\x,violet,very thick] plot ({\x},{0});
		%(b,\stratcr)
		\draw[domain=2/5:12/25,smooth,variable=\x,violet,very thick] plot ({\x},{-(\x-2/5)*(2/5-2)/(2/5*\x)});
		%(\stratcr,1)
		\draw[domain=12/25:1,smooth,variable=\x,violet,very thick] plot ({\x},{2/5*(2/5-2)*(1-\x)/((4/25+2*2/5-2)*\x)});
		%points
		%\filldraw [red] (0.2,0) circle (0.2pt) node[anchor=north] {$\stratcl$};
		\filldraw [red] (12/25,0) circle (0.2pt) node[anchor=north] {$\stratcr$};
		\filldraw [black] (0.4,0) circle (0.2pt) node[anchor=north] {$p$};
		%\filldraw [black] (-4.5*0.2/-5.5,0) circle (0.2pt);
		%\draw (-4.5*0.2/-5.5+0.007,-0.02) node[left]{$b$};
	\end{tikzpicture}
    \caption{Graph of the switch rate $\switch_{AB}$.}
\end{subfigure}
\hfill
\begin{subfigure}{0.48\textwidth}
	\begin{tikzpicture}[scale=5.3]
		\draw[thick] (0,0) -- (1,0);
		\draw[thick] (0,0) -- (0,1);
		\draw (0,1) node[left]{1};
		\draw (1,0) node[below]{1};
		\draw (-0.02,0) node[below]{0}; 
		\draw[thick,domain=0:1,variable=\y] plot ({1},{\y});
		\draw[thick,domain=0:1,variable=\x] plot ({\x},{1});
		%y=x
		%\draw[domain=0:1,smooth,variable=\x,red,very thick] plot ({\x},{\x});
		%p_21
		%(0,\stratcl)
		\draw[domain=0:1/5,smooth,variable=\x,orange,very thick] plot ({\x},{\x*(2-2/5)/(2/5*(1-\x))});
		%(\stratcl,b)
		\draw[domain=1/5:2/5,smooth,variable=\x,orange,very thick] plot ({\x},{(\x-2/5)*(2/5-2)/(2/5*(1-\x))});
		%(b,1)
		\draw[domain=2/5:1,smooth,variable=\x,orange,very thick] plot ({\x},{0});
		%points
		\filldraw [red] (0.2,0) circle (0.2pt) node[anchor=north] {$\stratcl$};
		%\filldraw [red] (0.2,0) circle (0.2pt);
		%\draw (0.2-0.007,-0.04) node[right,red]{$\stratcl$};
		%\filldraw [red] (0.6,0) circle (0.2pt) node[anchor=north] {$\stratcr$};
		\filldraw [black] (0.4,0) circle (0.2pt) node[anchor=north] {$p$};
		%\filldraw [black] (-4.5*0.2/-5.5,0) circle (0.2pt);
		%\draw (-4.5*0.2/-5.5+0.007,-0.04) node[left]{$b$};
	\end{tikzpicture}
    \caption{Graph of the switch rate $\switch_{BA}$.}
\end{subfigure}
\hfill
\begin{subfigure}{0.48\textwidth}
	\begin{tikzpicture}[scale=5.3]
		\draw[thick] (0,0) -- (1,0);
		\draw[thick] (0,0) -- (0,1);
		\draw (0,1) node[left]{1};
		\draw (1,0) node[below]{1};
		\draw (-0.02,0) node[below]{0}; 
		\draw[thick,domain=0:1,variable=\y] plot ({1},{\y});
		\draw[thick,domain=0:1,variable=\x] plot ({\x},{1});
		%y=x
		\draw[domain=0:1,smooth,variable=\x,black,dashed] plot ({\x},{\x});
		%(0,\stratcl)
		\draw[domain=0:1/5,smooth,variable=\x,blue,very thick] plot ({\x},{5*\x});
		%(\stratcl,\stratcr)
		\draw[domain=1/5:12/25,smooth,variable=\x,blue,very thick] plot ({\x},{-2*(1-2/5)*5/2*\x+2-2/5});
		%(\stratcr,1)
		\draw[domain=12/25:1,smooth,variable=\x,blue,very thick] plot ({\x},{2*(1-4/25)/(4/25+4/5-2)*(1-\x)+1});
		%points
		\filldraw [red] (1/5,0) circle (0.2pt) node[anchor=north] {$\stratcl$};
		\filldraw [black] (0.4,0) circle (0.2pt) node[anchor=north] {$p$};
		\filldraw [red] (12/25,0) circle (0.2pt) node[anchor=north] {$\stratcr$};
	\end{tikzpicture}
    \caption{Graph of the  map $\updmap$.}
\end{subfigure}
\hfill
\begin{subfigure}{0.48\textwidth}
	\begin{tikzpicture}[scale=5.3]
		\draw[thick] (0,0) -- (1,0);
		\draw[thick] (0,0) -- (0,1);
		\draw (0,1) node[left]{1};
		\draw (1,0) node[below]{1};
		\draw (-0.02,0) node[below]{0}; 
		\draw[thick,domain=0:1,variable=\y] plot ({1},{\y});
		\draw[thick,domain=0:1,variable=\x] plot ({\x},{1});
		%y=x
		\draw[domain=0:1,smooth,variable=\x,black,dashed] plot ({\x},{\x});
		%(0,\stratcl)
		\draw[domain=0:1-12/25,smooth,variable=\x,blue,very thick] plot ({\x},{-2*(1-4/25)/(4/25+4/5-2)*\x});
		%(\stratcl,\stratcr)
		\draw[domain=1-12/25:1-1/5,smooth,variable=\x,blue,very thick] plot ({\x},{2*(1-2/5)*5/2*(1-\x)-1+2/5});
		%(\stratcr,1)
		\draw[domain=1-1/5:1,smooth,variable=\x,blue,very thick] plot ({\x},{1-5*(1-\x)});
		%points
		\filldraw [red] (1-12/25,0) circle (0.2pt) node[anchor=north] {$\stratclalt$};
		\filldraw [black] (1-2/5,0) circle (0.2pt) node[anchor=north] {$\widetilde{\eq}$};
		\filldraw [red] (1-1/5,0) circle (0.2pt) node[anchor=north] {$\stratcralt$};
	\end{tikzpicture}
    \caption{Graph of the  map $\widetilde{\updmap}$.}
\end{subfigure}
\caption{Graphs of the elements of the dynamical system induced by an imitative revision protocol. {[\bf (a)} and {\bf (b)]} Graphs of the switch rates for Nash equilibrium $\eq=0{.}4$. {\bf (c)} Graph of the update map for Nash equilibrium $\eq=0{.}4$. {\bf (d)} Graph of the update map for Nash equilibrium $\widetilde{\eq}=0{.}6$. }
	\label{fig:g_and_gtilde}
\end{figure}

\section{A natural imitative protocol leading to chaos}
\label{sec:thm2}
The proof of Theorem~\ref{thm:anticoord-chaos} focuses on constructing a revision protocol that leads to chaotic game dynamics. This is a convenient way to show that certain dynamical scenarios are possible as long as the agents behave in an appropriate manner. However, the nature of this proof does not provide much insight into the extent to which their behavior may be sensible or natural. In fact, the revision protocols obtained in this way may be artificial and difficult to interpret, as the switch rates are not necessarily monotonously correlated with the payoff difference (see Figure~\ref{fig:switch_rates}). %and \ref{fig:rev_prot}).
For this reason, in this section we will present a modified variant of the pairwise proportional imitation protocol \eqref{eq:PPI} that has simple interpretation and induces Li-Yorke chaotic game dynamics for any game defined by \eqref{eq:game}. As PPI is probably the best known revision protocols (e.g., due to its connection with replicator dynamics in the continuous case), this result shows that unpredictable behavior of game dynamics is an important property which one should be aware of even when dealing with very simple games.  All proofs of results from this section are delegated to Appendix~\ref{app:proof-thm2}.

We begin with a simple observation that chaotic behavior can be detected in any imitative (symmetric) revision protocol. Thus, chaos is not an exceptional property but an attribute of imitation.
\begin{proposition} \label{thm:imitation12}
Consider a game defined by \eqref{eq:game} with $\gaind-\gainb=\gaina-\gainc$, that is the Nash equilibrium of the game is $\eq=1/2$. Assume that the conditional imitation rates $\iswitch_{AB}$ and $\iswitch_{BA}$ fullfill condition
\begin{equation} \label{eq:iswitchimit} \iswitch_{AB}(\stratx)+\iswitch_{AB}(1-\stratx)=\iswitch_{BA}(\stratx)+\iswitch_{BA}(1-\stratx)\end{equation}
for any $\stratx\in [0,1]$. Then there exists sufficiently large $\step$ such that the game dynamics introduced by \eqref{eq:imitation} is Li-Yorke chaotic and has periodic orbits of any period. 
\end{proposition}

Proposition \ref{thm:imitation12} guarantees that in the totally symmetric case imitative revision protocols will lead to unpredictable behavior, which is in stark contract with equilibrium predictions for small values of $\step$.
This result can be treated as a consequence of the proof of Theorem 1 from \cite{falniowski2024discrete}. However, we  present the proof of Proposition \ref{thm:imitation12} in Appendix~\ref{app:proof-thm2}.

%Recall that for a game $\game\equiv\gamefull$ defined by \eqref{eq:game} the entries of payoff matrix satisfy the inequalities $\gaina<\gainc$ and $\gaind<\gainb$. Moreover, there exists an interior Nash equilibrium given by
%\[
%\eq=\frac{\gaind-\gainb}{\gaina-\gainc+\gaind-\gainb}.
%\]

Let us consider a game $\game\equiv\gamefull$ defined by \eqref{eq:game}. %We first determine the elements of the dynamics induced by the pairwise proportional imitation protocol. To this aim we use \eqref{payoffs} to obtain (cf. \eqref{eq:PPI}, \eqref{eq:imitation} and \eqref{eq:updmap}):\\
%\begin{itemize}
While for \eqref{eq:PPI} the rate of switching from strategy $A$ to strategy $B$ is $(1-\stratx)\left[\frac{b-d}{\eq}(\stratx-\eq)\right]_+$ and the rate of switching from strategy $B$ to strategy $A$ is $\stratx\left[\frac{b-d}{\eq}(\eq-\stratx)\right]_+$,
\begin{comment}
are 
%the conditional imitation rates:
\begin{equation}\label{eq:PPI_cond_imit_rates}
%	\begin{aligned}
%\pospart{\pay_{B}(\stratx) - \pay_{A}(\stratx)}=
\switch_{AB}(\stratx) =
        (1-x)\left[\frac{b-d}{p}(p-\stratx)\right]_+ \;\;\text{and}\;\;
%		\begin{cases}
%			0, &\text{for}\ \stratx\in\left[0,p\right)\\
%			\frac{b-d}{p}(\stratx-p), &\text{for}\ \stratx\in\left[p,1\right]\\
%		\end{cases},\\
%	\pospart{\pay_{A}(\stratx) - \pay_{B}(\stratx)}=
\switch_{BA}(\stratx) =
         x\left[\frac{b-d}{p}(\stratx-p)\right]_+,\\
\end{equation}
\end{comment}
we define the perturbed conditional imitation rates as
\begin{equation}\label{eq:PPI_cond_imit_rates_pert}
\iswitch_{AB}(\stratx) = \eta \left[\frac{b-d}{\eq}(\stratx-\eq)\right]_+
        \qquad\text{and}\qquad
\iswitch_{BA}(\stratx) = \xi\left[\frac{b-d}{\eq}(\eq-\stratx)\right]_+
\end{equation}
for $\eta,\xi>0$. By including the maps~\eqref{eq:PPI_cond_imit_rates_pert} into the formula~\eqref{eq:game_dynamics_imit}, we obtain the family of maps

\begin{equation}\label{eq:PPI_map_pert}
	\updmap%_{\eta,\xi,\step}`
    (\stratx)=
	\begin{cases}
		x\left(1+\step\xi\frac{b-d}{p}(1-x)(p-x)\right), &\text{for}\ x\in[0,p)\\
		x\left(1-\step\eta\frac{b-d}{p}(1-x)(x-p)\right), &\text{for}\ x\in[p,1]\\
	\end{cases}.
\end{equation}

Observe that the map $\updmap$ in \eqref{eq:PPI_map_pert} for $\eta=\xi=1$ is the update map for \eqref{eq:PPI}. Then, by Proposition~\ref{thm:imitation12}, for $\eq=\frac{1}{2}$ and $\step$ sufficiently large, the map $\updmap$ is Li-Yorke chaotic. %(for $\eta=\xi=1$).

The multipliers $\eta$ and $\xi$ modify the conditional imitation rates of \eqref{eq:PPI} in two ways. First, they 
increase (when $\eta>1$ and $\xi > 1$) or decrease (when $\eta \in (0,1)$ and $\xi \in (0,1)$) the impact of differences in payoffs.
%make their values lager (when $\eta, \xi > 1$) or smaller (when $\eta, \xi \in (0,1)$). 
Second, when $\eta \neq \xi$ the model becomes non-symmetric.
We will use $\eta$ and $\xi$ to rescale the switch rates of \eqref{eq:PPI} in such a way that the resulting revision protocol will lead to chaotic behavior for other games (other choice of the game/Nash equilibrium $\eq$), where $\eq\neq 1/2$. %To observe this phenomenon we consider two multipliers $\eta, \xi > 0$ that allow us to

The main result of this section is the following fact.
\begin{theorem} \label{thm:perturbedPPIchaos}
    For every population game $\game\equiv\gamefull$ defined by \eqref{eq:game} there exist
    %$\varepsilon_1,\varepsilon_2>0$
    $\eta, \xi > 0$ such that in the imitative game dynamics introduced by the imitation rates \eqref{eq:PPI_cond_imit_rates_pert} %(and \eqref{eq:PPI_map_pert}) 
    there are periodic orbits of any period and the system is Li-Yorke chaotic for sufficiently large $\step$. %{\color{red}Furthermore there is a periodic orbit of any period $n\in\mathbb{N}_0$. [Jakoś bym to jeszcze przeformułował z orbitami okresowymi]}
\end{theorem}

We will restrict the range of parameters $\eta$, $\xi$ and $\step$ to a set for which $\eqref{eq:PPI_map_pert}$ forms a family of interval maps and the switch rates %$\varepsilon_1\switch_{AB}$ and $\varepsilon_2\switch_{BA}$
$\switch_{AB}$ and $\switch_{BA}$ derived from \eqref{eq:PPI_cond_imit_rates_pert} %defined by the conditional imitation rates $\eqref{eq:PPI_cond_imit_rates_pert}$
stay in the unit interval.

\begin{remark}
The choice of parameters $\eta$ and $\xi$ in \eqref{eq:PPI_cond_imit_rates_pert} reflects how important the payoff difference is for a revising agent. As the change of the strategy may happen only for a strategy, which gives larger payoff (thus, for any population state at most one of conditional switch rates is positive), 
a larger increase in value of a multiplier, results in a greater weight attached to the advantage of the benchmark strategy $j$ over the incumbent strategy $i$. This perspective is consistent with recent findings in machine learning and economics \cite{daskalakis2018last,stark2017class,obloj2017organization,quintana2016relative}.
\end{remark}

\subsection{Skeleton of the proof of Theorem \ref{thm:perturbedPPIchaos}.}
Let us consider a game $\game\equiv\gamefull$ defined by \eqref{eq:game} with the interior Nash equilibrium $p$. In the following, we present the skeleton of the proof of Theorem \ref{thm:perturbedPPIchaos} in three steps, while the details of the proof are included in Appendix~\ref{app:proof-thm2}.

\vspace{0.2cm}

\noindent{\bf Step 1. Introducing the set of parameters for which \eqref{eq:PPI_map_pert} forms a family of interval maps.}
The pair $([0,1],\updmap)$, with $\updmap$ defined as in \eqref{eq:PPI_map_pert}  will be a well-defined dynamical system  only when $\updmap : [0,1] \to [0,1]$. To meet this condition we need to restrict the set of values of parameters $\eta$ and $\xi$. For this purpose we define the set:
\[
 \Delta_p := \left\{ (\nu_1,\nu_2,\nu_3) \in (0,\infty)^3 : \nu_1 \leqslant \frac{4p}{(1-p)^2(b-d)}, \nu_2 \leqslant \frac{4}{p(b-d)} \text{ and } \nu_3 \in (0,1] \right\}.
\]

First we consider the case $p \in (0,\frac 12]$. For $p\in(\frac{1}{2},1)$, it follows from the topological conjugacy argument, which we describe in Step 3. The following lemma binds the parameters of the model with the set $\Delta_p$ such that the game dynamics is well-defined. 

\begin{lemma}\label{cor:PPI_interval_map}
Let $\eq\in(0,\frac{1}{2}]$ be the Nash equilibrium of the game $\game$.
%Let
%\[
% \Delta_p := \left\{ (\eta,\xi,\step) : \eta \leqslant \frac{4p}{(1-p)^2(b-d)}, \xi \leqslant \frac{4}{p(b-d)} \text{ and } \delta \in (0,1] \right\}.
%\]
Then for any parameters $(\eta,\xi,\step)\in\Delta_p$ we have $\updmap([0,1])=[0,1]$. In particular, $\updmap$ is an interval map.
\end{lemma}

We show Lemma \ref{cor:PPI_interval_map} in two steps: first we prove it for maximal possible values of  $\eta$, $\xi $ and $\delta$. We then extend the result for other values of those parameters.\footnote{For other properties of the family of maps \eqref{eq:PPI_map_pert} see Lemma \ref{prop:f_max} in Appendix \ref{app:proof-thm2}.}

\vspace{0.2cm}

\noindent{\bf Step 2. Proof of Theorem \ref{thm:perturbedPPIchaos} for $p\in(0,\frac{1}{2}]$.}
Note that if $\eq=\frac 12$, then for $\eta=\xi$ the perturbed conditional imitation rates from \eqref{eq:PPI_cond_imit_rates_pert} satisfy condition \eqref{eq:iswitchimit}. So, this case is a consequence of Proposition \ref{thm:imitation12}. Thus, we restrict our considerations to the case when  $\eq\in(0,\frac{1}{2})$. We will show that there exists a threshold value $\step_\eq$, % of the parameter $\step$ (depending on $\eq$), 
exceeding of which guarantees obtaining chaotic dynamics for an appropriate modification of PPI protocol.

\begin{proposition}\label{Thm2:step2}
Let $p\in(0,\frac{1}{2})$ be the Nash equilibrium of the game $\game$ and let $\eta = \frac{4p}{(1-p)^2(b-d)}$ and $\xi = \frac{4}{p(b-d)}$. Then there exists $\step_p \in (0,1)$ such that the map $\updmap$ is Li-Yorke chaotic and has periodic orbit of any period for each $\step\in(\step_p,1]$.
\end{proposition}

\begin{comment}
In order to simplify the calculations, we will consider the case of maximal perturbations (expressed by the values $\eta=N$ and $\xi=M$). However, chaotic behavior of the map $f_{\eta,\xi,\step}$ can also be observed for smaller values of the parameters $\eta$ and $\xi$. So, let us denote
$$ 
\step_p:=\max\left\lbrace\frac{1}{p+1},\;\frac{1}{8}\left(p+3+ \sqrt{\frac{-p^3-4p^2-13p+34}{2-p}}\right)\right\rbrace
$$
and consider a map $f:=f_{\eta,\xi,\step}$ defined by the parameters $(\eta,\xi):=(N,M)$ and $\step\in(\step_p,1]$. This choice of parameters ensures that the map $f$ satisfies the assumptions of Proposition~\ref{lem:chaos_cond} with points 
$$
z_l:=\frac{p}{2}\quad\text{and}\quad z_r:=\frac{p+1}{2},
$$
that is, $f(z_r)<z_l$, $f(z_l)>z_r$, and $f^2(z_l)>z_r$. Hence the map $f$ is Li-Yorke chaotic and has periodic orbit of any period (see Section~\ref{sec:discrete}), 
\end{comment}

In the proof of Proposition \ref{Thm2:step2} we determine the threshold $\step_p$ such that the conditions \ref{lem:chaos_cond1}, \ref{lem:chaos_cond2} and $f^2(z_l)>z_r$ of Proposition \ref{lem:chaos_cond} hold.
In order to simplify the calculations, we show the case of maximal possible perturbations (expressed by the values $\eta = \frac{4\eq}{(1-\eq)^2(\gainb-\gaind)}$ and $\xi = \frac{4}{\eq(\gainb-\gaind)}$). However, chaotic behavior of the family of maps $\updmap$ can also be observed for smaller values of the parameters $\eta$ and $\xi$.
From Proposition \ref{Thm2:step2} follows Theorem~\ref{thm:perturbedPPIchaos} for $p\in(0,\frac{1}{2})$.

\vspace{0.2cm}

\noindent{\bf Step 3. Proof of Theorem \ref{thm:perturbedPPIchaos} for $\widetilde{p}\in(\frac{1}{2},1)$.}
Using the topological conjugacy argument, we can extend the above result to any game $\widetilde{\mathcal{G}}\equiv\widetilde{\mathcal{G}}(\mathcal{A},\widetilde{\payv})$ defined by \eqref{eq:game} with the interior Nash equilibrium $\widetilde{\eq}\in(\frac{1}{2},1)$. Namely, then there exists a game $\game\equiv\gamefull$ with $p:=1-\widetilde{\eq}\in(0,\frac{1}{2})$ such that for any triplet $(\eta,\xi,\step)$, the condition $(\eta,\xi,\step)\in\Delta_p$ is equivalent to $(\xi,\eta,\step)\in\Delta_{\widetilde{p}}$. Moreover, with conditional imitation rates defined as
\begin{equation}\label{PPIiswitch-tilda}
\widetilde{\iswitch}_{AB}(\stratx):=\iswitch_{BA}(1-\stratx) \quad\text{and}\quad \widetilde{\iswitch}_{BA}(\stratx):=\iswitch_{AB}(1-\stratx),
\end{equation}
with $\iswitch_{AB}$ and $\iswitch_{BA}$ given by \eqref{eq:PPI_cond_imit_rates_pert}, we obtain the following %family of update 
maps of the unit interval introduced by the perturbed PPI protocol
\begin{equation}\label{thm2:step3:update_map}
	\widetilde{\updmap}(x)=
	\begin{cases}
		x(1+\step\eta\frac{b-d}{\widetilde{p}}(1-x)(\widetilde{\eq}-x)), &\text{for}\ x\in[0,\widetilde{\eq})\\
		x(1-\step\xi\frac{b-d}{\widetilde{p}}(1-x)(x-\widetilde{\eq})), &\text{for}\ x\in[\widetilde{\eq},1]\\
	\end{cases}.
\end{equation}

Then the following fact implies Theorem \ref{thm:perturbedPPIchaos} in the remaining case.

\begin{proposition}\label{Thm2:step3}
Let $\widetilde{p} \in (\frac 12, 1)$ be the Nash equilibrium of the game $\widetilde{\game}$ and let $\xi = \frac{4\widetilde{p}}{(1-\widetilde{p})^2(b-d)}$ and $\eta = \frac{4}{\widetilde{p}(b-d)}$. Then there exists $\step_{\widetilde{p}} \in (0,1)$ such that the map $\widetilde{\updmap}$ defined by \eqref{thm2:step3:update_map} is Li-Yorke chaotic and has periodic orbit of any period for each $\step\in(\step_{\widetilde{p}},1]$.
\end{proposition}

To prove Proposition \ref{Thm2:step3} we first show that under the assumption of the maximal values of the perturbations $\eta$ and $\xi$ the map $\widetilde{\updmap}$ in \eqref{thm2:step3:update_map} is an interval (self)map. Then, as dynamical systems $([0,1],\updmap)$ and $([0,1],\widetilde{\updmap})$ are topologically conjugate, by results for $\eq<\frac 12$ (showed in Step 2) Li-Yorke chaos for the map $\widetilde{\updmap}$ is implied. Thus, the proof of Theorem~\ref{thm:perturbedPPIchaos} is completed.

\subsection{Observability of chaos.}
%In particular, we will show that the interior Nash equilibrium $p\in(\frac{1}{3},\frac{1}{2})$ always becomes repelling for sufficiently large perturbations and time step size $\step$ (Corollary~\ref{cor:perturbed_PPI_repelling}). In the case of game $\game$ with $\eq\in(0,\frac{1}{3}]$, determining whether point $p$ is repelling or not requires a more elaborate analysis....
According to our interpretation, the parameters $(\eta,\xi) = \left( \frac{4p}{(1-p)^2(b-d)},\frac{4}{p(b-d)} \right)$ correspond to the situation in which agents attach the maximum possible importance to the difference in the payoffs; even a small value of $\pospart{\pay_{j}(\stratx) - \pay_{i}(\stratx)}$ can hugely influence the decision to change the strategy. From Theorem~\ref{thm:perturbedPPIchaos} it follows that in this case increasing $\step$ will eventually result in the emergence of chaos (see Figure \ref{fig:bif1}).
\begin{figure}[h]
    \centering
    \includegraphics[width=0.49\linewidth]{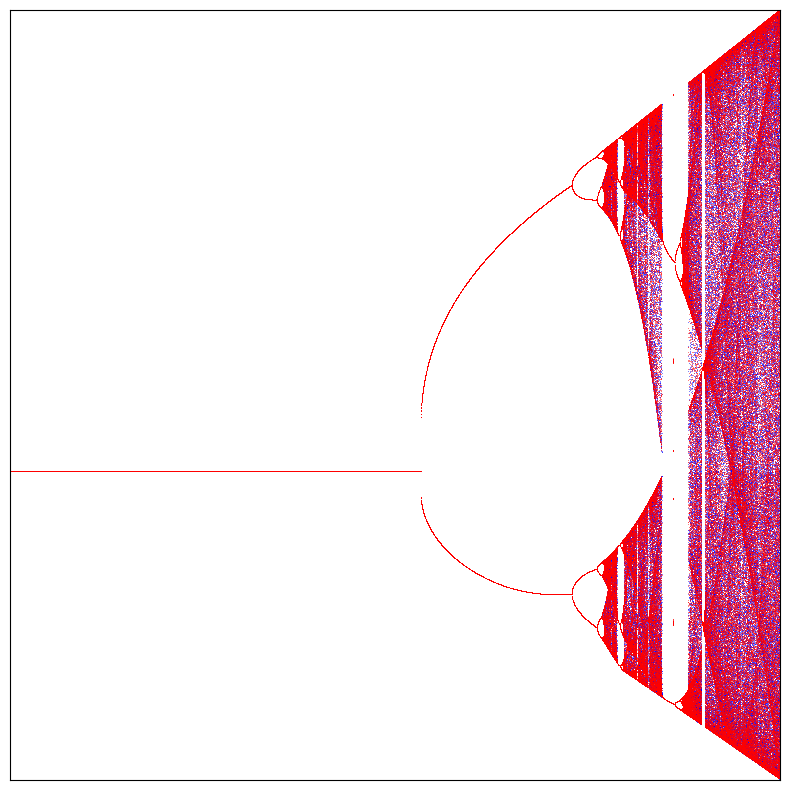}
    \includegraphics[width=0.49\linewidth]{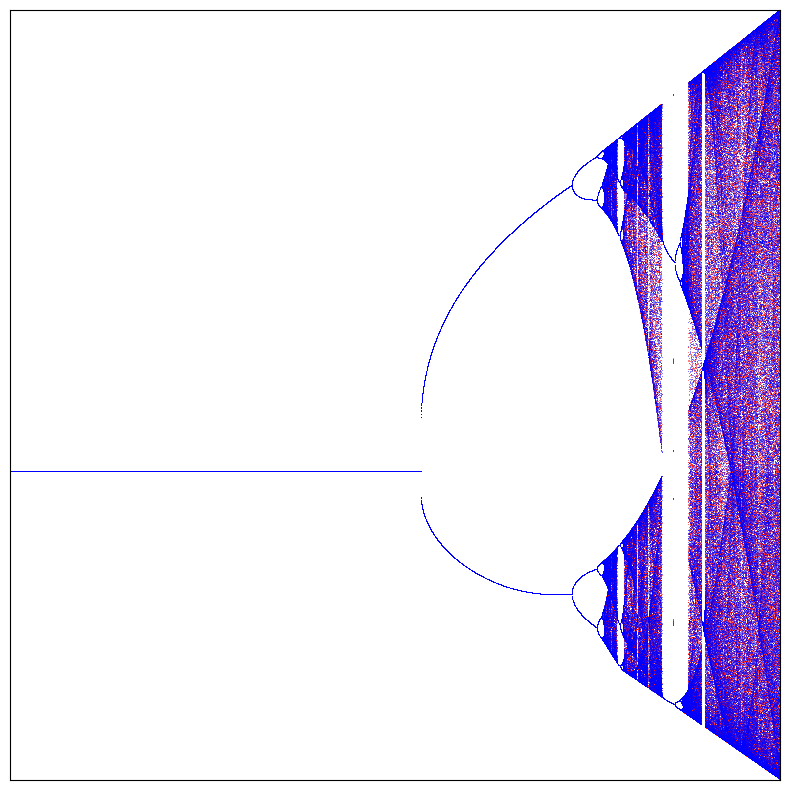}
    \caption{Bifurcation diagrams for the dynamical system of the map $\updmap$ in \eqref{eq:PPI_map_pert} for the maximal values of $\eta$ and $\xi$ and for $p=0.4$. The horizontal axis is the parameter $\delta \in (0,1]$. For each $\delta \in (0,1]$, 20000 iterations of the starting points were made and then next 100 iterations of the map $\updmap$ were plotted.
    %the vertical line shows last 100 iterations of the map $f_{\eta_{\max},\xi_{\max},\step}$ out of total 20100 {\color{red} Sprawdzcie czy sie zgadza} iterations.
On the left diagram first the iterations of the right critical point of $\updmap$ were plotted in \textcolor{blue}{blue}, then the iterations of the left critical point of $\updmap$ were plotted in \textcolor{red}{red}. On the right diagram the order of plotting is reversed. For small values of $\delta$ the fixed point $p$ attracts all trajectories of the dynamical system. As $\delta$ increases the point $p$ loses stability and we observe the period-doubling route to chaos.}
    \label{fig:bif1}
\end{figure}
Furthermore, the proof of this theorem indicates the threshold value of the parameter $\step$, above which chaos is guaranteed to occur. However, it does not provide any information about the stability of the point $\eq$, which may have a significant impact on the observability of the chaos (see Figure \ref{fig:bif2}). In this context, to complement the picture of the dynamics of $\updmap$ provided by Theorem~\ref{thm:perturbedPPIchaos} we present lemma on the behavior of the interior Nash equilibrium. % below include information that complements the picture of dynamics of $\updmap$ provided by .
\begin{figure}[h]
    \centering
    \includegraphics[width=0.5\linewidth]{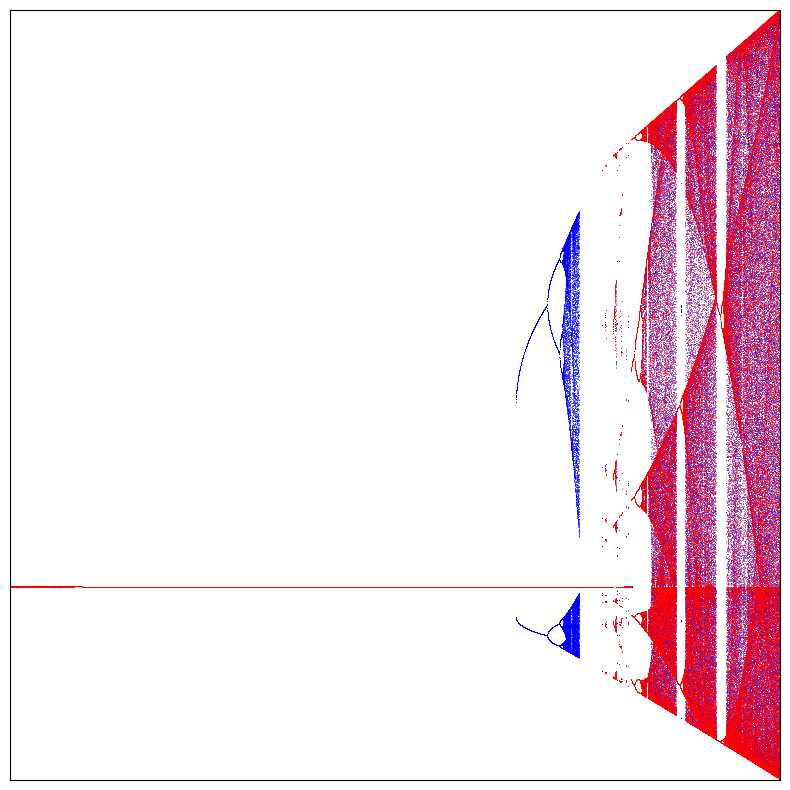}
    \caption{Bifurcation diagram for the dynamical system of the map $\updmap$ in \eqref{eq:PPI_map_pert} for the maximal values of $\eta$ and $\xi$ and for $p=0.25$. The horizontal axis is the parameter $\delta \in [0,1]$. For each $\delta \in [0,1]$, 20000 iterations of the starting points were made and then next 100 iterations of the map $\updmap$ were plotted.
    %the vertical line shows last 100 iterations of the map $f_{\eta_{\max},\xi_{\max},\step}$ out of total 20100 iterations.
    On the diagram first the iterations of the right critical point of $\updmap$ were plotted in \textcolor{blue}{blue}, then the iterations of the left critical point of $\updmap$ were plotted in \textcolor{red}{red}. For $\delta$ approximately in $(0.71799,0.73884)$ we can observe both: the locally attracting fixed point $p$ that attracts the trajectory of the left critical point and the chaotic attractor that attracts the trajectory of the right critical point.}
    \label{fig:bif2}
\end{figure}

\begin{lemma}\label{prop:dynamics_PPI}
Let $\game$ be a population game defined by~\eqref{eq:game} with the interior Nash equilibrium $p\in(0,\frac{1}{2})$
 and $\updmap$ be a map of the form~\eqref{eq:PPI_map_pert} with $(\eta,\xi) = \left( \frac{4p}{(1-p)^2(b-d)},\frac{4}{p(b-d)} \right)$. Let $\step\in(0,1]$. Then
	\begin{enumerate}[$(1)$]
        \item\label{prop:dynamics_PPI_0} Nash equilibrium $\eq\in (\frac 13,\frac 12)$ is repelling for sufficiently large $\step$;       
    %    $\updmap'_-(p)<-1$ if $\step$ is sufficiently large for any $\eq\in(0,\frac{1}{2})$,
		\item\label{prop:dynamics_PPI_1} for $\eq<\frac 13$ we have that $\updmap'_-(p)<-1$ if $\step$ is sufficiently large, however  $|\updmap'_+(p)|\leq 1$ for every $\step$.
		%\item\label{prop:dynamics_PPI_2} 
	\end{enumerate}
\end{lemma}

\begin{comment}

\begin{lemma}\label{prop:dynamics_PPI}
{\color{red}Przeformułować. Przekaz: dla $p\in(\frac{1}{3},\frac{1}{2})$ punkt stały zawsze staje się odpychający dla dużej delty, ale dla $p<1/3$ ciężko powiedzieć co się dzieje nawet dla największej delty.}
	Let $f:=f_{\eta,\xi,\step}$ be a map of the form~\eqref{eq:PPI_map_pert} defined by the parameters $\step\in(0,1]$ and $(\eta,\xi):=(N,M)$. Then the following conditions hold:
	\begin{enumerate}[$(1)$]
        \item\label{prop:dynamics_PPI_0} if $p\in(0,\frac{1}{2})$, then $\step_l:=\frac{p}{2(1-p)}<1$ and for every $\step\in(\step_l,1]$ we have $f'_-(p)<-1$;
		\item\label{prop:dynamics_PPI_1} if $p\in(\frac{1}{3},\frac{1}{2})$, then $\step_r:=\frac{1-p}{2p}<1$	and for every $\step\in(\step_r,1]$ we have $f'_+(p)<-1$;
		\item\label{prop:dynamics_PPI_2} if $p\in(0,\frac{1}{3}]$, then for every $\step\in(0,1]$ we have $|f'_+(p)|\leq1$.
	\end{enumerate}
\end{lemma}
\end{comment}

Lemma~\ref{prop:dynamics_PPI} describes stability of the fixed point of  $\updmap$ for the maximal perturbations (rescaling).
The condition~\ref{prop:dynamics_PPI_0} implies that for any $\eq \in(\frac 13,\frac{1}{2})$ 
increasing the parameter $\step$ will eventually lead to the loss of stability of the Nash equilibrium. By combining this observation with Theorem~\ref{thm:perturbedPPIchaos} and using conjugacy argument we obtain the following corollary. %, we obtain the following result. Note that using the topological conjugacy argument, we can extend this result to $\widetilde{\eq}\in(\frac{1}{2},\frac{2}{3})$ (cf. Step 3 of the proof of Theorem~\ref{thm:perturbedPPIchaos}).

\begin{corollary}\label{cor:perturbed_PPI_repelling}
Let $\eq\in(\frac{1}{3},\frac{2}{3})$. Then there exist $\xi$, $\eta$ and $\step_p \in (0,1)$ such that the map $\updmap$ defined by ~\eqref{eq:PPI_map_pert} %defined by the parameters $(\eta,\xi) = \left( \frac{4p}{(1-p)^2(b-d)},\frac{4}{p(b-d)} \right)$ and $\step\in(0,1]$. 
%Then there exists  such that
%$$
%\step>\left\lbrace\frac{1}{8}\left(p+3+ \sqrt{\frac{-p^3-4p^2-13p+34}{2-p}}\right),\frac{1-p}{2p}\right\rbrace,
%$$ the map $\updmap$ 
is Li-Yorke chaotic and has periodic orbit of any period for $\step>\step_{\eq}$. Moreover, the Nash equilibrium $\eq$ is repelling.
\end{corollary}
Observe that the condition \ref{prop:dynamics_PPI_1} doesn't imply neither convergence to the fixed point nor it precludes it. Thus, outside the interval $(\frac 13, \frac 23)$, determining whether the Nash equilibrium $p$ is repelling or not requires a more elaborate analysis. Moreover, it is needed since if the fixed point is attracting it may attract almost every trajectory even when the system is Li-Yorke chaotic (the scrambled set can be of measure zero). Thus, in the next section we will modify the revision protocols presented here to assure that the Nash equilibrium is repelling.

\section{Truncated dynamics}
\label{sec:thm3}
One can observe that for any game $\game\equiv\gamefull$, choosing sufficiently large $\eta$ and $\xi$ in \eqref{eq:PPI_map_pert} leads to the map $\updmap$ with the repelling point $\eq$. However, if $\eta$ or $\xi$ exceed  maximal values, the map $\updmap$ is no longer an interval (self)map. In order to increase the range of the parameter $\eta$ without violating the condition $\updmap([0,1])=[0,1]$, we modify the conditional imitation rate $\iswitch_{AB}$ to be constant in a neighborhood of $1$. A similar modification of $\iswitch_{BA}$ in the neighborhood of $0$ allows us to extend the range for the parameter $\xi$.

More precisely, for $\gamma\in[0,1]\backslash\{p\}$ we define the \textit{truncated} conditional imitation rates as 
\begin{equation}\label{eq:truncated_cond_imit_rate}
    		\iswitch_{AB}^{*}(\stratx):=\iswitch_{AB}(\stratx),\quad
        \iswitch_{BA}^{*}(\stratx):=
		\begin{cases}
			\iswitch_{BA}(\gamma), &\text{for}\ \stratx\in\left[0,\gamma\right]\\
			\iswitch_{BA}(\stratx), &\text{for}\ \stratx\in\left(\gamma,1\right]\\
		\end{cases},
\end{equation}
when $\gamma\in[0,\eq)$, and
\begin{equation}\label{eq:truncated_cond_imit_rate2}
            \iswitch_{BA}^{*}(\stratx):=\iswitch_{BA}(\stratx),\quad
    		\iswitch_{AB}^{*}(\stratx):=
		\begin{cases}
			\iswitch_{AB}(\stratx), &\text{for}\ \stratx\in\left[0,\gamma\right)\\
			\iswitch_{AB}(\gamma), &\text{for}\ \stratx\in\left[\gamma,1\right]\\
		\end{cases}
\end{equation}
when $\gamma\in(\eq,1]$. We exclude the point $p$ from the domain of $\gamma$, as otherwise, the limit case $\gamma=p$ would result in a piecewise constant conditional imitation rate, and therefore any increase of the positive payoffs difference would not change the behavior of the agents.

Observe that if $\gamma=0$ (or $\gamma=1$) then we get from \eqref{eq:truncated_cond_imit_rate} (or from \eqref{eq:truncated_cond_imit_rate2}) that $\iswitch_{AB}^{*}(\stratx)=\iswitch_{AB}(\stratx)$ and $\iswitch_{BA}^{*}(\stratx)=\iswitch_{BA}(\stratx)$ for all $x \in [0,1]$.

From the microeconomic point of view, the revision protocol described by \eqref{eq:truncated_cond_imit_rate2} expresses the following idea. When the payoff difference exceeds the limit value $\frac{\gainb-\gaind}{\eq}(\gamma-\eq)$ (which happens when the population share of $A$-strategists exceeds the level $\gamma$), a revising agent does not attach importance to whether it will be greater or not. From this point on any further increase in the payoffs difference no longer affects the decision to change the strategy. For \eqref{eq:truncated_cond_imit_rate} we observe the symmetric situation, when $B$-strategists {\cm fraction} exceeds $1-\gamma$.

With the modifications of conditional imitation rates %$\iswitch_{AB}^{*}$ and $\iswitch_{BA}^{*}$ into formula
we get from ~\eqref{eq:game_dynamics_imit} the family of maps:
\begin{comment}
\begin{equation}\label{eq:PPI_truncated_map}
	\updmap^{*}
    (\stratx)=
	\begin{cases}
		\stratx\left(1+\step\xi\frac{\gainb-\gaind}{\eq}(1-\stratx)(\eq-\gamma)\right), &\text{for}\ \stratx\in[0,\gamma)\\
		\stratx\left(1+\step\xi\frac{\gainb-\gaind}{\eq}(1-\stratx)(\eq-\stratx)\right), &\text{for}\ \stratx\in[\gamma,\eq)\\
		\stratx\left(1-\step\eta\frac{\gainb-\gaind}{\eq}(1-\stratx)(\stratx-\eq)\right), &\text{for}\ \stratx\in[\eq,1]\\
	\end{cases},
\end{equation}
\end{comment}
\begin{equation}\label{eq:PPI_truncated_map}
	\updmap^{*}
    (\stratx)=
	\begin{cases}
		\stratx\left(1+\step\xi\frac{\gainb-\gaind}{\eq}(1-\stratx)(\eq-\gamma)\right), &\text{for}\ \stratx\in[0,\gamma]\\
		F(x), &\text{for}\ \stratx\in(\gamma,1]
	\end{cases},
\end{equation}
if $\gamma\in[0,\eq)$, and
\begin{comment}
\begin{equation}\label{eq:PPI_truncated_map2}
	\updmap^{*}
    (\stratx)=
	\begin{cases}
		\stratx\left(1+\step\xi\frac{\gainb-\gaind}{\eq}(1-\stratx)(\eq-\stratx)\right), &\text{for}\ \stratx\in[0,\eq)\\
		\stratx\left(1-\step\eta\frac{\gainb-\gaind}{\eq}(1-\stratx)(\stratx-\eq)\right), &\text{for}\ \stratx\in[\eq,\gamma)\\
		\stratx\left(1-\step\eta\frac{\gainb-\gaind}{\eq}(1-\stratx)(\gamma-\eq)\right), &\text{for}\ \stratx\in[\gamma,1]\\
	\end{cases},
\end{equation}
\end{comment}
\begin{equation}\label{eq:PPI_truncated_map2}
	\updmap^{*}
    (\stratx)=
	\begin{cases}
		F(x), &\text{for}\ \stratx\in[0,\gamma)\\
		\stratx\left(1-\step\eta\frac{\gainb-\gaind}{\eq}(1-\stratx)(\gamma-\eq)\right), &\text{for}\ \stratx\in[\gamma,1]\\
	\end{cases},
\end{equation}
if $\gamma\in(\eq,1]$, where $F$ is given in \eqref{eq:PPI_map_pert} (see Figure \ref{fig:Fstar} to compare).
%Note that for $\gamma=\eq$ the resulting maps $F^{*}$ have the form \eqref{eq:PPI_map_pert}. 
Note that for $\gamma=0$ or $\gamma=1$ the resulting maps $F^{*}$ have the form \eqref{eq:PPI_map_pert}. 
Using truncated conditional imitation rates, we can get game dynamics with repelling Nash equilibrium for the full range of $p\in(0,1)$. Thus, $\updmap^*$ will be maps with anticipated properties. % to obtain a stronger result than in the case of \eqref{eq:PPI_cond_imit_rates_pert},  for the full range of $p\in(0,1)$

\begin{theorem} \label{thm:truncatedPPIchaos}
    For every population game $\game=\gamefull$ defined by \eqref{eq:game} there exist $\eta,\xi>0$ and $\gamma\in(0,1)\backslash\{p\}$ such that the imitative game dynamics %given by \eqref{eq:PPI_truncated_map} and \eqref{eq:PPI_truncated_map2}
    described by the map $\updmap^{*}$ is Li-Yorke chaotic, there exist periodic orbits of any period, and the interior Nash equilibrium $\eq$ is repelling for sufficiently large $\step$. 
\end{theorem}

 Thus, by using truncated imitation switch rates, the instability of the fixed point is restored (see e.g. Figure \ref{fig:Fstar}), and we exclude the possibility that one can observe convergence to Nash equilibrium on the set of positive Lebesgue measure. With the family of maps $F^*$ as the update map, the unpredictability of the system is guaranteed.

 \begin{figure}[!h]
\centering
\begin{subfigure}{0.48\textwidth}
\centering
	\begin{tikzpicture}[scale=0.68]
		\draw[thick,xscale=7,->] (0,0) -- (1.05,0);
		\draw[thick,->] (0,0) -- (0,6.35);
		%r_BA
		%(0,\eq)
		\draw[domain=0:1/4,smooth,variable=\x,orange,very thick,xscale=7] plot ({\x},{6*4*(1/4-\x)});
		%(\eq,1)
		\draw[domain=1/4:1,smooth,variable=\x,orange,very thick,xscale=7] plot 	({\x},{0});
		%r_AB
		%(0,\eq)
		\draw[domain=0:1/4,smooth,variable=\x,violet,very thick,xscale=7] plot ({\x},{0});
		%(\eq,1)
		\draw[domain=1/4:1,smooth,variable=\x,violet,very thick,xscale=7] plot 	({\x},{16/9*4*(\x-1/4)});
		%points
		\foreach \x in {0.25,0.5,0.75,1}
		\draw[xscale=7,shift={(\x,0)}] (0pt,2.5pt) -- (0pt,-2.5pt) node[below] {\footnotesize $\x$};
		\foreach \y in {0,1,2,...,5,6}
		\draw[shift={(0,\y)}] (2.5pt,0pt) -- (-2.5pt,0pt) node[left] {\footnotesize $\y$};
	\end{tikzpicture}
    \caption{Graph of {\color{violet}$\iswitch_{AB}$} and {\color{orange}$\iswitch_{BA}$}.}
\end{subfigure}
\hfill
\begin{subfigure}{0.48\textwidth}
\centering
	\begin{tikzpicture}[scale=4.3]
		\draw[thick] (0,0) -- (1,0);
		\draw[thick] (0,0) -- (0,1);
		\draw (0,1) node[left]{1};
		\draw (1,0) node[below]{1};
		\draw (-0.02,0) node[below]{0}; 
		\draw[thick,domain=0:1,variable=\y] plot ({1},{\y});
		\draw[thick,domain=0:1,variable=\x] plot ({\x},{1});
		%y=x
		\draw[domain=0:1,smooth,variable=\x,black,dashed] plot ({\x},{\x});
		%(0,\eq)
		\draw[domain=0:1/4,smooth,variable=\x,blue,very thick] plot ({\x},{\x*(1+1*6*4*(1-\x)*(1/4-\x)});
		%(\eq,1)
		\draw[domain=1/4:1,smooth,variable=\x,blue,very thick] plot ({\x},{\x*(1-1*16/9*4*(1-\x)*(\x-1/4)});
		%points
		\filldraw [red] (0.14,0) circle (0.2pt) node[anchor=north] {$\stratcl$};
		\filldraw [black] (0.25,0) circle (0.2pt) node[anchor=north] {$p$};
		\filldraw [red] (5/8,0) circle (0.2pt) node[anchor=north] {$\stratcr$};
	\end{tikzpicture}
    \caption{Graph of the  map $\updmap$.}
\end{subfigure}
\hfill
\begin{subfigure}{0.48\textwidth}
\centering
	\begin{tikzpicture}[scale=0.68]
		\draw[thick,xscale=7,->] (0,0) -- (1.05,0);
		\draw[thick,->] (0,0) -- (0,6.35);
		%r_BA
		%(0,\eq)
		\draw[domain=0:1/4,smooth,variable=\x,orange,very thick,xscale=7] plot ({\x},{6*4*(1/4-\x)});
		%(\eq,1)
		\draw[domain=1/4:1,smooth,variable=\x,orange,very thick,xscale=7] plot 	({\x},{0});
		%r_AB
		%(0,\eq)
		\draw[domain=0:1/4,smooth,variable=\x,violet,very thick,xscale=7] plot ({\x},{0});
		%(\eq,\gamma)
		\draw[domain=1/4:3/4,smooth,variable=\x,violet,very thick,xscale=7] plot 	({\x},{16/9*4*(\x-1/4)});
            %(\gamma,1)
		\draw[domain=3/4:1,smooth,variable=\x,violet,very thick,xscale=7] plot 	({\x},{16/9*4*(3/4-1/4)});
            %\draw[domain=1/4:9/32,smooth,variable=\x,violet,very thick,xscale=7] plot 	({\x},{256/23*4*(\x-1/4)});
            %\draw[domain=9/32:1,smooth,variable=\x,violet,very thick,xscale=7] plot 	({\x},{256/23*4*(9/32-1/4)});
		%points
		\foreach \x in {0.25,0.5,0.75,1}
		\draw[xscale=7,shift={(\x,0)}] (0pt,2.5pt) -- (0pt,-2.5pt) node[below] {\footnotesize $\x$};
		\foreach \y in {0,1,2,...,5,6}
		\draw[shift={(0,\y)}] (2.5pt,0pt) -- (-2.5pt,0pt) node[left] {\footnotesize $\y$};
	\end{tikzpicture}
    \caption{Graph of {\color{violet}$\iswitch^*_{AB}$} and {\color{orange}$\iswitch^*_{BA}$}.}
\end{subfigure}
\hfill
\begin{subfigure}{0.48\textwidth}
\centering
	\begin{tikzpicture}[scale=4.3]
		\draw[thick] (0,0) -- (1,0);
		\draw[thick] (0,0) -- (0,1);
		\draw (0,1) node[left]{1};
		\draw (1,0) node[below]{1};
		\draw (-0.02,0) node[below]{0}; 
		\draw[thick,domain=0:1,variable=\y] plot ({1},{\y});
		\draw[thick,domain=0:1,variable=\x] plot ({\x},{1});
		%y=x
		\draw[domain=0:1,smooth,variable=\x,black,dashed] plot ({\x},{\x});
		%(0,\eq)
		\draw[domain=0:1/4,smooth,variable=\x,blue,very thick] plot ({\x},{\x*(1+1*6*4*(1-\x)*(1/4-\x)});
		%(\eq,\gamma)
		\draw[domain=1/4:3/4,smooth,variable=\x,blue,very thick] plot ({\x},{\x*(1-1*16/9*4*(1-\x)*(\x-1/4)});
        %(\gamma,1)
		\draw[domain=3/4:1,smooth,variable=\x,blue,very thick] plot ({\x},{\x*(1-1*16/9*4*(1-\x)*(3/4-1/4)});
		%points
		\filldraw [red] (0.14,0) circle (0.2pt) node[anchor=north] {$\stratcl$};
		\filldraw [black] (0.25,0) circle (0.2pt) node[anchor=north] {$p$};
		\filldraw [red] (5/8,0) circle (0.2pt) node[anchor=north] {$\stratcr$};
	\end{tikzpicture}
    \caption{Graph of the  map $\updmap^*$.}
\end{subfigure}
\hfill
\begin{subfigure}{0.48\textwidth}
\centering
	\begin{tikzpicture}[scale=0.68]
		\draw[thick,xscale=7,->] (0,0) -- (1.05,0);
		\draw[thick,->] (0,0) -- (0,6.35);
		%r_BA
		%(0,\eq)
		\draw[domain=0:1/4,smooth,variable=\x,orange,very thick,xscale=7] plot ({\x},{6*4*(1/4-\x)});
		%(\eq,1)
		\draw[domain=1/4:1,smooth,variable=\x,orange,very thick,xscale=7] plot 	({\x},{0});
		%r_AB
		%(0,\eq)
		\draw[domain=0:1/4,smooth,variable=\x,violet,very thick,xscale=7] plot ({\x},{0});
		%(\eq,\gamma)
		%\draw[domain=1/4:3/4,smooth,variable=\x,violet,very thick,xscale=7] plot 	({\x},{16/9*4*(\x-1/4)});
            %(\gamma,1)
		%\draw[domain=3/4:1,smooth,variable=\x,violet,very thick,xscale=7] plot 	({\x},{16/9*4*(3/4-1/4)});
            \draw[domain=1/4:9/32,smooth,variable=\x,violet,very thick,xscale=7] plot 	({\x},{256/23*4*(\x-1/4)});
            \draw[domain=9/32:1,smooth,variable=\x,violet,very thick,xscale=7] plot 	({\x},{256/23*4*(9/32-1/4)});
		%points
		\foreach \x in {0.25,0.5,0.75,1}
		\draw[xscale=7,shift={(\x,0)}] (0pt,2.5pt) -- (0pt,-2.5pt) node[below] {\footnotesize $\x$};
		\foreach \y in {0,1,2,...,5,6}
		\draw[shift={(0,\y)}] (2.5pt,0pt) -- (-2.5pt,0pt) node[left] {\footnotesize $\y$};
	\end{tikzpicture}
    \caption{Graph of {\color{violet}$\iswitch^*_{AB}$} and {\color{orange}$\iswitch^*_{BA}$}.}
\end{subfigure}
\hfill
\begin{subfigure}{0.48\textwidth}
\centering
	\begin{tikzpicture}[scale=4.3]
		\draw[thick] (0,0) -- (1,0);
		\draw[thick] (0,0) -- (0,1);
		\draw (0,1) node[left]{1};
		\draw (1,0) node[below]{1};
		\draw (-0.02,0) node[below]{0}; 
		\draw[thick,domain=0:1,variable=\y] plot ({1},{\y});
		\draw[thick,domain=0:1,variable=\x] plot ({\x},{1});
		%y=x
		\draw[domain=0:1,smooth,variable=\x,black,dashed] plot ({\x},{\x});
		%(0,\eq)
		\draw[domain=0:1/4,smooth,variable=\x,blue,very thick] plot ({\x},{\x*(1+1*6*4*(1-\x)*(1/4-\x)});
		%(\eq,\gamma)
		\draw[domain=1/4:9/32,smooth,variable=\x,blue,very thick] plot ({\x},{\x*(1-1*256/23*4*(1-\x)*(\x-1/4)});
        %(\gamma,1)
		\draw[domain=9/32:1,smooth,variable=\x,blue,very thick] plot ({\x},{\x*(1-1*256/23*4*(1-\x)*(9/32-1/4)});
		%points
		\filldraw [red] (0.14,0) circle (0.2pt) node[anchor=north] {$\stratcl$};
		\filldraw [black] (0.25,0) circle (0.2pt) node[anchor=north] {\!\!$p$};
		\filldraw [red] (9/32,0) circle (0.2pt) node[anchor=north] {\,\,\,\,$\stratcr$};
	\end{tikzpicture}
    \caption{Graph of the  map $\updmap^*$.}
\end{subfigure}
\caption{Truncation of imitation revision rates of a game with Nash equilibrium $\eq=\frac 14$. [{\bf (a)} and {\bf (b)}] The perturbed imitation rates and the update map with $\xi=6$, $\step=1$ and maximal possible value of $\eta$ for \eqref{eq:PPI_map_pert} ($\eta=\frac{16}{9}$). [{\bf (c)} and {\bf (d)}] Truncation of {\bf (a)} at level $\gamma=\frac{3}{4}$ and the resulting update map \eqref{eq:PPI_truncated_map2}. [{\bf (e)} and {\bf (f)}] Truncation of {\bf (a)} at level $\gamma=p+\frac{p^2}{2}=\frac{9}{32}$ and the update map with $\eta$ increased to its maximal value for \eqref{eq:PPI_truncated_map2} ($\eta=\frac{256}{23}$). For the map $\updmap^*$ from {\bf (f)} Nash equilibrium $\eq$ is repelling.
}\label{fig:Fstar}
\end{figure}

\subsection{Skeleton of the proof of Theorem~\ref{thm:truncatedPPIchaos}}
Let $\game\equiv\gamefull$ be a game defined by \eqref{eq:game} with the interior Nash equilibrium $\eq$. %First, note that for $\eq=\frac12$ we put $\gamma:=\eq$. Then the maps $\updmap^{*}$ take form \eqref{eq:PPI_map_pert} and Theorem~\ref{thm:truncatedPPIchaos} follows from Corollary~\ref{cor:perturbed_PPI_repelling}.
First, note that for $\eq=\frac12$ we put $\gamma:=0$ (or $\gamma:=1$). Then the maps $\updmap^{*}$ take form \eqref{eq:PPI_map_pert} and Theorem~\ref{thm:truncatedPPIchaos} follows from Corollary~\ref{cor:perturbed_PPI_repelling}.
Thus, it remains to consider the cases $\eq\in(0,\frac12)$ and $\widetilde{\eq}\in(\frac12,1)$. The rest of the proof follows similar steps as the proof of Theorem~\ref{thm:perturbedPPIchaos}. All details, especially proofs of auxiliary results, can be found in Appendix~\ref{app:proof-thm3}.

\vspace{0.2cm}

\noindent{\bf Step 1. Setting parameters for which \eqref{eq:PPI_truncated_map} and \eqref{eq:PPI_truncated_map2} form a family of interval maps.}
The modifications \eqref{eq:truncated_cond_imit_rate} and \eqref{eq:truncated_cond_imit_rate2} of conditional imitation rates \eqref{eq:PPI_cond_imit_rates_pert} allow us to extend the set of parameters $\Delta_{\eq}$ for which the game dynamics is well-defined. Namely, we will show that for an appropriate choice of $\gamma$, the sets
%Since $\eq\in(0,\frac{1}{2})$, we have $\gamma\in(\eq,1)$. 
%When considering the family~\eqref{eq:PPI_truncated_map} we will still assume that $\xi\in(0,M]$ and $\step\in(0,1]$. We extend the set of parameters $\Delta_p$ to
%$$
%\Delta_p^*:=(0,N^*]\times(0,M]\times(0,1],
%$$
%where
%$$
%N^*:=\frac{4}{p^2(2-2p-p^2)}\geq\frac{4}{(1-p)^2}=N.
%$$
%(see Corollary~\ref{cor:f_maps_truncated}). 
\[
 \Delta^{*}_p := \left\{ (\nu_1,\nu_2,\nu_3) \in (0,\infty)^3 : \nu_1 \leqslant \frac{4}{\eq(2-2\eq-\eq^2)(\gainb-\gaind)}, \nu_2 \leqslant \frac{4}{\eq(\gainb-\gaind)} \text{ and } \nu_3 \in (0,1] \right\}
\]
and
\[
\begin{aligned}
 \Gamma^{*}_p := \left\{ (\nu_1,\nu_2,\nu_3) \in (0,\infty)^3 :\!\!\!\!\!\!\!\!\phantom{\frac{1}{1^1}}\right. &\nu_1 \leqslant \frac{4\eq}{(1-\eq)^2(\gainb-\gaind)},\\
 &\!\left.\nu_2 \leqslant \frac{4\eq}{(1-\eq)^2(-1+4\eq-\eq^2)(\gainb-\gaind)} \text{ and } \nu_3 \in (0,1] \right\}
 \end{aligned}
\]
describe the parameters $(\eta,\xi,\step)$, for which the pair $([0,1],\updmap^{*})$ forms a dynamical system. We will use the set $\Delta^{*}_p$ in all three steps of the proof, while the set $\Gamma^{*}_p$ will be needed in Step 3.

In this step, we focus on the case $\eq \in (0,\frac 12)$. For $\widetilde{\eq}\in(\frac{1}{2},1)$, we will again use the topological conjugacy argument (see Step 3).

%For our analysis we set $\gamma:=\eq+\frac{\eq^2}{2}$.

\begin{lemma}\label{cor:f_maps_truncated}
Let $\eq\in(0,\frac{1}{2})$ be the Nash equilibrium of the game $\game$.
%Let
%\[
% \Delta_p := \left\{ (\eta,\xi,\step) : \eta \leqslant \frac{4p}{(1-p)^2(b-d)}, \xi \leqslant \frac{4}{p(b-d)} \text{ and } \delta \in (0,1] \right\}.
%\]
If $\gamma=\eq+\frac{\eq^2}{2}$, then for any parameters $(\eta,\xi,\step)\in\Delta^{*}_p$ we have $\updmap^{*}([0,1])=[0,1]$. In particular, $\updmap^{*}$ is an interval map.
\end{lemma}

The proof of Lemma \ref{cor:f_maps_truncated} proceeds similarly to that of Lemma \ref{cor:PPI_interval_map}: first we prove the claim for the maximal possible values of $\eta$, $\xi $, $\step$ and then extend it to all parameters from the set $\Delta^{*}_p$.\footnote{See Lemma \ref{prop:f_max_truncated} in Appendix \ref{app:proof-thm3} for other useful properties of the family of maps \eqref{eq:PPI_truncated_map2}.}

\vspace{0.2cm}

\noindent{\bf Step 2. Proof of Theorem \ref{thm:truncatedPPIchaos} for $\eq\in(0,\frac{1}{2})$.}
Using Proposition \ref{lem:chaos_cond}, we show that for any $\eq\in(0,\frac12)$ and an appropriate choice of parameters $\gamma$, $\eta$ and $\xi$ in \eqref{eq:PPI_truncated_map2} there is a threshold value $\step_{\eq}$ above which the game dynamics is guaranteed to be chaotic. Furthermore, the increased range of the parameter $\eta$ allows us to choose $\step_{\eq}$ in such a way that the repelling of the Nash equilibrium $\eq$ is also ensured.
%In this step, we obtain the result in the spirit of Proposition~\ref{Thm2:step2}.
\begin{proposition}\label{Thm3:step2}
Let $\eq\in(0,\frac{1}{2})$ be the Nash equilibrium of the game $\game$ and let $\gamma=\eq+\frac{\eq^2}{2}$, $\eta = \frac{4}{\eq(2-2\eq-\eq^2)(\gainb-\gaind)}$ and $\xi = \frac{4}{\eq(\gainb-\gaind)}$. Then there exists $\step_{\eq} \in (0,1)$ such that the map $\updmap^{*}$ is Li-Yorke chaotic, has periodic orbit of any period and the point $\eq$ is repelling for each $\step\in(\step_{\eq},1]$.
\end{proposition}
The above result completes the proof of Theorem~\ref{thm:truncatedPPIchaos} for $\eq\in(0,\frac{1}{2})$.

%\begin{corollary}\label{cor:f_maps_truncated}
%	For any interior Nash equilibrium $p\in(0,\frac{1}{2})$ and parameters $(\eta,\xi,\step)\in\Delta_p^*$ we have $f^*_{\eta,\xi,\step}([0,1])=[0,1]$. In particular, $f^*_{\eta,\xi,\step}$ is an interval map.
%\end{corollary}

\vspace{0.2cm}

\noindent{\bf Step 3. Proof of Theorem \ref{thm:truncatedPPIchaos} for $\widetilde{\eq}\in(\frac{1}{2},1)$.}
Consider a game $\widetilde{\mathcal{G}}\equiv\widetilde{\mathcal{G}}(\mathcal{A},\widetilde{\payv})$ defined by \eqref{eq:game} with the interior Nash equilibrium $\widetilde{\eq}\in(\frac{1}{2},1)$. Denote $\eq:=1-\widetilde{\eq}\in(0,\frac{1}{2})$ and $\gamma:=\eq+\frac{\eq^2}{2}$. Then there exists a game $\game\equiv\gamefull$ with the interior Nash equilibrium $\eq$ such that for $\widetilde{\gamma}:=1-\gamma$ and any triplet $(\eta,\xi,\step)$, the condition $(\eta,\xi,\step)\in\Delta^{*}_p$ is equivalent to $(\xi,\eta,\step)\in\Gamma^{*}_{\widetilde{p}}$. Moreover, the conditional imitation rates defined as
\begin{equation}\label{eq:thm3_conjugacy}
\widetilde{\iswitch}^{\,*}_{AB}(\stratx):=\iswitch_{BA}^{*}(1-\stratx) \quad\text{and}\quad \widetilde{\iswitch}^{\,*}_{BA}(\stratx):=\iswitch^{*}_{AB}(1-\stratx),
\end{equation}
where $\iswitch^{*}_{AB}$ and $\iswitch^{*}_{BA}$ are given by \eqref{eq:truncated_cond_imit_rate2}, lead to the maps of following form (cf. \eqref{eq:PPI_truncated_map}):

\begin{equation}\label{thm3:step3:update_map}
	\widetilde{\updmap}^{*}
    (\stratx)=
	\begin{cases}
		\stratx\left(1+\step\eta\frac{\gainb-\gaind}{\widetilde{\eq}}(1-\stratx)(\widetilde{\eq}-\widetilde{\gamma})\right), &\text{for}\ \stratx\in[0,\widetilde{\gamma}]\\
		\widetilde{\updmap}(x), &\text{for}\ \stratx\in(\gamma,1]
	\end{cases},
\end{equation}
where $\widetilde{\updmap}$ is given in \eqref{thm2:step3:update_map}.
\begin{proposition}\label{Thm3:step3}
Let $\widetilde{\eq} \in (\frac 12, 1)$ be the Nash equilibrium of the game $\widetilde{\game}$ and let $\widetilde{\gamma}=\widetilde{\eq}-\frac{(1-\widetilde{\eq})^2}{2}$, $\eta=\frac{4\widetilde{\eq}}{(1-\widetilde{\eq})^2(-1+4\widetilde{\eq}-\widetilde{\eq}^2)(\gainb-\gaind)}$ and $\xi = \frac{4\widetilde{\eq}}{(1-\widetilde{\eq})^2(\gainb-\gaind)}$. Then there exists $\step_{\widetilde{\eq}} \in (0,1)$ such that the map $\widetilde{\updmap}^{*}$ defined by \eqref{thm3:step3:update_map} is Li-Yorke chaotic, has periodic orbit of any period and the point $\widetilde{\eq}$ is repelling for each $\step\in(\step_{\widetilde{\eq}},1]$.
\end{proposition}

The proof of Proposition~\ref{Thm3:step3} is based on the fact that there is a topological conjugacy between the dynamical systems formed from maps \eqref{eq:PPI_truncated_map} and those given by \eqref{eq:PPI_truncated_map2}. This relation is established by~\eqref{eq:thm3_conjugacy}. Thus, from the previously obtained results for $\eq<\frac12$, Proposition~\ref{Thm3:step3} follows and the proof of Theorem~\ref{thm:truncatedPPIchaos} is completed.

\section{Conclusions and future work}
\label{sec:concl}
In this work we show that the microeconomic idea of a revision protocol exhibits complex
behavior in discrete time even in the simplest of games --- a $2\times 2$ symmetric random matching
anti-coordination game. This is true both for imitative and innovative protocols.  In fact, chaotic behavior, when the length of the unit time interval between periods in a discrete time horizon is large, is encoded into the imitative dynamics. Moreover, we show that revision protocols for which complex behavior can be seen is not far from the famous pairwise proportional imitation model. Thus, one should expect complex behavior of discrete time revision driven game dynamics even in very simple evolutionary game-theoretic settings.
These findings show the possibility that the {\it invisible hand of the market}  can be in deep delirium and produce senseless, non-optimal outcomes. 
There is an established work, see e.g. \cite{burke2014entry}, which describes equilibrium overshoot in markets due to the fact that institutional actors such as corporations can change their policies according to the yearly reports, thus they are correlated and can act simultaneously at the end of the year. Our results point to another microeconomic source of possible instability of the population dynamics: instability may be a direct consequence of comparisons. 
These findings are in line with recent results from multiagent reinforcement learning problems, where complex behavior can be observed also in simple games \cite{2017arXiv170109043P,BCFKMP21,bielawski2024memory}. %{\color{red} [zacytowac nasze PNAS?]}
 
Finally, we are convinced that our results are robust. We already know that if one considers
more strategies, %one can detect 
the chaotic behavior of agents following the multiplicative weights update algorithm
(a learning model merged in continuous time with PPI by the replicator equation) can be detected \cite{CFMP2019}. An open
question to follow is whether similar results can be obtained for coordination games when agents are prone to some errors  and fixed points of the dynamics are no longer Nash equilibria but quantal response equilibria? 
Moreover, one may ask if it is possible to get chaotic behavior for other well-known innovative dynamics like pairwise comparison (Smith) dynamics introduced by \eqref{eq:Smithprot} or Brown-von Neumann-Nash dynamics \cite{San10,brown1950solutions} or even best response dynamics. This would probably require a more complex structure of the game.

\vskip 0.3in
{\bf Acknowledgements.} Jakub Bielawski, Łukasz Cholewa and Fryderyk Falniowski greatfully acknowledge the support from National Science Centre, Poland, grant nr. 2023/51/B/HS4/01343.
\vskip 0.2in
{\bf Funding.} National Science Centre, Poland, grant nr. 2023/51/B/HS4/01343.
\vskip 0.2in
{\bf Data Availability.} No data is associated with the manuscript.
\vskip 0.2in
{\bf Conflict of interest.} The authors have no conflict of interest to declare.

\bibliographystyle{abbrvnat} 
\bibliography{ms,Bibliography,IEEEabrv}

\section*{Appendix}
\label{sec:appendix}
\appendix

%%%%%%%%%%%%%%%%%%%%%%%%%%%%%%%%%%%%%%%%%%%%%%%%%%%%%%%%%%%%%%%%%%%%%%%%%%%%%%%%

\section{Proof of Theorem \ref{thm:anticoord-chaos}}\label{app:proof-thm1}

\begin{proof}[Proof of Lemma \ref{cor:lin_bimodal1}]
First, observe that if a continuous bimodal map $\bimap \colon[0,1]\to[0,1]$ with critical points $\stratcl<\stratcr$, satisfies conditions:
\begin{multicols}{3}
	\begin{enumerate}[$(1'')$]
		\item\label{prop:cond1} $\bimap(\stratcr)>\stratcl$,
		\item\label{prop:cond2} $\bimap(\stratcl)=0$,
		\item\label{prop:cond3} $\bimap(0)>\stratcr$,
	\end{enumerate}
    \end{multicols}
\noindent then, by Proposition \ref{lem:chaos_cond}, $\bimap(x)<x<\bimap^3(x)$ for some point $x\in(\stratcl,\stratcr)$.
%Note that, this observation is a direct consequence of Proposition \ref{lem:chaos_cond}. 
We will show that the map $f$ in \eqref{eq:f_form} %with the switch rates $\switch_{AB}(x)$ and $\switch_{BA}(x)$ given in \eqref{rho12inn} and \eqref{rho21inn}, respectively, 
satisfies the conditions \ref{prop:cond1}-\ref{prop:cond3} above.

Observe that a map $\bimap$ of the form \eqref{PLBmap} is a continuous self-map of the interval $[0,1]$ if and only if it satisfies the following conditions:
\begin{equation}\label{eq:f_conditions}
	\begin{aligned}
		\bimap(0)&=\alpha_1\in[0,1],\;\qquad \bimap(\stratcl)=\beta_1\stratcl+\alpha_1=\beta_2\stratcl+\alpha_2\in[0,1],\\
		\bimap(\stratcr)&=\beta_2\stratcr+\alpha_2=\beta_3\stratcr+\alpha_3\in[0,1],\;\qquad \bimap(1)=\beta_3+\alpha_3\in[0,1].
	\end{aligned}
\end{equation}
Hence
\[
\alpha_2=-\beta_2\stratcl+\beta_1\stratcl+\alpha_1 \;\; \text{and} \;\; \alpha_3=-\beta_3\stratcr+\beta_2\stratcr-\beta_2\stratcl+\beta_1\stratcl+\alpha_1,
\]
which allows us to eliminate parameters $\alpha_2$ and $\alpha_3$. So in fact the map $\bimap$ is of the form
\begin{equation*}
	\bimap(x)=\begin{cases}
		\beta_1x+\alpha_1, &\text{for}\ x\in[0,\stratcl),\\
		\beta_2(x-\stratcl)+\beta_1\stratcl+\alpha_1, &\text{for}\ x\in [\stratcl,\stratcr),\\
		\beta_3(x-\stratcr)+\beta_2(\stratcr-\stratcl)+\beta_1\stratcl+\alpha_1, &\text{for}\ x\in [\stratcr,1].
	\end{cases}
\end{equation*}
Now observe that
$\bimap(\stratcl)=0$  if and only if $\alpha_1=-\beta_1\stratcl$,
which simplify the form of map $\bimap$ to the form \eqref{eq:f_form}.
Furthermore, the following equivalences hold
$$
\bimap(0)=-\beta_1\stratcl\in(\stratcr,1]\iff\beta_1\in B_1:=\left[ -\frac{1}{\stratcl},-\frac{\stratcr}{\stratcl}\right)
$$
and
$$
\bimap(\stratcr)=\beta_2(\stratcr-\stratcl)\in(\stratcl,1]\iff\beta_2\in B_2:=\left(\frac{\stratcl}{\stratcr-\stratcl},\frac{1}{\stratcr-\stratcl}\right].
$$
Let $\beta_2\in B_2$. Then we obtain
$$
\beta_3(1-\stratcr)+\stratcl<\bimap(1)=\beta_3(1-\stratcr)+\beta_2(\stratcr-\stratcl)\leq\beta_3(1-\stratcr)+1<1.
$$
So for any 
$$
\beta_3\in B_3:=\left[ -\frac{\stratcl}{1-\stratcr},0\right) 
$$
we have $\bimap(1)\in[0,1]$.

From the above considerations it follows that for any parameters $(\beta_1,\beta_2,\beta_3)\in B_1\times B_2\times B_3$ a map of the form~\eqref{eq:f_form} satisfies the conditions~\eqref{eq:f_conditions} as well as \ref{prop:cond1}-\ref{prop:cond3} from the beginning of the proof.

%\ref{prop:cond1}--\ref{prop:cond3} from  Corollary~\ref{prop:cond123}. 
\end{proof}

\begin{proof}[Proof of Proposition \ref{prop:inn_switch_rates}]
Note that for any bimodal map $\bimap$ of the form~\eqref{eq:f_form} we have $\bimap(0)=-\beta_1\stratcl>0$ and $\bimap(\stratcl)=0$, so $\bimap$ has a fixed point 
$$
x^*:=\frac{\beta_1\stratcl}{\beta_1-1}\in(0,\stratcl).
$$
Moreover, $x^*$ is a unique fixed point of the map $\bimap$ if and only if 
$\bimap(\stratcr)=\beta_2(\stratcr-\stratcl)<\stratcr$,
i.e. $\beta_2\in(0,\frac{\stratcr}{\stratcr-\stratcl})$. % (see Figure~\ref{fig:lin_bimodal}). 
With $\stratcl=\frac{\eq}{1-\eq}$ and $\stratcr=2\eq$, and since $p\in(0,\frac{1}{2})$, we obtain that $\eq$ belongs to the first lap, that is $0<p<\stratcl<\stratcr<1$.

Next, observe that
$$
x^*=p\iff (\beta_1-1)p=\beta_1\stratcl\iff\beta_1=\frac{p-1}{p}
$$
and
\[
\frac{p-1}{p}\in\left[-\frac{1}{\stratcl},-\frac{\stratcr}{\stratcl}\right)=\left[\frac{p-1}{p},\,2(p-1)\right).
\]
Therefore, the map 
\begin{equation}\label{eq:map_f_form}
	\updmap(x)=
	\begin{cases}
		\frac{p-1}{p}\, x+1, &\text{for}\ x\in[0,\frac{p}{1-p}),\\
		\beta_2(x-\frac{p}{1-p}), &\text{for}\ x\in [\frac{p}{1-p},2p),\\
		\beta_3(x-2p)+\beta_2\frac{p(1-2p)}{1-p}, &\text{for}\ x\in [2p,1],
	\end{cases}
\end{equation}
defined by the parameters $\beta_2 \in \left(\frac{1}{1-2p},\frac{2(1-p)}{1-2p}\right)$ and $\beta_3 \in \left[ -\frac{p}{(1-2p)(1-p)},0\right)$,
has the form \eqref{eq:f_form} from Lemma \ref{cor:lin_bimodal1}.
Moreover, as $\eq<1/2$, for every point $x\in(0,\frac{p}{1-p})$ we have $\updmap'(x)=\frac{p-1}{p}<-1$.
Therefore, the fixed point $p$ is repelling.
Since $u_A(x)>u_B(x)$ if and only if $x<p$, we obtain that
\begin{equation*}
\begin{aligned}
 \switch_{AB}(x) = 1-\frac{\updmap(x)}{x}
 %&= \begin{cases}
%	1-\frac{p-1}{p}\left(1-\frac{\stratcl}{x} \right)   , &\text{for}\ x\in[p,\stratcl),\\
%	1-\beta_2\left(1-\frac{\stratcl}{x} \right)   , &\text{for}\ x\in[\stratcl,\stratcr),\\
%	1-\beta_3\left(1-\frac{\stratcr}{x} \right)-\frac{\beta_2}{x}(\stratcr-\stratcl)   , &\text{for}\ x\in[\stratcr,1],\\
%\end{cases}\\
&=
\begin{cases}
	\frac{1}{p}-\frac{1}{x}  , &\text{for}\ x\in[p,\frac{p}{1-p}),\\
	1-\beta_2\left(1-\frac{p}{(1-p)x} \right)   , &\text{for}\ x\in[\frac{p}{1-p},2p),\\
	1-\beta_3\left(1-\frac{2p}{x} \right)-\beta_2\frac{b(1-2b)}{(1-b)x}  , &\text{for}\ x\in[2p,1],\\
\end{cases}\\
\end{aligned}
\end{equation*}
 for $x\in[p,1]$, and $\switch_{AB}(x) = 0$ for $x\in[0,p)$, while the conditional switch rate $\switch_{BA}$ is given by
\begin{equation*}
 \switch_{BA}(x) = \frac{\updmap(x)-x}{1-x} %= \frac{\frac{p-1}{p}(x-\stratcl)-x}{1-x} 
 = \frac{p-x}{p(1-x)}
\end{equation*}
 for $x\in[0,p)$, and $\switch_{BA}(x) = 0$ for $x\in[p,1]$.

Since the map $\switch_{AB}$ is continuous on $[0,1]$ and has bounded derivative on each of intervals $(0,p)$, $(p,\stratcl)$, $(\stratcl,\stratcr)$ and $(\stratcr,1)$, it is Lipschitz continuous on $[0,1]$. By similar argument $\switch_{BA}$ is Lipschitz continuous on $[0,1]$ as well.\footnote{Figure~\ref{fig:switch_rates} shows the graphs of $\switch_{AB}$ in \eqref{rho12inn} and $\switch_{BA}$ in \eqref{rho21inn} for example parameters $\eq=0{.}2$, $\beta_2 = 2 \in \left(\frac{1}{1-2\eq},\frac{2(1-\eq)}{1-2\eq}\right) = \left(\frac{5}{3},\frac{8}{3}\right)$ and $\beta_3 = -\frac{1}{3} \in \left[ -\frac{\eq}{(1-2\eq)(1-\eq)},0\right) = \left[ -\frac{5}{12},0\right)$.} 

Finally, recall that $u_A(x)>u_B(x)$ if and only if $x<p$. So for every point $x\in[0,1]$ the switch rates $\switch_{AB}$ in \eqref{rho12inn} and $\switch_{BA}$ in \eqref{rho21inn} satisfy conditions
\begin{equation}\label{eq:switch_rates_sgn}
\begin{aligned}
\sign\left( \switch_{BA}(x)\right)&=\sign\left(\left[u_A(x)-u_B(x) \right]_+  \right),\\
\sign\left( \switch_{AB}(x)\right)&=\sign\left(\left[u_B(x)-u_A(x) \right]_+  \right), 
\end{aligned}
\end{equation}
where $[d]_+=\max\{0,d\}$. This implies that the maps $\switch_{AB}$ and $\switch_{BA}$ indeed form an innovative revision protocol of innovative dynamics. The proof of Proposition~\ref{prop:inn_switch_rates} is completed. 
\end{proof}

\begin{proof}[Proof of Proposition \ref{prop:chaos_inn_p>}]
%    \begin{proof}[Proof of Lemma \ref{lem:symmetry}]
%   We begin with the proof for innovative revision protocols. 
% It is sufficient to 
 
 We first show that if \eqref{def_tilda_switch} holds then the dynamical systems $([0,1],\updmap)$  and $([0,1],\widetilde{\updmap})$, defined by \eqref{eq:game_dynamics-PC} for $\game$ and $\widetilde{\game}$ respectively, are topologically conjugate by the homeomorphism $\varphi(\stratx)=1-\stratx$.
 So assume that \eqref{def_tilda_switch} holds. Then
	\begin{equation}\label{fiffi}
		\begin{aligned}
			\left(\varphi^{-1}\circ \updmap\circ \varphi\right)(\stratx)&=1-\updmap(1-\stratx)\\
			&=
			\begin{cases}
				1-\left(1-\stratx+\step \stratx\switch_{BA}(1-\stratx)\right) , &\text{ for } 1-\stratx\in[0,\eq),\\
				1-\left(1-\stratx\right)\left(1-\step\switch_{AB}(1-\stratx)\right) , &\text{ for } 1-\stratx\in[\eq,1],\\
			\end{cases}\\
			&=
			\begin{cases}
				\stratx\left(1-\step\widetilde{\switch}_{AB}(\stratx)\right) , &\text{ for } \stratx\in(\widetilde{\eq},1],\\
				\stratx+\step\left(1-\stratx\right)\widetilde{\switch}_{BA}(\stratx) , & \text{ for }\stratx\in[0,\widetilde{\eq}],\\
			\end{cases}\\
			&=\widetilde{\updmap}(\stratx).
		\end{aligned}
	\end{equation}
Therefore, the map $\varphi(\stratx)=1-\stratx$ is a topological conjugacy between the dynamical systems $([0,1],\updmap)$ and $([0,1],\widetilde{\updmap})$.

Note that the maps $\widetilde{\switch}_{AB}$ and $\widetilde{\switch}_{BA}$ in \eqref{def_tilda_switch} are Lipschitz continuous. From \eqref{eq:switch_rates_sgn} and the fact that 

\begin{equation}\label{eq:payoff_vector_equiv}
\widetilde{\pay}_{1}(\stratx)>\widetilde{\pay}_{2}(\stratx)\iff \stratx<\widetilde{\eq}\iff1-\stratx>\eq\iff\pay_{2}(1-\stratx)>\pay_{1}(1-\stratx),
\end{equation}
we obtain conditions
$$
\begin{aligned}
	\sign\left( \widetilde{\switch}_{BA}(x)\right)&=\sign\left( \switch_{AB}(1-x)\right)=\sign\left(\left[u_B(1-x)-u_A(1-x) \right]_+  \right)=\sign\left(\left[\widetilde{u}_A(x)-\widetilde{u}_B(x) \right]_+  \right),\\
	\sign\left( \widetilde{\switch}_{AB}(x)\right)&=\sign\left( \switch_{BA}(1-x)\right)=\sign\left(\left[u_A(1-x)-u_B(1-x) \right]_+  \right)=\sign\left(\left[\widetilde{u}_B(x)-\widetilde{u}_A(x) \right]_+  \right), 
\end{aligned}
$$
so $\widetilde{\switch}_{AB}$ and $\widetilde{\switch}_{BA}$ are well-defined switch rates of an innovative revision protocol.
\begin{comment}
. By \eqref{fiffi} the map
$$
\widetilde{f}(x)=\begin{cases}
	x+\delta(1-x)\widetilde{\switch}_{BA}(x), &x\in[0,\widetilde{p}),\\
	x(1-\delta\widetilde{\switch}_{AB}(x)), &x\in[\widetilde{p},1],
\end{cases}
$$
is topologically conjugate to $f$, i.e.
$$
\widetilde{f}(x)=\left(\varphi^{-1}\circ f\circ \varphi\right)(x),
$$
where the map $\varphi\colon[0,1]\to[0,1]$ is given by $\varphi(x)=1-x$
\end{comment}
In particular, the map $\widetilde{\updmap}$ is a piecewise linear bimodal map with local minimum at $\stratclalt:=\varphi(\stratcr)$ and local maximum at $\stratcralt:=\varphi(\stratcl)$, and is also Li-Yorke chaotic. Moreover, by topological conjugacy argument $\widetilde{p}$ is the unique fixed point of $\widetilde{\updmap}$ and that $\widetilde{p}$ is repelling (see Figure~\ref{fig:g_and_gtilde}).
%(see Figure~\ref{fig:g_and_gtilde}).
\end{proof}

 %The proof of the Theorem~\ref{thm:anticoord-chaos} is completed. %[{\color{red}except for $p=\frac{1}{2}$}{\color{red}{. Dla $p=1/2$ odwzorowanie musi mieć inną postać}. Albo wykorzystujemy wczesniej znane fakty}]

%%%%%%%%%%%%%%%%%%%%%%%%%%%%%%%%%%%%%%%%%%%%%%%%%%%%%%%%%%%%%%%%%%%%%%%%%%%%%%%%

\section{Proof of Theorem \ref{thm:anticoord-chaos} for imitative revision protocols}\label{app:imitative}

%\noindent {\bf Proof of Theorem \ref{thm:anticoord-chaos} for imitative revision protocols.}

For the proofs of the results describing the behavior of the agents following an imitative revision protocol we will need the following lemma. It allows us to bind the cases when the Nash equilibrium is $p \in (0, \frac 12)$ and when it is $\widetilde{p} \in (\frac 12,1)$.

\begin{lemma}\label{lem:symmetry}
Let $\updmap$ and $\widetilde{\updmap}$ be the maps of form~\eqref{eq:game_dynamics_imit} corresponding to conditional imitation rates $\iswitch_{BA}$, $\iswitch_{AB}$ and $\widetilde{\iswitch}_{BA}$, $\widetilde{\iswitch}_{AB}$ that satisfy the following condition:
    \begin{equation}\label{eq:imit_symmetry}
	\widetilde{\iswitch}_{BA}(\stratx)-\widetilde{\iswitch}_{AB}(\stratx)=\iswitch_{AB}(1-\stratx)-\iswitch_{BA}(1-\stratx),
	\end{equation}
    %\end{itemize}
	for every point $\stratx\in[0,1]$. Assume that $\updmap$ is a self-map of the interval $[0,1]$. Then $\widetilde{\updmap}$ is also a self-map of $[0,1]$. Moreover, the dynamical systems $([0,1],\updmap)$ and $([0,1],\widetilde{\updmap})$ are topologically conjugate by the homeomorphism $\varphi(x)=1-x$.
\end{lemma}

\begin{proof}
Assume that the conditional imitation rates satisfy the condition~\eqref{eq:imit_symmetry}. Then for every point $\stratx\in[0,1]$ we have $\widetilde{h}(\stratx)=-h(1-\stratx)$ and consequently
\begin{equation*}
		\left(\varphi^{-1}\circ \updmap\circ \varphi\right)(\stratx)=\stratx\big(1-\delta(1-\stratx)h(1-\stratx)\big)=\stratx\big(1+\delta(1-\stratx)\widetilde{h}(\stratx)\big)=\widetilde{\updmap}(\stratx),
\end{equation*}
which completes the proof.
\end{proof}

In the proof of Theorem \ref{thm:anticoord-chaos} for the imitative case we will also assume that $\updmap \in \mathcal{B}_L$, that is, the update map is piecewise linear and bimodal. 
We will show the proof of Theorem \ref{thm:anticoord-chaos} in three steps.

\vspace{0.2cm}

\noindent {\bf Step 1. Construction of a family of bimodal maps that satisfy the assumptions of Proposition ~\ref{lem:chaos_cond}.}

\vspace{0.2cm}

Let $\bimap\colon[0,1]\to[0,1]$ be a bimodal map with local maximum at $\stratcl$ and local minimum at $\stratcr$, where $0<\stratcl<\stratcr<1$ (i.e., the map $\bimap$ is increasing on $(0,\stratcl)$ and $(\stratcr,1)$). One can observe the following consequence of Proposition \ref{lem:chaos_cond}.
%that in this case a result analogous to Corollary~\ref{prop:cond123} holds.

\begin{lemma}\label{cor:lin_bimodal2}
	Let $0<\stratcl<\stratcr<1$ and $\bimap \in \mathcal{B}_L$ be a map of the form %~\eqref{eq:g_form} defined by the parameter
    \begin{equation}\label{eq:g_form}
	\bimap(x)=\begin{cases}
		\frac{x}{\stratcl}, &\text{for}\ x\in[0,\stratcl),\\
		1-\beta_2(x-\stratcl), &\text{for}\ x\in [\stratcl,\stratcr),\\
		1-\beta_2(\stratcr-\stratcl)\left(1-\frac{x-\stratcr}{1-\stratcr} \right), &\text{for}\ x\in [\stratcr,1],
	\end{cases}
\end{equation}
    where $\beta_2\in\widetilde{B}_2:=\left(\frac{1-\stratcl}{\stratcr-\stratcl},\frac{1}{\stratcr-\stratcl} \right]$.
	Then $\bimap$ is a well-defined bimodal map with local maximum at $\stratcl$ and local minimum at $\stratcr$. Moreover, there is a point $x\in(\stratcl,\stratcr)$ such that $\bimap(x)<x<\bimap^3(x)$.
\end{lemma}

\begin{proof}
First, observe that if a continuous bimodal map $f \colon[0,1]\to[0,1]$ with critical points $\stratcl<\stratcr$, satisfies conditions
\begin{multicols}{3}
	\begin{enumerate}[$(1'')$]
		\item\label{prop:cond1prim} $f(\stratcr)<\stratcl$,
		\item\label{prop:cond2prim} $f(\stratcl)=1$,
		\item\label{prop:cond3prim} $f(1)>\stratcr$,
	\end{enumerate}
    \end{multicols}
	then, by Proposition \ref{lem:chaos_cond}, $f(x)<x<f^3(x)$ for some point $x\in(\stratcl,\stratcr)$. %Note that, this observation is a direct consequence of Proposition \ref{lem:chaos_cond}, thus we leave it without proof.

Now, let $g\colon[0,1]\to[0,1]$ be a map of the form \eqref{eq:f_form} which additionally satisfies the following conditions:
$$ 
 g(0)=1,\quad g(1)=0\quad\text{and}\quad g(\stratcr)\in(1-\stratcl,1].
$$
Using the form of the map $g$ we obtain equivalences
\begin{equation*}
	\begin{aligned}
		g(0)=1&\iff \beta_1=-\frac{1}{\stratcl},\\
		g(1)=0&\iff \beta_3=-\frac{\beta_2(\stratcr-\stratcl)}{1-\stratcr},\\
		g(\stratcr)\in(1-\stratcl,1]&\iff\beta_2\in\widetilde{B}_2:=\left(\frac{1-\stratcl}{\stratcr-\stratcl},\frac{1}{\stratcr-\stratcl} \right].
	\end{aligned}
\end{equation*}
In particular, the form of the map $g$ reduces to
\begin{equation}\label{eq:f_form2}
	g(x)=\begin{cases}
		-\frac{x}{\stratcl}+1, &\text{for}\ x\in[0,\stratcl),\\
		\beta_2(x-\stratcl), &\text{for}\ x\in [\stratcl,\stratcr),\\
		\beta_2(\stratcr-\stratcl)\left(1-\frac{x-\stratcr}{1-\stratcr} \right), &\text{for}\ x\in [\stratcr,1],
	\end{cases}
\end{equation}
so it depends only on the parameters $0<\stratcl<\stratcr<1$ and $\beta_2\in\widetilde{B}_2$.

Next, observe that for the map $\bimap$ in \eqref{eq:g_form} and the map $g$ in \eqref{eq:f_form2} we have $\bimap(x)=1-g(x)$
\begin{comment}
let us consider a map $f\colon[0,1]\to[0,1]$ defined as
\begin{equation}\label{eq:g_form}
	f(x):=1-g(x)=\begin{cases}
		\frac{x}{\stratcl}, &\text{for}\ x\in[0,\stratcl),\\
		1-\beta_2(x-\stratcl), &\text{for}\ x\in [\stratcl,\stratcr),\\
		1-\beta_2(\stratcr-\stratcl)\left(1-\frac{x-\stratcr}{1-\stratcr} \right), &\text{for}\ x\in [\stratcr,1],
	\end{cases}
\end{equation}
\end{comment}
(cf. Figure~\ref{fig:f_and_g}). Thus, $\bimap$ is a piecewise linear bimodal map with local maximum at $\stratcl$ and local minimum at $\stratcr$ such that
$$
\bimap(0)=1-g(0)=0,\quad \bimap(1)=1-g(1)=1,\quad \bimap(\stratcl)=1-g(\stratcl)=1
$$
and 
$$
\bimap(\stratcr)=1-g(\stratcr)\in[0,\stratcl).
$$
Therefore the map $\bimap$ satisfies conditions \ref{prop:cond1prim}--\ref{prop:cond3prim} from the beginning of the proof.
\end{proof}

\begin{figure}[!h]
	\centering
	\begin{tikzpicture}[scale=5.3]
		\draw[thick] (0,0) -- (1,0);
		\draw[thick] (0,0) -- (0,1);
		\draw (0,1) node[left]{1};
		\draw (1,0) node[below]{1};
		\draw (-0.02,0) node[below]{0}; 
		\draw[thick,domain=0:1,variable=\y] plot ({1},{\y});
		\draw[thick,domain=0:1,variable=\x] plot ({\x},{1});
		%y=x
		\draw[domain=0:1,smooth,variable=\x,black,dashed] plot ({\x},{\x});
		%(0,\stratcl)
		\draw[domain=0:0.2,smooth,variable=\x,blue,very thick] plot ({\x},{-\x/0.2+1});
		%(\stratcl,\stratcr)
		\draw[domain=0.2:0.6,smooth,variable=\x,blue,very thick] plot ({\x},{2.3*(\x-0.2)});
		%(\stratcr,1)
		\draw[domain=0.6:1,smooth,variable=\x,blue,very thick] plot ({\x},{2.3*(0.6-0.2)*(1-(\x-0.6)/(1-0.6))});
		%points
		\filldraw [red] (0.2,0) circle (0.2pt) node[anchor=north] {$\stratcl$};
		\filldraw [red] (0.6,0) circle (0.2pt) node[anchor=north] {$\stratcr$};
	\end{tikzpicture}
	\qquad
	\begin{tikzpicture}[scale=5.3]
		\draw[thick] (0,0) -- (1,0);
		\draw[thick] (0,0) -- (0,1);
		\draw (0,1) node[left]{1};
		\draw (1,0) node[below]{1};
		\draw (-0.02,0) node[below]{0}; 
		\draw[thick,domain=0:1,variable=\y] plot ({1},{\y});
		\draw[thick,domain=0:1,variable=\x] plot ({\x},{1});
		%y=x
		\draw[domain=0:1,smooth,variable=\x,black,dashed] plot ({\x},{\x});
		%(0,\stratcl)
		\draw[domain=0:0.2,smooth,variable=\x,blue,very thick] plot ({\x},{\x/0.2});
		%(\stratcl,\stratcr)
		\draw[domain=0.2:0.6,smooth,variable=\x,blue,very thick] plot ({\x},{1-2.3*(\x-0.2)});
		%(\stratcr,1)
		\draw[domain=0.6:1,smooth,variable=\x,blue,very thick] plot ({\x},{1-2.3*(0.6-0.2)*(1-(\x-0.6)/(1-0.6))});
		%points
		\filldraw [red] (0.2,0) circle (0.2pt) node[anchor=north] {$\stratcl$};
		\filldraw [red] (0.6,0) circle (0.2pt) node[anchor=north] {$\stratcr$};
	\end{tikzpicture}
	\caption{Piecewise linear bimodal maps $g$ of the form~\eqref{eq:f_form2} (on the left) and $f$ of the form~\eqref{eq:g_form} (on the right) for the parameters $\stratcl=0{.}2$, $\stratcr=0{.}6$ and $\beta_2=2{.}3$.}
	\label{fig:f_and_g}
\end{figure}

\begin{comment}
\begin{corollary}\label{prop:cond456}
	If the map $f\colon[0,1]\to[0,1]$ satisfies the following conditions
	\begin{enumerate}[$(1')$]
		\item\label{prop:cond1prim} $f(\stratcr)<\stratcl$,
		\item\label{prop:cond2prim} $f(\stratcl)=1$,
		\item\label{prop:cond3prim} $f(1)>\stratcr$,
	\end{enumerate}
	then $f(x)<x<f^3(x)$ for some point $x\in(\stratcl,\stratcr)$.
\end{corollary}

The result of Corollary \ref{prop:cond456} follows directly from Proposition \ref{lem:chaos_cond}, thus we leave it without proof.
\end{comment}

\vspace{0.2cm}

\noindent{\bf Step 2. Proof of Theorem \ref{thm:anticoord-chaos} for $p\in(0,\frac{1}{2})$.}

\vspace{0.2cm}

We will again consider a $2\times 2$ anti-coordination game $\mathcal{G}\equiv\mathcal{G}(\mathcal{A},v)$ defined by \eqref{eq:game} with a Nash equilibrium $p\in(0,\frac{1}{2})$. We reduce the number of parameters of the model by denoting
\begin{equation}\label{imit:clcrp}
\stratcl:=\frac{p}{2}\quad\text{and}\quad \stratcr:=p+\frac{p^2}{2}.
\end{equation}
%As a result, we obtain the inequalities $0<\stratcl<p<\stratcr<1$.

\begin{proposition}\label{prop:f_imit_chaotic}
Let $p \in (0,\frac 12)$ and let $\stratcl$ and $\stratcr$ be defined above. Define the switch rates:
$$
\begin{aligned}
\switch_{BA}(x)&=x\iswitch_{BA}(x):=\begin{cases}
	\frac{(2-p)x}{p(1-x)}, &\text{for}\ x\in[0,\frac{p}{2}),\\
	\frac{(x-p)(p-2)}{p(1-x)}, &\text{for}\ x\in[\frac{p}{2},p),\\
	0, &\text{for}\ x\in[p,1],
\end{cases}\\
\switch_{AB}(x)&=(1-x)\iswitch_{AB}(x):=
\begin{cases}
	0, &\text{for}\ x\in[0,p),\\
	\frac{(x-p)(2-p)}{px}, &\text{for}\ x\in[p,p+\frac{p^2}{2}),\\
	\frac{p(p-2)(1-x)}{(p^2+2p-2)x}, &\text{for}\ x\in[p+\frac{p^2}{2},1].\\
\end{cases}
\end{aligned}
$$
Then the switch rates $\switch_{BA}$ and $\switch_{AB}$ are Lipschitz continuous on $[0,1]$.
The map $\updmap$ from \eqref{eq:game_dynamics_imit} has the following form
\begin{equation}\label{eq:map_g_form}
	\updmap(x)=\begin{cases}
		\frac{2}{p}\, x, &\text{for}\ x\in[0,\frac{p}{2}),\\
		-\frac{2(1-p)}{p}\,x +2-p, &\text{for}\ x\in [\frac{p}{2},p+\frac{p^2}{2}),\\
		\frac{2(1-p^2)}{p^2+2p-2}(1-x)+1 , &\text{for}\ x\in [b+\frac{p^2}{2},1].
	\end{cases}
\end{equation}
The map $\updmap$ is piecewise linear, the dynamical system induced by $\updmap$ is Li-Yorke chaotic and has exactly one fixed point $p$ in the open interval $(0,1)$ which is repelling.
\end{proposition}

%The proof of Proposition \ref{prop:f_imit_chaotic} is based on the construction of the map $f$ for innovative revision protocols in Lemma \ref{cor:lin_bimodal1}.

\begin{proof}%[Proof of Proposition \ref{prop:f_imit_chaotic}]

Let $\updmap$ be a map of the form \eqref{eq:g_form}. From the proof of Lemma \ref{cor:lin_bimodal2} we know that
 $\updmap(\stratcl)=1$ and $\updmap(\stratcr)<\stratcl<\stratcr$. Observe that these conditions guarantee that in the interval $(\stratcl,\stratcr)$ the map $\updmap$ has a unique fixed point
$$
x^*:=\frac{1+\beta_2\stratcl}{1+\beta_2}.
$$
By \eqref{imit:clcrp}
%$$
%\stratcl=\frac{p}{2}\quad\text{and}\quad \stratcr=p+\frac{p^2}{2},
%$$
we obtain the inequalities $0<\stratcl<p<\stratcr<1$, because $p\in(0,\frac{1}{2})$. Moreover, note that
$$
x^*=p\iff (1+\beta_2)p=1+\beta_2\stratcl\iff\beta_2=\frac{1-p}{p-\stratcl}\iff\beta_2=\frac{2(1-p)}{p}
$$
and
$$
\frac{2(1-p)}{p}\in\left(\frac{1-\stratcl}{\stratcr-\stratcl},\frac{1}{\stratcr-\stratcl} \right]=\left(\frac{2-p}{p(1+p)},\frac{2}{p(1+p)} \right].
$$
Thus the map $\updmap$ from \eqref{eq:g_form} is given by
\begin{equation*}
	\begin{aligned}
	\updmap(x)&=\begin{cases}
		\frac{x}{\stratcl}, &\text{for}\ x\in[0,\stratcl),\\
		1-\frac{2(1-p)}{p}(x-\stratcl), &\text{for}\ x\in [\stratcl,\stratcr),\\
		1-\frac{2(1-p)}{p}(\stratcr-\stratcl)\left(1-\frac{x-\stratcr}{1-\stratcr} \right), &\text{for}\ x\in [\stratcr,1],
	\end{cases}\\
	&=\begin{cases}
		\frac{2}{p}\, x, &\text{for}\ x\in[0,\frac{p}{2}),\\
		-\frac{2(1-p)}{p}\,x +2-p, &\text{for}\ x\in [\frac{p}{2},p+\frac{p^2}{2}),\\
		\frac{2(1-p^2)}{p^2+2p-2}(1-x)+1 , &\text{for}\ x\in [p+\frac{p^2}{2},1].
	\end{cases}
	\end{aligned}
\end{equation*}
The map $\updmap$ satisfies the assumptions of Lemma \ref{cor:lin_bimodal2}, thus it is Li-Yorke chaotic. From its form we deduce that $\updmap$ is piecewise linear and that $p$ is the only fixed point in the open interval $(0,1)$.  Moreover, for every point $x\in(\stratcl,\stratcr)$ we have
$$
\updmap'(x)=-\frac{2(1-p)}{p}<-1,
$$
because $0<p<\frac{1}{2}<\frac{2}{3}$. Therefore the fixed point $p$ is repelling.

Our next goal is to find a map $h(x)$ for which the map $\updmap$ can be presented in the form \eqref{eq:game_dynamics_imit}
where $\delta>0$. Let us put $\delta:=1$ and observe that for every point $x\in(0,1)$ simple calculations give
$$
h(x)=\frac{\updmap(x)}{x(1-x)}-\frac{1}{(1-x)}=\frac{\updmap(x)-x}{x(1-x)}.%=\iswitch_{BA}(x)-\iswitch_{AB}(x).
$$
Furthermore we define
$$
\begin{aligned}
	h(0)&:=\lim\limits_{x\to0^+}h(x)=\lim\limits_{x\to0^+}\frac{\frac{2x}{p}-x}{ x(1-x)}=\frac{2-p}{p},\\
	h(1)&:=\lim\limits_{x\to1^-}h(x)=\lim\limits_{x\to1^-}\frac{\frac{2(1-p^2)}{p^2+2p-2}(1-x)+1-x}{x(1-x)}\\
	&=\lim\limits_{x\to1^-}\frac{\left(1-x\right)\left(\frac{2(1-p^2)}{p^2+2p-2}+1 \right)}{x(1-x)}=\frac{p(2-p)}{p^2+2p-2}.
\end{aligned}
$$
Next we will find the conditional imitation rates $\iswitch_{AB}$ and $\iswitch_{BA}$ for which 
\begin{equation}\label{eq:h_r}
	h(x)=\iswitch_{BA}(x)-\iswitch_{AB}(x),\quad x\in[0,1].
\end{equation}
Let
$$
\begin{aligned}
\iswitch_{BA}(x)&:=\begin{cases}
	h(x), &\text{for}\ x\in[0,p),\\
	0, &\text{for}\ x\in[p,1],\\
\end{cases}
=
\begin{cases}
	\frac{2-p}{p}, &\text{for}\ x=0,\\
	\frac{\updmap(x)-x}{x(1-x)}, &\text{for}\ x\in(0,p),\\
	0, &\text{for}\ x\in[p,1],\\
\end{cases}\\
\iswitch_{AB}(x)&:=\begin{cases}
	0, &\text{for}\ x\in[0,p),\\
	-h(x), &\text{for}\ x\in[p,1],\\
\end{cases}
=
\begin{cases}
	0, &\text{for}\ x\in[0,p),\\
	\frac{x-\updmap(x)}{ x(1-x)}, &\text{for}\ x\in[p,1),\\
	\frac{p(p-2)}{p^2+2p-2}, &\text{for}\ x=1.\\
\end{cases}
\end{aligned}
$$
It follows directly from the definition of $\iswitch_{BA}$ and $\iswitch_{AB}$ that the condition~\eqref{eq:h_r} is satisfied. Furthermore, the following equivalences hold (cf.~\eqref{eq:payoff_vector_equiv})
\begin{equation}\label{eq:cond_imit_rates_monotone}
u_A(x)\geq u_B(x) \iff x\leq p \iff \iswitch_{BA}(x)\geq \iswitch_{AB}(x),
\end{equation}
so the maps $\iswitch_{BA}$ and $\iswitch_{AB}$ also satisfy the condition~\eqref{eq:net_cond_imit_rates}. Using the form~\eqref{eq:map_g_form} of the map $\updmap$ we get
$$
\begin{aligned}
\iswitch_{BA}(x)&=\begin{cases}
	\frac{2-p}{p(1-x)}, &\text{for}\ x\in[0,\frac{p}{2}),\\
	\frac{(x-p)(p-2)}{px(1-x)}, &\text{for}\ x\in[\frac{p}{2},p),\\
	0, &\text{for}\ x\in[p,1],
\end{cases}\\
\iswitch_{AB}(x)&=
\begin{cases}
	0, &\text{for}\ x\in[0,p),\\
	\frac{(x-p)(2-p)}{px(1-x)}, &\text{for}\ x\in[p,p+\frac{p^2}{2}),\\
	\frac{p(p-2)}{(p^2+2p-2)x}, &\text{for}\ x\in[p+\frac{p^2}{2},1].\\
\end{cases}
\end{aligned}
$$
Observe that maps $\iswitch_{BA}$ and $\iswitch_{AB}$ are continuous on $[0,1]$ and have bounded derivative on each interval $(0,\frac{p}{2})$, $(\frac{p}{2},p)$, $(p,1)$ and $(0,p)$, $(p,p+\frac{p^2}{2})$, $(p+\frac{p^2}{2},1)$, respectively. Hence they are Lipschitz continuous on $[0,1]$. Thus the maps $\iswitch_{BA}$, $\iswitch_{AB}$ are indeed conditional imitation rates and the corresponding switch rates
$$
\begin{aligned}
\switch_{BA}(x)&:=x\iswitch_{BA}(x)=\begin{cases}
	\frac{(2-p)x}{p(1-x)}, &\text{for}\ x\in[0,\frac{p}{2}),\\
	\frac{(x-p)(p-2)}{p(1-x)}, &\text{for}\ x\in[\frac{p}{2},p),\\
	0, &\text{for}\ x\in[p,1],
\end{cases}\\
\switch_{AB}(x)&:=(1-x)\iswitch_{AB}(x)=
\begin{cases}
	0, &\text{for}\ x\in[0,p),\\
	\frac{(x-p)(2-p)}{px}, &\text{for}\ x\in[p,p+\frac{p^2}{2}),\\
	\frac{p(p-2)(1-x)}{(p^2+2p-2)x}, &\text{for}\ x\in[p+\frac{p^2}{2},1].\\
\end{cases}
\end{aligned}
$$
form an imitative revision protocol.
\end{proof}

\vspace{0.2cm}

\noindent{\bf Step 3. Proof of Theorem~\ref{thm:anticoord-chaos} for $\widetilde{p}\in(\frac{1}{2},1)$.}

\vspace{0.2cm}

Now, we consider game $\widetilde{\mathcal{G}}\equiv\widetilde{\mathcal{G}}(\mathcal{A},\widetilde{\payv})$ with the unique (symmetric) Nash equilibrium $\widetilde{\eq}\in(\frac{1}{2},1)$. 
Then, similarly as in the innovative case, there exists a population game $\game \equiv \gamefull$ with the unique Nash equilibrium $\eq=1-\widetilde{\eq}\in (0,\frac 12)$ and the conditional imitation rates $\iswitch_{AB}$, $\iswitch_{BA}$ from Proposition \ref{prop:f_imit_chaotic}. Thus, to complete the proof of Theorem \ref{thm:anticoord-chaos} for the imitative case it is sufficient to show the following fact.

\begin{proposition}\label{prop:chaos_imit_p>}
Let $\widetilde{\mathcal{G}}\equiv\widetilde{\mathcal{G}}(\mathcal{A},\widetilde{v})$ be a $2\times 2$ anti-coordination game defined by \eqref{eq:game} with a Nash equilibrium $\widetilde{\eq}\in(\frac{1}{2},1)$. Define
\begin{equation}\label{def_tilda_switch_imit}
    \widetilde{\iswitch}_{BA}(x):=\iswitch_{AB}(1-x),\quad\text{and}\qquad\widetilde{\iswitch}_{AB}(x):=\iswitch_{BA}(1-x),
\end{equation}
%$\widetilde{\switch}_{AB}(\stratx):=\switch_{BA}(1-\stratx)$ and $\widetilde{\switch}_{BA}(\stratx):=\switch_{AB}(1-\stratx)$,
for every point $\stratx\in[0,1]$, where $\iswitch_{AB}$ and $\iswitch_{BA}$ are given in Proposition \ref{prop:f_imit_chaotic}. Then the map
\[
\widetilde{\updmap}(\stratx)=\stratx\big(1+\step(1-\stratx)\widetilde{h}(\stratx)\big),
\]
with $\widetilde{h}(\stratx)=\widetilde{\iswitch}_{BA}(\stratx)-\widetilde{\iswitch}_{AB}(\stratx)$
is a piecewise linear bimodal map with local minimum at $\stratclalt:=1-\stratcr$ and local maximum at $\stratcralt:=1-\stratcl$, and the dynamical system induced by $\widetilde{\updmap}$ is Li-Yorke chaotic. Moreover, $\widetilde{\eq}$ is the unique fixed point of $\widetilde{\updmap}$ and it is repelling.
\end{proposition}

\begin{proof}
\begin{comment}

For a game $\widetilde{\mathcal{G}}\equiv\widetilde{\mathcal{G}}(\mathcal{A},\widetilde{v})$ with the unique Nash equilibrium $\widetilde{p}\in(\frac{1}{2},1)$ we define the maps
$$
\widetilde{\iswitch}_{BA}(x):=\iswitch_{AB}(1-x),\quad\widetilde{\iswitch}_{AB}(x):=\iswitch_{BA}(1-x),
$$
where $\iswitch_{AB}$ and $\iswitch_{BA}$ are the conditional imitation rates constructed above for the game $\game$ with Nash equilibrium $p=1-\widetilde{p}$.
\end{comment}
Note that the maps $\widetilde{\iswitch}_{AB}$ and $\widetilde{\iswitch}_{BA}$ are Lipschitz continuous. Moreover, it follows from \eqref{eq:payoff_vector_equiv} and \eqref{eq:cond_imit_rates_monotone} that
$$
\widetilde{\pay}_{1}(\stratx)\geq\widetilde{\pay}_{2}(\stratx)\iff \pay_{2}(1-\stratx)\geq\pay_{1}(1-\stratx)
\iff \iswitch_{BA}(1-x)\leq \iswitch_{AB}(1-x)\iff \widetilde{\iswitch}_{BA}(x)\geq \widetilde{\iswitch}_{AB}(x),
$$ 
so $\widetilde{\iswitch}_{AB}$ and $\widetilde{\iswitch}_{BA}$ are well-defined conditional imitation rates. Then the update map $\widetilde{\updmap}$ is given by
\begin{equation}\label{app:ftilde}
 \widetilde{\updmap}(x)=x\big(1+\delta(1-x)\widetilde{h}(x)\big)=x\big(1+\delta(1-x)(\widetilde{\iswitch}_{BA}(x)-\widetilde{\iswitch}_{AB}(x))\big).
\end{equation}

Next, observe that for every point $x\in[0,1]$ the relation~\eqref{eq:imit_symmetry} is satisfied. Thus, by  Lemma~\ref{lem:symmetry}, the map $\widetilde{\updmap}$ in \eqref{app:ftilde} is topologically conjugate to the map $\updmap$ (given by \eqref{eq:map_g_form}) by the homeomorphism $\varphi(x)=1-x$, i.e.
$$
\widetilde{\updmap}(x)=\left(\varphi^{-1}\circ \updmap\circ \varphi\right)(x).
$$ 
As a consequence $\widetilde{\updmap}$ is a Li-Yorke chaotic piecewise linear bimodal map with local maximum at $\stratclalt:=\varphi(\stratcr)$ and local minimum at $\stratcralt:=\varphi(\stratcl)$ (see Figure~\ref{fig:g_and_gtilde}). Moreover, $\widetilde{p}$ is a repelling fixed point of $\widetilde{\updmap}$.
\end{proof}

Since the case $p=\frac{1}{2}$ was shown in \cite{falniowski2024discrete},
%covered in Proposition~\ref{thm:imitation12},
the proof of the Theorem~\ref{thm:anticoord-chaos} for the imitative case is completed.

%%%%%%%%%%%%%%%%%%%%%%%%%%%%%%%%%%%%%%%%%%%%%%%%%%%%%%%%%%%%%%%%%%%%%%%%%%%%%%%%

%\section{Proofs}\label{app:proofs}

\begin{comment}
\begin{proof}[Proof of Lemma \ref{lemma:Self_map}]
Observe that a map $f$ of the form \eqref{PLBmap} is a continuous self-map of the interval $[0,1]$ if and only if it satisfies the following conditions:
\begin{equation}\label{eq:f_conditions}
	\begin{aligned}
		f(0)&=\alpha_1\in[0,1],\\
		f(\stratcl)&=\beta_1\stratcl+\alpha_1=\beta_2\stratcl+\alpha_2\in[0,1],\\
		f(\stratcr)&=\beta_2\stratcr+\alpha_2=\beta_3\stratcr+\alpha_3\in[0,1],\\
		f(1)&=\beta_3+\alpha_3\in[0,1].
	\end{aligned}
\end{equation}
Hence
$$
\alpha_2=-\beta_2\stratcl+\beta_1\stratcl+\alpha_1
$$
and
$$
\alpha_3=-\beta_3\stratcr+\beta_2\stratcr+\alpha_2=-\beta_3\stratcr+\beta_2\stratcr-\beta_2\stratcl+\beta_1\stratcl+\alpha_1,
$$
which allows us to eliminate parameters $\alpha_2$ and $\alpha_3$. So in fact the map $f$ is of the form
\begin{equation*}
	f(x)=\begin{cases}
		\beta_1x+\alpha_1, &\text{for}\ x\in[0,\stratcl),\\
		\beta_2(x-\stratcl)+\beta_1\stratcl+\alpha_1, &\text{for}\ x\in [\stratcl,\stratcr),\\
		\beta_3(x-\stratcr)+\beta_2(\stratcr-\stratcl)+\beta_1\stratcl+\alpha_1, &\text{for}\ x\in [\stratcr,1].
	\end{cases}
\end{equation*}
\end{proof}
\end{comment}

%{\color{red} to też włączmy do głównego tekstu}

\section{Proofs from Section~\ref{sec:thm2}}\label{app:proof-thm2}

\begin{proof}[Proof of Proposition \ref{thm:imitation12}]
Let us consider game dynamics induced by the conditional imitation rates $\iswitch_{AB}$ and $\iswitch_{BA}$, that is
$$
\updmap(x)=x\big(1+\step(1-x)h(x)\big),
$$
where $h(x)=\iswitch_{BA}(x)-\iswitch_{AB}(x)$. From the condition \eqref{eq:net_cond_imit_rates} and the form of map $\updmap$, we have
$$
\begin{aligned}
x\in[0,p)& \implies h(x)>0 \implies \updmap(x)>x,\\
x\in(p,1]&\implies h(x)<0\implies \updmap(x)<x.
\end{aligned}
$$
Moreover, by \eqref{eq:iswitchimit} we have $h(1-x)=-h(x)$ and consequently
\begin{equation}\label{eq:f_imit_symmetry}
\begin{aligned}
1-\updmap(1-x)&=1-(1-x)\big(1+\step xh(1-x)\big)=1-(1-x)\big(1-\step xh(x)\big)\\
&=x+\step xh(x)-\step x^2h(x)=x\big(1+\step(1-x)h(x)\big)=\updmap(x).
\end{aligned}
\end{equation}
Since the map $\updmap$ is continuous on the interval $[0,1]$, there exists a point $z_l^\step\in[0,p]$ such that $\updmap(x)\leq \updmap(z_l^\step)$ for any $x\in[0,p]$. Let us denote $z_r^\step:=1-z_l^\step$. Then, it follows from \eqref{eq:f_imit_symmetry} that $\updmap(x)\geq \updmap(z_r^\step)$ for any point $x\in[p,1]$, and that $\updmap(z_r^\step)=0\iff \updmap(z_l^\step)=1$. 
But $\updmap(z_l^\step)=1$ when $\step=\frac{1}{\stratcl^{\step} (1-h(\stratcl^\step))}$.
%Moreover, by \eqref{eq:f_imit_symmetry}, $\updmap(z_l^{\step})=1-\updmap(z_r^{\step})$, so %for critical points we have
Thus, for $\step^*:=\frac{1}{\stratcl^{\step} (1-h(\stratcl^\step))}$ we have $\updmap([0,1])=[0,1]$ and the points $z_l^{\step^*}$, $z_r^{\step^*}$ satisfy the assumptions of Proposition \ref{lem:chaos_cond}. By   \citet{li1982odd} the map $\updmap$ has periodic orbit of period 3, and by \cite{liyorke} the map $\updmap$ is Li-Yorke chaotic and has a periodic point of any period.
\end{proof}

\begin{lemma}\label{prop:f_max}
Let $\updmap_{max}$ be the map from \eqref{eq:PPI_map_pert} defined by the parameters $(\eta,\xi,\step) = \left( \frac{4p}{(1-p)^2(b-d)},\frac{4}{p(b-d)},1 \right)$.
	For any interior Nash equilibrium $p\in(0,\frac{1}{2}]$ the map $\updmap_{max}$ has local maximum at $\stratcl=\frac{p}{2}$, local minimum at $\stratcr=\frac{p+1}{2}$ and the image of $\updmap_{max}$ is equal to $[0,1]$. Moreover
	\begin{enumerate}[$(1)$]
		\item\label{prop:f_max_1} if $p\in[\frac{1}{5},\frac{1}{2}]$, then $\updmap_{max}$ is a bimodal map;
		\item\label{prop:f_max_2} if $p\in(0,\frac{1}{5})$, then $\updmap_{max}$ has an additional local maximum at $x_p:=\frac{p+1}{6}$.
	\end{enumerate}
\end{lemma}

\begin{proof}
At first observe that for $x\in(0,p)$ the map $\updmap_{max}$ takes the form
$$
\updmap_{max}(x)%=x\left(1+\frac{4}{p^2}(1-x)(p-x)\right)
=x+\frac{4}{p^2}\left(x^3-px^2-x^2+px \right), 
$$
so
$$
\updmap'_{max}(x)=1+\frac{4}{p^2}\left(3x^2-2px-2x+p \right).
$$
Hence
$$
\begin{aligned}
	\updmap'_{max}(x)<0&\iff 12x^2-8(p+1)x+p(p+4)<0\\
	%&\iff\frac{8(p+1)-\sqrt{16(2-p)^2}}{24}<x<\frac{8(p+1)+\sqrt{16(2-p)^2}}{24}\\
	&\iff \frac{p}{2}<x<\frac{4+p}{6}.
\end{aligned}
$$
Since $p\in(0,\frac{1}{2}]$, the following inequalities hold 
$$
0<\frac{p}{2}<p<\frac{4+p}{6}<1.
$$
Therefore, the map $\updmap_{max}$ is increasing on $(0,\stratcl)$, decreasing on $(\stratcl,p)$ and has a local maximum at $\stratcl$. Similarly, for $x\in(p,1)$ the map $\updmap_{max}$ is given by
$$
\updmap_{max}(x)%=x\left(1-\frac{4}{(1-p)^2}(1-x)(x-p)\right)
=x+\frac{4}{(1-p)^2}\left(x^3-px^2-x^2+px \right),
$$
thus
$$
\updmap'_{max}(x)=1+\frac{4}{(1-p)^2}\left(3x^2-2px-2x+p \right).
$$
Next, note that
$$
\begin{aligned}
	\updmap'_{max}(x)<0&\iff 12x^2-8(p+1)x+(p+1)^2<0\\
	%&\iff\frac{8(p+1)-\sqrt{16(p+1)^2}}{24}<x<\frac{8(p+1)+\sqrt{16(p+1)^2}}{24}\\
	&\iff\frac{p+1}{6}<x<\frac{p+1}{2}\\
	&\iff x_p<x<\stratcr.
\end{aligned}
$$
Let us consider two cases
\begin{itemize}
	\item if $p\in[\frac{1}{5},\frac{1}{2}]$, then $x_p\leq p<\stratcr$ and the map $\updmap_{max}$ is decreasing on $(p,\stratcr)$. Therefore, by the continuity, it is decreasing on the whole interval $(\stratcl,\stratcr)$.
	\item if $p\in(0,\frac{1}{5})$, then $p<x_p<\stratcr$, the map $\updmap_{max}$ is increasing on $(p,x_p)$ and decreasing on $(x_p,\stratcr)$. In particular, $x_p$ is a local maximum of $\updmap_{max}$.
\end{itemize}
Observe that in both cases $\updmap_{max}$ is increasing on $(\stratcr,1)$ and has a local minimum at $\stratcr$. 

In order to prove the equality $\updmap_{max}([0,1])=[0,1]$, recall that $\updmap_{max}(0)=0$, $\updmap_{max}(1)=1$ and 
$$
\begin{aligned}
x\in[0,p]&\implies \updmap_{max}(x)\geq x,\\
x\in[p,1]&\implies \updmap_{max}(x)\leq x.
\end{aligned}
$$
Hence, it follows from the above considerations that $\updmap_{max}([0,1])=[0,1]$ if and only if $\updmap_{max}(\stratcl)\leq1$ and $\updmap_{max}(\stratcr)\geq0$. We obtain
$$
\updmap_{max}(\stratcl)=\updmap_{max}\left(\frac{p}{2}\right)=\frac{p}{2}\left(1+\frac{4}{p^2}\left(1-\frac{p}{2}\right) \left( p-\frac{p}{2}\right) \right)=1
$$
and
$$
\updmap_{max}(\stratcr)=\updmap_{max}\left(\frac{p+1}{2}\right)=\frac{p+1}{2}\left(1-\frac{4}{(1-p)^2}\left( 1-\frac{p+1}{2}\right) \left( \frac{p+1}{2}-p\right) \right)=0,
$$
which completes the proof.
\end{proof}

\begin{proof}[Proof of Lemma \ref{cor:PPI_interval_map}]
Let $\updmap_{max}$ be the map from \eqref{eq:PPI_map_pert} defined by the parameters $(\eta,\xi,\step):=\left( \frac{4p}{(1-p)^2(b-d)},\frac{4}{p(b-d)},1 \right)$.
First, from Lemma \ref{prop:f_max} we have that $\updmap_{max}([0,1])=[0,1]$.
Second, it follows directly from~\eqref{eq:PPI_map_pert} that for any $(\eta,\xi,\step)\in\Delta_p$ the map $\updmap$ is continuous on the interval $[0,1]$ and differentiable for all points $x\in(0,1)$ except (at most) the equilibrium $\eq$. Moreover, $\updmap(0)=0$, $\updmap(1)=1$, $\updmap(p)=p$ and the following inequalities hold: %{\color{red}(cf. Figure ...)}:
\begin{equation}\label{eq:PPI_f_max_ineq}
\begin{aligned}
	&x\leq \updmap(x)\leq \updmap_{max}(x),\quad\text{for}\quad x\in[0,p]\\
	&x\geq \updmap(x)\geq \updmap_{max}(x),\quad\text{for}\quad x\in[p,1].
\end{aligned}
\end{equation}
The equality $\updmap_{max}([0,1])=[0,1]$ together with the inequalities~\eqref{eq:PPI_f_max_ineq} compete the proof.
\end{proof}

\begin{proof}[Proof of Proposition \ref{Thm2:step2}]
%For $\eq =\frac 12$ by taking $\varepsilon_1=\varepsilon_2$ we get pairwise proportional imitation protocol. For this protocol result chaotic behavior is already known \cite{falniowski2024discrete}. Thus, we can focus on $\eq\neq \frac 12$.  
    Let $\game$ be a population game defined by~\eqref{eq:game} with the interior Nash equilibrium $p\in(0,\frac{1}{2})$ and denote
    $$
    \step_1:=\frac{1}{p+1},\quad\step_2:=\frac{1}{2-p},\quad\step_3:=\frac{1}{8}\left(p+3+ \sqrt{\frac{-p^3-4p^2-13p+34}{2-p}}\right).
    $$
    Since $p\in(0,\frac{1}{2})$, simple calculations give $\step_2\leq\step_1$ and
    \begin{equation}\label{step_p}
    \step_p:=\max\left\lbrace\step_1,\step_2,\step_3\right\rbrace=\max\left\lbrace\step_1,\step_3  \right\rbrace<1.
    \end{equation}
    Let $\updmap$ be a map of the form~\eqref{eq:PPI_map_pert} that describes the game dynamics of $\game$ introduced by the perturbed PPI protocol~\eqref{eq:PPI_cond_imit_rates_pert}. Consider the points $\stratcl=\frac{p}{2}$ and $\stratcr=\frac{p+1}{2}$ from Lemma~\ref{prop:f_max} (i.e. the local maximum and minimum of the map $\updmap_{max}$, respectively). Recall that
 	$$
 	0<\stratcl<p<\stratcr<1.
 	$$
    We will show that for the parameters $(\eta,\xi) = \left( \frac{4p}{(1-p)^2(b-d)},\frac{4}{p(b-d)} \right)$ and time step size $\step\in(\step_p,1]$ there exists a point $x\in(\stratcl,\stratcr)$ such that $\updmap(x)<x<\updmap^3(x)$. First, observe that it is sufficient to show that the points $\stratcl$ and $\stratcr$ satisfy assumptions of Proposition~\ref{lem:chaos_cond}. We will show that
	 $\updmap(\stratcr)<\stratcl$, $\updmap(\stratcl)>\stratcr$,
		and $\updmap^2(\stratcl)>\stratcr$.
	Thus, the rest of the proof is devoted to showing that these inequalities  hold. It follows from $\step>\step_1$ and $\step>\step_2$ that
	$$
	\begin{aligned}
		\updmap(\stratcr)&=\frac{p+1}{2}\left(1-\frac{4\step}{(1-p)^2}\left( 1-\frac{p+1}{2}\right) \left( \frac{p+1}{2}-p\right) \right)=\frac{p+1}{2}\left(1-\step\right)\\
		&<\frac{p+1}{2}\left(1-\step_1\right)=\frac{p+1}{2}\left(1-\frac{1}{p+1}\right)=\frac{p}{2}=\stratcl,\\
		\updmap(\stratcl)&=\frac{p}{2}\left(1+\frac{4\step}{p^2}\left(1-\frac{p}{2}\right) \left( p-\frac{p}{2}\right) \right) =\frac{p}{2}+\step\left(1-\frac{p}{2} \right)\\
		&>\frac{p}{2}+\step_2\left(1-\frac{p}{2} \right)=\frac{p}{2}+\frac{1}{2-p}\left(1-\frac{p}{2} \right)=\frac{p+1}{2}=\stratcr.
	\end{aligned}
	$$
	Next, note that for $x\in(p,1)$ we have
	$$
	\updmap'(x)=1+\frac{4\step}{(1-p)^2}\left(3x^2-2px-2x+p \right), 
	$$
	so
	$$
	K:=\sup_{x\in(p,1)}\updmap'(x)=1+\frac{4\step}{(1-p)^2}\sup_{x\in(p,1)}\left(3x^2-2px-2x+p \right)=1+\frac{4\step}{1-p}.
	$$
	Since $\updmap(\stratcl)\in(\stratcr,1]\subset(p,1]$ and $\updmap(1)=1$, we obtain
	$$
	\left|1-\updmap^2(\stratcl)\right|=\left|\updmap(1)-\updmap^2(\stratcl)\right|\leq K\left|1-\updmap(\stratcl) \right|.
	$$
	Furthermore, the following equivalences hold
	$$
	\begin{aligned}
		K\left|1-\updmap(\stratcl) \right|<\left|1-\stratcr \right| &\iff\left(1+\frac{4\step}{1-p} \right)\left(1-\frac{p}{2}-\step\left(1-\frac{p}{2} \right)\right)< 1-\frac{p+1}{2}\\
		%&\iff\left(1+\frac{4\step}{1-p} \right)\left(1-\frac{p}{2} \right)\left(1-\step\right)< \frac{1-p}{2}\\
		&\iff\frac{1-p+4\step}{1-p}\cdot\frac{2-p}{2}\left(1-\step\right)< \frac{1-p}{2}\\
		%&\iff\left(1-p+4\step\right)\left(1-\step\right)< \frac{(1-p)^2}{2-p}\\
		%&\iff-4\step^2+(p+3)\step+1-p-\frac{(1-p)^2}{2-p}<0\\
		%&\iff-4\step^2+(p+3)\step+\left(1-p\right)\left(1-\frac{1-p}{2-p}\right) <0\\
		&\iff4\step^2-(p+3)\step-\frac{1-p}{2-p} >0.
	\end{aligned}
	$$
%	The determinant is given by
%	$$
%	\Delta_{\step}:=\frac{-p^3-4p^2-13p+34}{2-p}.
%	$$
%	Clearly $\Delta_{\step}>0$, because $0<p<\frac{1}{2}$.
    Hence the inequality $|1-\updmap^2(\stratcl)|<|1-\stratcr|$ is satisfied for every
	$$
	\step>\frac{1}{8}\left(p+3+ \sqrt{\frac{-p^3-4p^2-13p+34}{2-p}}\right)=\step_3.
	$$
	Finally, let us note that the inequality $|1-\updmap^2(\stratcl)|<|1-\stratcr|$ implies $\stratcr<\updmap^2(\stratcl)$. Applying the results of Section \ref{sec:discrete} completes the proof.% for $\eq\in (0,\frac 12)$.
%    \cite{li1982odd} and \cite{liyorke} completes the proof for $\eq\in (0,\frac 12)$.
 %\end{proof}

\begin{comment}
  It remains to show our theorems for $\eq>\frac 12$. This will follow from the topological conjugacy argument. For a game $\widetilde{\mathcal{G}}\equiv\widetilde{\mathcal{G}}(\mathcal{A},\widetilde{\payv})$ with the unique Nash equilibrium $\widetilde{\eq} = 1 - \eq \in(\frac{1}{2},1)$ one obtain similar results by considering conditional imitation rates $\widetilde{\iswitch}_{BA}$ and $\widetilde{\iswitch}_{AB}$ symmetric to \eqref{eq:PPI_cond_imit_rates_pert} and \eqref{eq:truncated_cond_imit_rate}. To this aim we will use the following lemma. %(see Lemma \ref{lem:symmetry}  below).
%\end{remark}

Let $\widetilde{\mathcal{G}}$ be the game \eqref{eq:game} with the interior Nash equilibrium $\widetilde{p}\in(\frac{1}{2},1)$
Define
\begin{equation}\label{PPIiswitch-tilda}
\widetilde{\iswitch}_{AB}(\stratx):=\iswitch_{BA}(1-\stratx) \quad\text{and}\quad \widetilde{\iswitch}_{BA}(\stratx):=\iswitch_{AB}(1-\stratx),
\end{equation}
for every point $\stratx\in[0,1]$, where $\iswitch_{AB}$ and $\iswitch_{BA}$ are given in \eqref{eq:PPI_cond_imit_rates_pert} or when  $\iswitch_{AB}$ is given in \eqref{eq:truncated_cond_imit_rate}.
Observe that by \eqref{PPIiswitch-tilda} the condition \eqref{eq:imit_symmetry} is satisfied. Application of Lemma \ref{lem:symmetry} completes the proof. %and Corollary \ref{col:symmetric_properties}.
\end{comment}
\end{proof}

\begin{proof}[Proof of Proposition \ref{Thm2:step3}]
Let $\widetilde{\mathcal{G}}\equiv\widetilde{\mathcal{G}}(\mathcal{A},\widetilde{\payv})$ be a game defined by \eqref{eq:game} with $\widetilde{\eq}\in(\frac{1}{2},1)$ and denote $\eq:=1-\widetilde{\eq}\in(0,\frac{1}{2})$. Then the interior Nash equilibrium of the game $\game\equiv\gamefull$ given by the payoff matrix
\begin{equation}\label{game:conjugacy}
\begin{array}{l|cc}
	&A	&B\\
\hline
A	&\gaind		&\gainc\\
B	&\gainb	& \gaina
\end{array}  
\end{equation}
is equal to $\eq$. Moreover, for any triplet $(\eta,\xi,\step)\in (0,\infty)^3$, we have $(\eta,\xi,\step)\in\Delta_{\eq}$ if and only if $(\xi,\eta,\step)\in\Delta_{\widetilde{\eq}}$. %Let $\updmap$ be the map $\eta = \frac{4p}{(1-p)^2(b-d)}$ and $\xi = \frac{4}{p(b-d)}$ 
%For every parameters $(\xi,\eta,\step)\in\Delta_{\widetilde{p}}$ 

Let $\iswitch_{AB}$ and $\iswitch_{BA}$ be the conditional imitation rates of the game $\game$ given by~\eqref{eq:PPI_cond_imit_rates_pert}. From  Lemma~\ref{cor:PPI_interval_map} it follows that $\updmap$ in \eqref{eq:PPI_map_pert} is an interval map for any parameters $(\eta,\xi,\step)\in\Delta_{\eq}$. Furthermore, the conditional imitation rates $\widetilde{\iswitch}_{AB}$ and $\widetilde{\iswitch}_{BA}$ of the game $\widetilde{\game}$ defined by \eqref{PPIiswitch-tilda} satisfy the condition~\eqref{eq:imit_symmetry}. So, the corresponding maps $\updmap$ and $\widetilde{\updmap}$ satisfy the assumptions of Lemma~\ref{lem:symmetry}. Therefore, $\widetilde{\updmap}$ is also an interval map, and the dynamical systems $([0,1],\updmap)$ and $([0,1],\widetilde{\updmap})$ are topologically conjugate. From Proposition \ref{Thm2:step2} we conclude that there exists $\step_{\widetilde{p}} \in (0,1)$ such that for $\xi = \frac{4\widetilde{p}}{(1-\widetilde{p})^2(b-d)}$, $\eta = \frac{4}{\widetilde{p}(b-d)}$ and $\step\in(\step_{\widetilde{p}},1]$ the map $\widetilde{\updmap}$ is Li-Yorke chaotic and has periodic orbit of any period.
\end{proof}

\begin{proof}[Proof of Lemma \ref{prop:dynamics_PPI}]
For $p\in(0,\frac{1}{2})$ we denote
\[
 \step_4:=\frac{p}{2(1-p)} \qquad \text{ and } \qquad \step_5:=\frac{1-p}{2p}.
\]
	Let us compute the one-sided derivatives of the map $\updmap$ at the point $p$:
		$$
	\begin{aligned}
         \updmap'_-(p)&=\lim\limits_{x\to p^-}\frac{\updmap(x)-\updmap(p)}{x-p}=\lim\limits_{x\to p^-}\frac{x\left(1+\frac{4\step}{p^2}(1-x)(p-x)\right)-p}{x-p}\\
		&=\lim\limits_{x\to p^-}\left(1-\frac{4\step}{p^2}\,x(1-x)\right)=1-\frac{4\step}{p}(1-p),\\
		\updmap'_+(p)&=\lim\limits_{x\to p^+}\frac{\updmap(x)-\updmap(p)}{x-p}=\lim\limits_{x\to p^+}\frac{x\left(1-\frac{4\step}{(1-p)^2}(1-x)(x-p)\right)-p}{x-p}\\
		&=\lim\limits_{x\to p^+}\left(1-\frac{4\step}{(1-p)^2}\,x(1-x)\right)=1-\frac{4\step p}{1-p}.
	\end{aligned}
	$$
        Note that 
        $$
        \step_4 = \frac{p}{2(1-p)}<1\iff p<\frac{2}{3}
        $$
        and for $\step>\step_4$, we obtain
        $$
        \updmap'_-(p)=1-\frac{4\step}{p}(1-p)<1-\frac{4\step_4}{p}(1-p)=1-\frac{4}{p}\cdot\frac{p(1-p)}{2(1-p)}=-1.
        $$
	Next, observe that
	$$
	\updmap'_+(p)<-1\iff1-\frac{4\step p}{1-p}<-1\iff2(1-p)<4\step p\iff\step>\frac{1-p}{2p}
	$$
	and
	$$
	\step_5 = \frac{1-p}{2p}<1\iff p>\frac{1}{3}.
	$$
	Hence, for $p\in(\frac{1}{3},\frac{1}{2})$ the interval $(\step_5,1]$ is well-defined and $\step>\step_5$ implies $\updmap'_+(p)<-1$. On the other hand, for $p\in(0,\frac{1}{3}]$ we get the inequalities
	$$
	0<\frac{4\step p}{1-p}\leq\frac{4p}{1-p}\leq2.
	$$
	So $|\updmap'_+(p)|\leq1$ for every parameter $\step\in(0,1]$. Thus, the proof is complete.
\end{proof}

\begin{proof}[Proof of Corollary \ref{cor:perturbed_PPI_repelling}]
Let $\game$ be a population game defined by~\eqref{eq:game} with the interior Nash equilibrium $p\in(0,\frac{1}{2})$ and let $\updmap$ be the map from \eqref{eq:PPI_map_pert} with $(\eta,\xi) = \left( \frac{4p}{(1-p)^2(b-d)},\frac{4}{p(b-d)} \right)$.
Recall that
$$
\step_1=\frac{1}{p+1},\quad\step_2=\frac{1}{2-p},\quad\step_3=\frac{1}{8}\left(p+3+ \sqrt{\frac{-p^3-4p^2-13p+34}{2-p}}\right)
$$
and
$$
\step_4=\frac{p}{2(1-p)},\quad\step_5=\frac{1-p}{2p}.
$$
Observe that for any point $p\in(\frac{1}{3},\frac{1}{2})$ we have $\max\{\step_1,\step_2,\step_3,\step_4,\step_5\}=\max\{\step_3,\step_5\}$. Therefore, it follows from the proof of Theorem~\ref{thm:perturbedPPIchaos}, that for any $\step>\max\{\step_3,\step_5\}$ the map $\updmap$ is Li-Yorke chaotic and has a periodic point of every period. From Lemma~\ref{prop:dynamics_PPI} we conclude that the point $p$ is repelling. By the conjugacy argument we complete the proof for $\eq\in (\frac 13,\frac 23)$.
\end{proof}

\section{Proofs of results from Section~\ref{sec:thm3}}\label{app:proof-thm3}

\begin{lemma}\label{prop:f_max_truncated}
Assume that $\eq\in(0,\frac12)$ and let $\updmap^{*}_{max}$ be the map of form \eqref{eq:PPI_truncated_map2} defined by the parameters $\gamma=\eq+\frac{\eq^2}{2}$ and $(\eta,\xi,\step)=\left(\frac{4}{\eq(2-2\eq-\eq^2)(\gainb-\gaind)},\frac{4}{\eq(\gainb-\gaind)},1\right)$. Then the image of the map $\updmap^{*}_{max}$ is equal to $[0,1]$. Moreover, $\updmap^{*}_{max}\colon[0,1]\to[0,1]$ is a bimodal map with the local maximum at $\stratcl=\frac{\eq}{2}$ and the local minimum at $\gamma$.
\end{lemma}

\begin{proof}%[Proof of Lemma \ref{prop:f_max_truncated}]
	
	Since for any $x\in[0,p]$ we have $\updmap^{*}_{max}(x)=\updmap_{max}(x)$, it follows from the proof of Lemma~\ref{prop:f_max} that the map $\updmap^{*}_{max}$ is increasing on $(0,\stratcl)$, decreasing on $(\stratcl,p)$ and has a local maximum at $\stratcl$ with $\updmap^{*}_{max}(\stratcl)=1$. Therefore it suffices to investigate the behavior of $\updmap^{*}_{max}$ to the right from the fixed point $p$. 
	
	Firstly, let $x\in(p,\gamma)$ and observe that
	$$
	\begin{aligned}
	\updmap^{*}_{max}(x)%&=x\left(1-\frac{4}{p^2(2-2p-p^2)}(1-x)(x-p)\right)\\
	&=x+\frac{4}{p^2(2-2p-p^2)}\left(x^3-px^2-x^2+px \right).
	\end{aligned}
	$$
	Thus,
	$$
	(\updmap^{*}_{max})'(x)=1+\frac{4}{p^2(2-2p-p^2)}\left(3x^2-2px-2x+p \right),
	$$
	which leads to the following equivalence
	$$
	(\updmap^{*}_{max})'(x)<0\iff 12x^2-8(p+1)x-p(p^3+2p^2-2p-4)<0.
	$$
%	Hence, the determinant of the LHS of the inequality is given by 
%	\[
%	\Delta_x=16(3p^4+6p^3-2p^2-4p+4)>0,
%    \]
%    and 
    The roots of the LHS of the last inequality are given by
	\[x_0=\frac{2p+2-\sqrt{3p^4+6p^3-2p^2-4p+4}}{6}, \;\; x_1=\frac{2p+2+\sqrt{3p^4+6p^3-2p^2-4p+4}}{6}.
%	\end{aligned}
	\]
	So, we conclude that $(\updmap^{*}_{max})'(x)<0$ if and only if $x_0 <x<x_1$. Note that $p\in(0,\frac{1}{2})$ implies the inclusion $(p,\gamma)\subset(x_0,x_1)$. So $(\updmap^{*}_{max})'(x)<0$ for every point $x\in(p,\gamma)$. Since the map $\updmap^{*}_{max}$ is decreasing on the intervals $(\stratcl,p)$ and $(p,\gamma)$, then by continuity it is decreasing on the interval $(\stratcl,\gamma)$.
	
	Next, let $x\in(\gamma,1)$. In this case the map $\updmap^{*}_{max}$ takes the form
	$$
	\updmap^{*}_{max}(x)=x\left(1-\frac{4}{p^2(2-2p-p^2)}\cdot\frac{p^2}{2}(1-x)\right)=x+\frac{2}{2-2p-p^2}(x^2-x).
	$$
	So
	$$
	(\updmap^{*}_{max})'(x)=1+\frac{2}{2-2p-p^2}(2x-1)
	$$
	and consequently
	$$
	(\updmap^{*}_{max})'(x)>0\iff4x-2p-p^2>0\iff x>\frac{p}{2}+\frac{p^2}{4}=\frac{\gamma}{2}.
	$$
	Thus the map $\updmap^{*}_{max}$ is increasing on the interval $(\gamma,1)$ and has a local minimum at $\gamma$. Finally let us note that
	$$
	\updmap^{*}_{max}(\gamma)=\gamma\left(1-\frac{2}{2-2p-p^2}(1-\gamma)\right)=\left(p+\frac{p^2}{2}\right) \left(1-\frac{2}{2-2p-p^2}\left(1-p-\frac{p^2}{2}\right)\right)=0.
	$$
	Recall that $\updmap^{*}_{max}(x)\geq x$ for $x\in[0,p]$, $\updmap_{max}(x)\leq x$ for $x\in[p,1]$ and $\updmap^{*}_{max}(0)=0$, $\updmap^{*}_{max}(1)=1$. Therefore we have shown that $\updmap^{*}_{max}([0,1])=[0,1]$ and $\updmap^{*}_{max}$ is a bimodal map with local maximum and minimum at $\stratcl$ and $\gamma$, respectively.
\end{proof}

\begin{proof}[Proof of Lemma \ref{cor:f_maps_truncated}]
%{\color{red}TODO}

%Clearly, for any parameters $(\eta,\xi,\step)\in\Delta_p$ and point $x\in[0,\gamma]$ we have $f^*_{\eta,\xi,\step}(x)=f_{\eta,\xi,\step}(x)$. 

%We begin the analysis of the family of maps~\eqref{eq:PPI_truncated_map} in a similar way as in the case of~\eqref{eq:PPI_map_pert}, that is, by investigating the properties of map 
Let $\updmap^{*}_{max}$ be the map from Lemma~\ref{prop:f_max_truncated}. Observe that for $\gamma=\eq+\frac{\eq^2}{2}$ and any triplet $(\eta,\xi,\step)\in\Delta_p^*$ the map $\updmap^{*}$ of the form \eqref{eq:PPI_truncated_map2} satisfy the inequalities analogous to~\eqref{eq:PPI_f_max_ineq}, i.e.
\begin{equation}\label{eq:PPI_f_max_trun_ineq}
	\begin{aligned}
		&x\leq \updmap^{*}(x)\leq \updmap^{*}_{max}(x),\quad\text{for}\quad x\in[0,p]\\
		&x\geq \updmap^{*}(x)\geq \updmap^{*}_{max}(x),\quad\text{for}\quad x\in[p,1].
	\end{aligned}
\end{equation}
As a simple consequence of Lemma~\ref{prop:f_max_truncated} and inequalities~\eqref{eq:PPI_f_max_trun_ineq} we get $\updmap^{*}([0,1])=[0,1]$.
\end{proof}

\begin{proof}[Proof of Proposition \ref{Thm3:step2}]
Since the proof of Proposition~\ref{Thm3:step2} proceeds similarly to the proof of Proposition~\ref{Thm2:step2} and Lemma \ref{prop:dynamics_PPI}, we will skip all calculations and present only its outline.

Let $\game$ be a population game defined by~\eqref{eq:game} with the interior Nash equilibrium $p\in(0,\frac{1}{2})$. We denote
$$
\step^*_1:=\frac{p+1}{p+2},\quad\step^*_2:=\frac{p(p+1)}{2-p},\quad\step^*_3:=\frac{1}{4}\left(p^2+2p+ \sqrt{\frac{p\left(p^4+10p^3+20p^2-8p-16\right) }{p-2}}\right)
$$
and
$$
\step^*_4:=\step_4=\frac{p}{2(1-p)},\quad \step^*_5:=\frac{p(2-2p-p^2)}{2(1-p)}.
$$
Since $p\in(0,\frac{1}{2})$, we obtain the inequalities $\step^*_i<\step^*_1$ for $i=2,4,5$. Therefore
$$
\step^*_p:=\max\left\lbrace\step^*_1,\step^*_2,\step^*_3,\step^*_4,\step^*_5  \right\rbrace=\max\left\lbrace\step^*_1,\step^*_3  \right\rbrace.
$$
Note that $\step^*_p<1$.

Let $\updmap^{*}$ be a map of the form~\eqref{eq:PPI_truncated_map2} defined by the parameters $\gamma=\eq+\frac{\eq^2}{2}$, $\eta = \frac{4}{\eq(2-2\eq-\eq^2)(\gainb-\gaind)}$, $\xi = \frac{4}{\eq(\gainb-\gaind)}$ and $\step\in(\step^*_p,1]$. Consider the points $\stratcl=\frac{p}{2}$ and $\gamma$ (i.e. the local maximum and minimum of the map $\updmap^{*}_{max}$ from Lemma~\ref{prop:f_max_truncated}). Clearly, we have
 	$$
 	0<\stratcl<p<\gamma<1.
 	$$
Similarly as in the proof of Proposition~\ref{Thm2:step2}, the condition $\step>\max\{\step^*_1,\step^*_2\}$ ensures that the inequalities $\updmap^{*}(\gamma)<\stratcl$ and $\updmap^{*}(\stratcl)>\gamma$ hold. Furthermore, for any $\step>\step^*_3$ we obtain
$$
    \left|1-(\updmap^{*})^2(\stratcl)\right|=\left|\updmap^{*}(1)-(\updmap^{*})^2(\stratcl)\right|\leq L\left|1-\updmap^{*}(\stratcl) \right|<\left|1-\gamma \right|,
$$
where
$$
    L:=\sup_{x\in(\gamma,1)}(\updmap^{*})'(x)=1+\frac{2\step}{2-2p-p^2}.
$$
Hence $(\updmap^{*})^2(\stratcl)>\gamma$. So the points $\stratcl$ and $\gamma$ satisfy assumptions of Proposition~\ref{lem:chaos_cond}. Consequently, the map $\updmap^{*}$ is Li-Yorke chaotic and has a periodic point of any period $n\in\mathbb{N}$ (see Section \ref{sec:discrete}).

    Finally, let us observe that the one-sided derivatives at the point $p$ satisfy the inequalities $(\updmap^{*})'_-(p)<-1$ and $(\updmap^{*})'_+(p)<-1$ for any $\step>\max\{\step^*_4,\step^*_5\}$ (cf. the proof of Lemma \ref{prop:dynamics_PPI}), hence the point $p$ is repelling for $\updmap^{*}$. %The proof is for $\eq<\frac 12$ is completed. For $\eq>\frac 12$ Lemma \ref{lem:symmetry} and standard topological conjugacy argument completes the proof.
\end{proof}

\begin{proof}[Proof of Proposition \ref{Thm3:step3}]
Let $\widetilde{\mathcal{G}}\equiv\widetilde{\mathcal{G}}(\mathcal{A},\widetilde{\payv})$ and $\game\equiv\gamefull$ be as in the proof of Proposition~\ref{Thm2:step3}. Denote $\gamma:=\eq+\frac{\eq^2}{2}$. Then for $\widetilde{\gamma}:=1-\gamma$ and any triplet $(\eta,\xi,\step)\in(0,\infty)^3$, the condition $(\eta,\xi,\step)\in\Delta^{*}_p$ is equivalent to $(\xi,\eta,\step)\in\Gamma^{*}_{\widetilde{p}}$.

Let $\iswitch^*_{AB}$, $\iswitch^*_{BA}$ and $\widetilde{\iswitch}^{\,*}_{AB}$, $\widetilde{\iswitch}^{\,*}_{BA}$ be the conditional imitation rates of the games $\game$ and $\widetilde{\game}$ given by~\eqref{eq:truncated_cond_imit_rate2} and \eqref{eq:thm3_conjugacy}, respectively. Then, application of Lemma~\ref{cor:f_maps_truncated}, Lemma~\ref{lem:symmetry} and Proposition~\ref{Thm3:step2} completes the proof.
\end{proof}
\end{document}